%% file: main.tex
\documentclass[11pt]{amsart}
\usepackage[utf8]{inputenc}
\usepackage{amssymb,amsmath,amsthm,amsfonts}
\usepackage{tikz}
\usepackage{fullpage}
\usepackage[all]{xy}
\usepackage{dsfont}
\usepackage{blkarray, bigstrut}
\usepackage{multicol}

\usepackage{hyperref}
\hypersetup{
    colorlinks=true,
    linkcolor=blue,
    citecolor=magenta,
    filecolor=magenta,      
    urlcolor=magenta,
}
\usepackage{float}

\newtheorem{theorem}{Theorem}[section]

\newtheorem{conjecture}[theorem]{Conjecture}

\newtheorem{lemma}[theorem]{Lemma}
\newtheorem{corollary}[theorem]{Corollary}

\theoremstyle{definition}
\newtheorem{remark}[theorem]{Remark}
\newtheorem{definition}[theorem]{Definition}

\newtheorem{example}[theorem]{Example}

\makeatletter
\newtheorem*{rep@theorem}{\rep@title}\newcommand{\newreptheorem}[2]{%
\newenvironment{rep#1}[1]{%
\def\rep@title{\bf #2 \ref{##1}}%
\begin{rep@theorem}}%
{\end{rep@theorem}}}

\newcommand{\addresseshere}{%
  \enddoc@text\let\enddoc@text\relax
}

\makeatother
\newreptheorem{theorem}{Theorem}

\renewcommand\emptyset{\varnothing}

\newcommand{\R}{\mathbb{R}}

\newcommand{\calC}{\mathcal{C}}
\newcommand{\calG}{\mathcal{G}}
\newcommand{\calO}{\mathcal{O}}
\newcommand{\cO}{\mathcal{O}}
\newcommand{\cT}{\mathcal{T}}
\newcommand{\frakT}{\mathfrak{T}}

\newcommand{\calV}{\mathcal{V}}
\newcommand{\be}{\mathbf{e}}
\newcommand{\bp}{\mathbf{p}}
\newcommand{\bu}{\mathbf{u}}
\newcommand{\bv}{\mathbf{v}}
\newcommand{\bw}{\mathbf{w}}
\newcommand{\bx}{\mathbf{x}}

\newcommand\hatP{\widehat{P}}
\newcommand\Irr{\mathrm{Irr}}
\newcommand\sgn{\mathrm{sgn}}
\newcommand\Sq{\mathrm{Sq}}
\newcommand{\mL}{\mathcal{L}}

\newcommand{\udots}{\rotatebox[origin=c]{80}{$\ddots$}}
\newcommand{\Ddots}{\rotatebox[origin=c]{170}{$\ddots$}}

\DeclareMathOperator{\vol}{vol}

\newcommand\commentout[1]{}

\definecolor{green}{RGB}{34, 139, 34}

\definecolor{darkblue}{rgb}{0.0, 0.0, 1}

\definecolor{teal}{rgb}{0.0, 0.5, 0.5}

\tikzstyle{pvx}=[circle, draw, fill=black, inner sep=1.3pt, minimum size=4pt]
\newcommand{\pvx}{\node[pvx]}
\tikzstyle{gvx}=[circle, draw=blue, fill=blue!25, inner sep=1.3pt, minimum size=4pt]
\newcommand{\gvx}{\node[gvx]}

\begin{document}

\title{Triangulations, Order Polytopes, and Generalized Snake Posets}

\author{Matias von Bell}
\address{Department of Mathematics\\
         University of Kentucky\\
\url{http://www.ms.uky.edu/~mkvo222/}}
\email{matias.vonbell@uky.edu}

\author{Benjamin Braun}
\address{Department of Mathematics\\
         University of Kentucky\\
\url{https://sites.google.com/view/braunmath/}}
\email{benjamin.braun@uky.edu}

\author{Derek Hanely}
\address{Department of Mathematics\\
         University of Kentucky\\
\url{https://math.as.uky.edu/users/dwha232}}
\email{derek.hanely@uky.edu}

\author{Khrystyna Serhiyenko}
\address{Department of Mathematics\\
         University of Kentucky\\
\url{https://math.as.uky.edu/users/kse246}}
\email{khrystyna.serhiyenko@uky.edu}

\author{Julianne Vega}
\address{Department of Mathematics\\
         Kennesaw State University\\
\url{https://sites.google.com/view/julievega/}}
\email{jvega30@kennesaw.edu}

\author{Andr\'es R. Vindas-Mel\'endez}
\address{Department of Mathematics\\
         University of Kentucky\\
\url{https://ms.uky.edu/\~arvi222}}
\email{andres.vindas@uky.edu}

\author{Martha Yip}
\address{Department of Mathematics\\
         University of Kentucky\\
\url{http://www.ms.uky.edu/~myip/}}
\email{martha.yip@uky.edu}

\date{\today}

\begin{abstract}
This work regards the order polytopes arising from the class of generalized snake posets and their posets of meet-irreducible elements.
Among generalized snake posets of the same rank, we characterize those whose order polytopes have minimal and maximal volume.
We give a combinatorial characterization of the circuits in these order polytopes and then conclude that every triangulation is unimodular.
For a generalized snake word, we count the number of flips for the canonical triangulation of these order polytopes. 
We determine that the flip graph of the order polytope of the poset whose lattice of filters comes from a ladder is the Cayley graph of a symmetric group.
Lastly, we introduce an operation on triangulations called twists and prove that twists preserve regular triangulations.
\end{abstract}

\thanks{The authors thank Raman Sanyal for helpful discussions about triangulations of order polytopes, and Paco Santos for helpful comments on a preliminary version of this article.
MB was partially supported by a University of Kentucky Mathematics Steckler Fellowship.
BB was partially supported by National Science Foundation award DMS-1953785.
DH was partially supported by National Science Foundation award DUE-1356253.
KS was partially supported by funding from the University of Kentucky College of Arts and Sciences.
ARVM was partially supported by National Science Foundation Graduate Research Fellowship DGE-1247392 and National Science Foundation
KY-WV LSAMP Bridge to Doctorate Fellowship HRD-2004710.
MY was partially supported
by Simons Collaboration Grant 429920.}

\maketitle


\section{Introduction}
In 1986, Richard Stanley~\cite{StanleyTwoPosetPolytopes} introduced two geometric objects associated to a finite partially ordered set, or \emph{poset}, known as the order polytope and the chain polytope. 
Since then, the study of order polytopes has been an active area of research, e.g., geometric and algebraic properties~\cite{DoignonRexhep, HaaseKohlTsuchiya, HibiLiLiMuTsuchiya, HibiMatsuda, HibiMatsudaKazunoriOhsugiHidefumiShibata}, connections between flow polytopes and order polytopes~\cite{LiuMeszarosStDizier, MeszarosMoralesStriker}, and lattice-point enumeration~\cite{ChappellFriedlSanyal, LiuTsuchiya}.

One of Stanley's fundamental observations is that the arrangement given by all hyperplanes of the form $x_i=x_j$ for $i\neq j$ induces a regular unimodular triangulation of the order polytope for any poset.
This triangulation is known as the \emph{canonical triangulation} of an order polytope, see Subsection~\ref{sec:orderpolys}.
It is well-known that the set of all regular triangulations of a polytope  correspond to the vertices of its secondary polytope, and that these triangulations are connected via flips; definitions are given in Subsection~\ref{sec:triangulations} and further information can be found in~\cite{DeLoeraRambauSantos}.
Various triangulations of order polytopes have been constructed or considered, often for special classes of posets. See, for example, Santos, Stump, and Welker for products of chains~\cite{SantosStumpWelker}, F\'eray and Reiner for non-unimodular triangulations related to graph-associahedra~\cite{FerayReiner}, Reiner and Welker for graded posets~\cite{ReinerWelker}, Br\"anden and Solus for $\mathbf{s}$-lecture hall order polytopes~\cite{BrandenSolusLectureHall}, disjoint unions of chains~\cite[Section 6.2]{DeLoeraRambauSantos}, and others.
However, the general space of regular triangulations of an order polytope, i.e., the $1$-skeleton of the secondary polytope of an order polytope, does not appear to have been studied in detail and motivates our work.



Our contributions in this paper add to the literature on order polytopes and further the study of the general space of regular triangulations of order polytopes.
Specifically, we investigate circuits, flips, and regular triangulations of order polytopes arising from a certain class of posets, called generalized snake posets.  These posets are constructed recursively by adding a square face at the bottom and gluing it to an edge of the previous square. First, we prove results regarding the volumes of their corresponding order polytopes.  In particular, for generalized snake posets of the same rank, Theorem~\ref{thm:minmaxvolumes} characterizes those with minimal and maximal normalized volumes.  

Next, we turn our attention to the poset $Q$, whose lattice of filters is a generalized snake poset $P$, and study the combinatorial properties of the corresponding order polytope $\mathcal{O}(Q)$.  Thus, the vertices of $\mathcal{O}(Q)$ are given by the elements of $P$.
In Theorem \ref{thm.bijection}, we prove that there exists a bijection between the set of nonempty connected induced subgraphs associated to the faces of $P$ and the set of circuits of $\mathcal{O}(Q)$.   As a consequence, we obtain Theorem \ref{thm:unimodular}, which states that every vertex of the secondary polytope of $\mathcal{O}(Q)$ corresponds to a unimodular triangulation.
The combinatorial characterization of the circuits also implies that the canonical triangulation of $\mathcal{O}(Q)$ admits the same number of flips as there are faces in $P$, see Theorem \ref{thm:flipsfromcanonical}.
Then, in Theorem~\ref{thm.Cayleygraph} we determine that the flip graph of the order polytope $\mathcal{O}(Q)$, when $P$ is a ladder, is in fact the Cayley graph of a symmetric group.

Finally, we introduce an action on the vertices of $P$ given by the so-called twists.  It extends to an action on regular triangulations, and in Theorem \ref{thm:regularity} we prove that twists of a canonical triangulation of $\mathcal{O}(Q)$ are again regular triangulations.  Moreover, twists preserve circuits of $\mathcal{O}(Q)$, and hence they commute with applying flips, see Theorem~\ref{thm:square}.
In particular, this implies that twists give an action on the component of the flip graph of $\mathcal{O}(Q)$ containing all regular triangulations.

The article is organized as follows.  In Section~\ref{sec:background}, we review some background and establish notation for triangulations and order polytopes.   In Section~\ref{sec:generalizedsnakes}, we introduce the family of generalized snake posets $P$ and study volumes of their corresponding order polytopes.  The characterization of circuits of the order polytope $\mathcal{O}(Q)$ of the poset of filters of $P$ is given in Section~\ref{sec:circuits}.  
Section \ref{sec:flips_twists} is devoted to introducing twists, and then proving four theorems regarding twists, flips, and triangulations of $\mathcal{O}(Q)$.   Lastly, in Section \ref{sec:future} we conclude the paper with conjectures for future work.

\section{Background and Notation}\label{sec:background}
\subsection{Triangulations}\label{sec:triangulations}

Our primary focus in this paper is the study of triangulations for a particular family of order polytopes. 
We begin by providing the necessary background for triangulations following the presentation in De Loera, Rambau, and Santos~\cite[Section 2.4]{DeLoeraRambauSantos}.


\begin{definition}
Given a point configuration $\mathbf{A} \subseteq \mathbb{R}^d$, let $\mathrm{conv}(\mathbf{A})$ denote the convex hull of $\mathbf{A}$.
A \emph{triangulation} of $\mathbf{A}$ is a collection $\mathcal{T}$ of $d$-simplices all of whose vertices are points in $\mathbf{A}$ that satisfies the following two properties: 
\begin{enumerate}
    \item The union of all of these simplices equals $\operatorname{conv}(\mathbf{A})$. (\emph{Union Property}) 
    \item Any pair of these simplices intersects in a (possibly empty) common face. (\emph{Intersection property})
\end{enumerate}
A triangulation is \emph{unimodular} if every simplex has normalized volume one.
A triangulation  of a point configuration $\mathbf{A} \subseteq \mathbb{R}^d$ is \emph{regular} if it can be obtained by projecting the lower envelope of a lifting of $\mathbf{A}$ from $\mathbb{R}^{d+1}.$ 
\end{definition}

\begin{example}
\label{ex:triangulation}
Consider the polytope generated by the convex hull of the points $(0,0,0,0),$ $(1,0,0,0),$ $(1,1,0,0), (1,0,1,0), (1,1,1,0),$ and $(1,1,1,1)$. A triangulation of this point configuration consists of the simplices: 
\begin{align*}
\sigma_{4231} &= \mathrm{conv}\{(0,0,0,0), (1,0,0,0), (1,0,1,0), (1,1,1,0), (1,1,1,1) \}\\
\sigma_{4321} &= \mathrm{conv}\{(0,0,0,0), (1,0,0,0), (1,1,0,0), (1,1,1,0), (1,1,1,1) \}.    
\end{align*}
\end{example}

\begin{definition}
A point configuration $\mathbf{A}$ with index set $J$ has \emph{corank one} if and only if it has an affine dependence relation $\sum_{j \in J} \lambda_j\mathbf{v}_j = 0$ with $\sum_{j \in J} \lambda_j = 0$ that is unique up to multiplication by a constant. 
This affine dependence partitions $J$ into three subsets: 
\[ J_+ := \{j \in J: \lambda_j >0\}, J_0 := \{j \in J: \lambda_j = 0\}, \text{ and } J_- := \{j \in J: \lambda_j <0\}.\]

In the case when $A$ has corank one, $J_+$ and $J_-$ are the only disjoint subsets of $J$ with the property that their relative interiors intersect at the point 
\[ \sum\limits_{j \in J_+} \lambda_j \mathbf{v}_j = \sum\limits_{j \in J_-} |\lambda_j| \mathbf{v}_j,\]
where the $\lambda_j$ are assumed to be normalized so that $
\sum_{j\in J_+}\lambda_j= \sum_{j\in J_-}|\lambda_j|=1$.
The set $J_+ \bigcup J_-$ is called a \emph{circuit} in $J$ and the pair $(J_+, J_-)$ is called the {\em oriented circuit}, or {\em Radon partition} , of $\mathbf{A}$.  
\end{definition}

\begin{definition}
\label{def:circuit} 
Let $\mathbf{A}$ be a point configuration with index set $J$.
In general, a subset $Z$ of $J$ is a \emph{circuit} if it is a minimal dependent set (that is, it is dependent but every proper subset is independent).
Let $(Z_+,Z_-)$ be a partition of $Z$, 
such that $\mathrm{conv}(Z_+) \cap \mathrm{conv}(Z_-)$ is nonempty. 
The partition $(Z_+, Z_-)$ is called an \emph{oriented circuit}.
We say the circuit is of \emph{type} $(|Z_+|,|Z_-|)$.
\end{definition}



From the circuits we can generate triangulations by using flips to locally transform one triangulation into another.

\begin{lemma}{\cite[Lemma 2.4.2]{DeLoeraRambauSantos} } 
Let $\mathbf{A}$ be a point configuration of corank one and $J = J_+ \cup J_0 \cup J_-$ be its label set, partitioned by the unique oriented circuit of $\mathbf{A}$. 
Then the following are the only two triangulations of $\mathbf{A}:$ 

\[ \mathcal{T}_+ = \{J \smallsetminus \{j\}: j \in J_+\}, \text{ and } \mathcal{T}_-= \{J \smallsetminus \{j\} : j \in J_-\}.\]
\end{lemma}


\begin{example} \label{ex:triangulation_2}
We return to Example~\ref{ex:triangulation}, where we considered the convex hull of the corank one point configuration
\[
\mathbf{v}_0 = (0,0,0,0),
\mathbf{v}_1 = (1,0,0,0),
\mathbf{v}_2 = (1,1,0,0),
\mathbf{v}_3 = (1,0,1,0), 
\mathbf{v}_4 = (1,1,1,0),
\mathbf{v}_5 = (1,1,1,1).
\]
Since $\mathbf{v}_1 - \mathbf{v}_2 - \mathbf{v}_3 + \mathbf{v}_4 = \mathbf{0}$, then an oriented circuit is $J_+ = \{1,4\}$ and $J_- = \{2,3\}$.
The circuit type is $(2,2)$ and the two triangulations are:  
\begin{align*}
 \mathcal{T}_+ &= \{J \setminus\{j\} : j \in J_+\} = \{\{0,2,3,4,5\},\{0,1,2,3,5\}\}\\ 
\mathcal{T}_- &= \{J \setminus\{j\} : j \in J_-\} = \{\{0,1,3,4,5\},\{0,1,2,4,5\}\}. 
\end{align*}
\end{example}

Triangulation of $\mathbf{A}$ is a simplicial complex on $\mathbf{A}$. Recall that an \emph{(abstract) simplicial complex} $\Delta$ on a set $X$ is a collection of subsets of $X$ such that if $\sigma \in \Delta$ and $\tau \subseteq \sigma$, then $\tau \in \Delta.$ 
The elements of a simplicial complex are called \emph{faces} and a subcomplex $\Delta'$ of $\Delta$ is a subcollection of $\Delta$ which is also a simplicial complex. 
The \emph{link} of a face $\sigma \in \Delta$ is the simplicial complex
$$\mathrm{lk}_\Delta(\sigma) = \{\tau \in \Delta : \sigma \cup \tau \in \Delta \text{ and } \sigma \cap \tau = \emptyset \}.
$$
If $\Delta$ and $\Delta'$ are simplicial complexes, then their {\it join} is $\Delta\ast\Delta'=\{\sigma\cup\sigma' : \sigma\in\Delta \text{ and }  \sigma'\in\Delta'\}$.




\begin{theorem}{\cite[Theorem 4.4.1]{DeLoeraRambauSantos}}\label{thm:flip}
Let $\mathcal{T}_1$ and $\mathcal{T}_2$ be two triangulations of a point configuration $\mathbf{A}$. 
Then $\mathcal{T}_1$ and $\mathcal{T}_2$ differ by a flip if and only if there is a circuit $Z$ of $\mathbf{A}$ such that 
\begin{enumerate}
    \item[(i)] They contain, respectively, the two triangulations $\mathcal{T}_Z^+$ and $\mathcal{T}_Z^-$ of $Z$. 
    \item[(ii)] All the maximal simplices of $\mathcal{T}_Z^+$ and $\mathcal{T}_Z^-$ have the same link $\mathrm{L}$ in $\mathcal{T}_1.$
    \item[(iii)] 
    Removing the subcomplex $\mathcal{T}_Z^+ \ast \mathrm{L}$ from $\mathcal{T}_1$ and replacing it by $\mathcal{T}_Z^- \ast \mathrm{L}$ gives $\mathcal{T}_2.$ 
\end{enumerate}
\end{theorem}

Two triangulations of $\mathbf{A}$ are \emph{adjacent} if they differ by a flip.
The set of all triangulations of $\mathbf{A}$, under adjacency by flips, forms the \emph{graph of triangulations}, or \emph{flip graph}, of $\mathbf{A}$.

\begin{example}
Continuing from Example~\ref{ex:triangulation_2} we demonstrate a flip from $\mathcal{T}_+$ to $\mathcal{T}_-$. 
We have 
\[
\mathcal{T}_+ = \{\{0,2,3,4,5\}, \{0,1,2,3,5\}\} \in \mathcal{T}.
\]
The circuit $Z =\{1,2,3,4\}$ has triangulations $\mathcal{T}_Z^+ = \{\{2,3,4\}, \{1,2,3\} \}$ and $\mathcal{T}_Z^- = \{\{1,3,4\}, \{1,2,4\} \}$.
The link of the simplices is
$\mathrm{L}= \{\{\mathbf{v}_0,\mathbf{v}_5\}\}$,
so flipping at the circuit supported at $Z$ gives the triangulation
$$\mathcal{T}_Z^- \ast \mathrm{L} = \{\{0,1,3,4,5\},\{0,1,2,4,5\}\}=  \mathcal{T}_-.$$ 
\end{example}



In Sections \ref{sec:circuits} and \ref{sec:flips_twists}, we will take a look at the \emph{secondary polytope} whose vertices are in bijection with regular triangulations of a point configuration.
Recall that we can define for each triangulation of a point configuration $\mathbf{A}$ a GKZ-vector.
As stated in the following definition, the convex hull of the GKZ-vectors for $\mathbf{A}$ is the secondary polytope. 
See De Loera, Rambau, and Santos~\cite[Section 5.1]{DeLoeraRambauSantos} for a further discussion of secondary polytopes and GKZ-vectors. 

\begin{definition}[Secondary Polytope.] For a point configuration $\mathbf{A}$ the secondary polytope of $\mathbf{A}$ is conv\{$\varphi_{\mathbf{A}}(\mathcal{T})\mid \mathcal{T}\text{ triangulation of } \mathcal{A}$\}, 
where $\varphi_{\mathbf{A}}(\mathcal{T})$ represents the GKZ-vector of $\mathcal{T}$ in $\mathbf{A}$.
\end{definition}

 The flip graph, which is the graph of all triangulations connected by flips, is in general not connected, but the flip graph of regular triangulations is connected and contains the 1-skeleton of the secondary polytope as a spanning subgraph. 

\subsection{Order polytopes}\label{sec:orderpolys}
Let $P$ be a partially ordered set on the set of elements $[d]:=\{1,\ldots,d\}$. 
We abuse notation and write $P$ to denote the elements of $P$.
The \emph{order polytope} of $P$, introduced by Stanley~\cite{StanleyTwoPosetPolytopes}, is defined as
\[
\mathcal{O}(P) = \left\{\bx=(x_1,\ldots, x_d)\in [0,1]^d: x_i\leq x_j \text{ for }i<_Pj\right\}.
\]
See Example~\ref{ex:order_polytope}.
An \emph{upper order ideal of $P$}, also called a \emph{filter}, is a set $A\subseteq P$ such that if $i\in A$ and $i<_Pj$, then $j\in A$.
Let $J(P)$ denote the poset of upper order ideals of $P$, ordered by reverse inclusion. We use $\langle p_1,...,p_k\rangle$ to denote the ideal generated by elements $p_1,...,p_k \in P$.
Let $\be_1,\ldots,\be_d$ denote the standard basis vectors of $\R^d$.
For an upper order ideal $A\in J(P)$, define the \emph{characteristic vector} $\bv_A:=\sum_{i\in A}\be_i$.
The vertices of $\mathcal{O}(P)$ are given by
\[
V(\mathcal{O}(P))=\left\{\bv_A:A\in J(P)\right\}.
\]
Define a hyperplane $\mathcal{H}_{i,j}= \{ \bx\in \R^d : x_i = x_j \}$ for $1\leq i < j \leq d$. 
The set of all such hyperplanes, called the $d$-dimensional braid arrangement of type A, induces a triangulation $\mathcal{T}$ of $\mathcal{O}(P)$ known as the \emph{canonical triangulation}, which has the following three fundamental properties:
\begin{enumerate}
    \item $\mathcal{T}$ is unimodular,
    \item the simplices are in bijection with the linear extensions of $P$, so the normalized volume of the order polytope is
    $$ \vol (\mathcal{O}(P)) = \# \text{ of linear extensions of } P, \text{ and} $$
    \item the simplex corresponding to a linear extension $(a_1,\ldots,a_d)$ of $P$ is
    $$ \sigma_{a_1,...,a_d} = \left\{\bx \in [0,1]^d : x_{a_1} \leq x_{a_2} \leq \cdots \leq x_{a_d}\right\}, $$
    with vertex set 
    $\{\mathbf{0}, \mathbf{e}_{a_d}, \mathbf{e}_{a_{d-1}}+ \mathbf{e}_{a_d}, \ldots ,  \mathbf{e}_{a_1}+\cdots+\mathbf{e}_{a_d} = \mathbf{1}\}.$
\end{enumerate}


\begin{example}[Order polytope and triangulations]\label{ex:order_polytope}

Let $P$ be the diamond poset 
$$
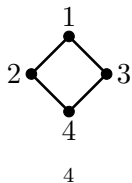
\begin{tikzpicture}
\begin{scope}[xshift=0, yshift=0, scale=0.5]
	\pvx[]  at (0,2)  {};
	\pvx[] at (1,1)  {};
	\pvx[] at (-1,1)  {};
	\pvx[] at (0,0)  {};
	
	\draw[line width = 1pt] (0,0) -- (-1,1) -- (0,2) -- (1,1) -- (0,0);

	\node[anchor = north] at (0,0)  {$4$};
	\node[anchor = east] at (-1,1)  {$2$};	
	\node[anchor = west] at (1,1)  {$3$};
	\node[anchor = south] at (0,2)  {$1$};
\end{scope}
\end{tikzpicture}
$$
Then $\mathcal{O}(P) = \{(x_1,x_2,x_3,x_4) \in [0,1]^4  : x_4\leq x_2\leq x_1 \text{ and } x_4 \leq x_3 \leq x_1\}$.
The six upper order ideals of $P$ are
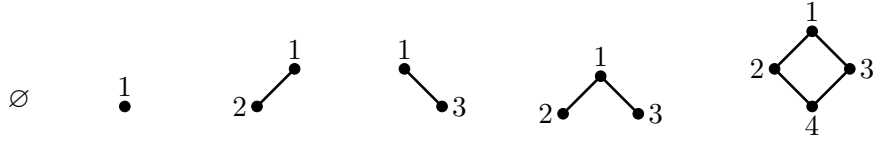
\begin{center}
\begin{tikzpicture}
\begin{scope}[xshift=-10, yshift=-5, scale=0.5]
	\node ()  at (0,0)  {};
	
	\node[anchor = south] at (0,0)  {$\emptyset$};
\end{scope}
\begin{scope}[xshift=30, yshift=0, scale=0.5]
	\pvx[] at (0,0)  {};
	
	\node[anchor = south] at (0,0)  {$1$};
\end{scope}
\begin{scope}[xshift=80, yshift=0, scale=0.5]
	\pvx[] at (1,1)  {};
	\pvx[] at (0,0)  {};
	
	\draw[line width = 1pt] (1,1) -- (0,0);

	\node[anchor = east] at (0,0)  {$2$};
	\node[anchor = south] at (1,1)  {$1$};

\end{scope}
\begin{scope}[xshift=150, yshift=0, scale=0.5]
	\pvx[] at (-1,1)  {};
	\pvx[] at (0,0)  {};
	
	\draw[line width = 1pt] (0,0) -- (-1,1);

	\node[anchor = west] at (0,0)  {$3$};
	\node[anchor = south] at (-1,1)  {$1$};	
\end{scope}
\begin{scope}[xshift=210, yshift=-17, scale=0.5]
	\pvx[] at (1,1)  {};
	\pvx[] at (-1,1)  {};
	\pvx[] at (0,2)  {};
	
	\draw[line width = 1pt] (1,1) -- (0,2) -- (-1,1);

	\node[anchor = south] at (0,2)  {$1$};
	\node[anchor = east] at (-1,1)  {$2$};	
	\node[anchor = west] at (1,1)  {$3$};
\end{scope}
\begin{scope}[xshift=290, yshift=0, scale=0.5]

	\pvx[] at (0,2)  {};
	\pvx[] at (1,1)  {};
	\pvx[] at (-1,1)  {};
	\pvx[] at (0,0)  {};
	
	\draw[line width = 1pt] (0,0) -- (-1,1) -- (0,2) -- (1,1) -- (0,0);

	\node[anchor = north] at (0,0)  {$4$};
	\node[anchor = east] at (-1,1)  {$2$};	
	\node[anchor = west] at (1,1)  {$3$};
	\node[anchor = south] at (0,2)  {$1$};
\end{scope}
\end{tikzpicture}
\end{center}
so $\mathcal{O}(P)$ is the convex hull of the points $(0,0,0,0), (1,0,0,0), (1,1,0,0), (1,0,1,0)$, $(1,1,1,0)$ and $(1,1,1,1)$. 
The poset $P$ has two linear extensions, namely $4,2,3,1$ and $4,3,2,1$.
The canonical triangulation of $\cO(P)$ then consists of the following simplices:
\begin{align*}
\sigma_{4231} &= \mathrm{conv}\{(0,0,0,0), (1,0,0,0), (1,0,1,0), (1,1,1,0), (1,1,1,1) \}\\
\sigma_{4321} &= \mathrm{conv}\{(0,0,0,0), (1,0,0,0), (1,1,0,0), (1,1,1,0), (1,1,1,1) \}.
\end{align*}
\end{example}




\section{Generalized snake posets}\label{sec:generalizedsnakes}

We introduce the family of generalized snake posets $P(\bw)$, which are distributive lattices with width two, and give a recursive formula for the normalized volume of the order polytope of $P(\bw)$. 
For generalized snake posets of the same rank, we characterize those with minimal and maximal normalized volumes.

\begin{definition}
For $n\in \mathbb{Z}_{\geq0}$, a \emph{generalized snake word} is a word of the form $\bw=w_0 w_1 \cdots w_n$ where $w_0 =\varepsilon$ is the empty letter and $w_i$ is in the alphabet $\{L,R\}$ for $i=1,\ldots, n$.
The \emph{length} of the word is $n$, which is the number of letters in $\{L,R\}$.
\end{definition}

\begin{definition}
Given a generalized snake word $\bw=w_0w_1\cdots w_n$, we define the \emph{generalized snake poset} $P(\bw)$ recursively in the following way:
\begin{itemize}
    \item $P(w_0) = P(\varepsilon)$ is the poset on elements $\{0,1,2,3\}$ with cover relations $1\prec 0$, $2\prec 0$, $3\prec 1$ and $3\prec 2$. 

    \item $P(w_0w_1\cdots w_n)$ is the poset $P(w_0w_1\cdots w_{n-1}) \cup \{2n+2,2n+3\}$ with the added cover relations $2n+3 \prec 2n+1$, $2n+3 \prec 2n+2$, and 
    $$\begin{cases}
    2n+2 \prec 2n-1, 
        & \text{ if } n=1 \text{ and } w_n = L, \text{ or } n \geq 2 \text{ and }  w_{n-1}w_n \in \{RL,LR\},\\
    2n+2 \prec 2n, 
        & \text{ if } n=1 \text{ and } w_n = R, \text{ or } n \geq 2 \text{ and }  w_{n-1}w_n \in \{LL,RR\}.
    \end{cases}$$
\end{itemize}
In this definition, the minimal element of the poset $P(\bw)$ is $\widehat0=2n+3$, and the maximal element of the poset is $\widehat1 = 0$.
\end{definition} 

If $\bw=w_0w_1\cdots w_n$ is a generalized snake word of length $n$, then $P(\bw)$ is a distributive lattice of width two and rank $n+2$.
We point out two special cases of generalized snake posets.
For the length $n$ word $\varepsilon LRLR\cdots$, $S_n:=P(\varepsilon LRLR\cdots)$ is the \emph{snake poset}, and for the length $n$ word $\varepsilon LLLL\cdots$, $\mL_n:=P(\varepsilon LLLL\cdots)$ is the \emph{ladder poset}. 
For an example, refer to Figure \ref{fig.snake_and_ladder}.

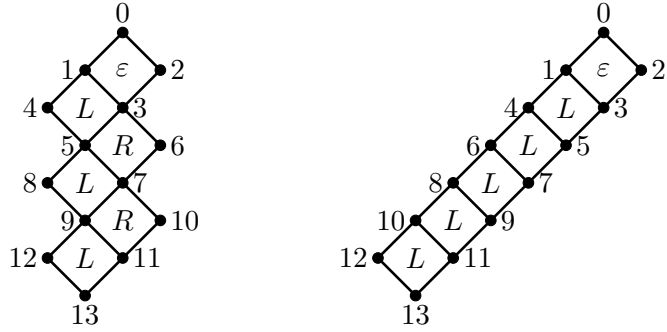
\begin{figure}[ht!]
\begin{center}
\begin{tikzpicture}[scale=.5]
\begin{scope}[xshift=0, yshift=0]
	\pvx[] at (1,7)  {};
	\pvx[] at (2,6)  {};		
	\pvx[] at (0,6)  {};
	\pvx[] at (-1,5)  {};	
	\pvx[] at (1,5)  {};
	\pvx[] at (2,4)  {};	
	\pvx[] at (-1,3)  {};
    \pvx[] at (0,4)  {};	
	\pvx[] at (1,3)  {};
	\pvx[] at (2,2)  {};
	\pvx[] at (0,2)  {};
	\pvx[] at (1,1)  {};
	\pvx[] at (-1,1)  {};
	\pvx[] at (0,0)  {};
	
	\draw[line width = 1pt] (0,0) -- (-1,1) -- (0,2) -- (1,1) -- (0,0);
	\draw[line width = 1pt] (0,2) -- (1,3) -- (2,2) -- (1,1);
	\draw[line width = 1pt] (0,2) -- (-1,3) -- (0,4) -- (1,3);
	\draw[line width = 1pt] (0,4) -- (1,5) -- (2,4) -- (1,3);
	\draw[line width = 1pt] (0,4) -- (-1,5) -- (0,6) -- (1,5);
	\draw[line width = 1pt] (0,6) -- (1,7) -- (2,6) -- (1,5);
	
	\node[anchor = south] at (1,7)  {$0$};
	\node[anchor = east] at (0,6)  {$1$};
	\node[anchor = west] at (2,6)  {$2$};
	\node[anchor = west] at (1,5)  {$3$};
	\node[anchor = east] at (-1,5)  {$4$};
	\node[anchor = east] at (0,4)  {$5$};
	\node[anchor = west] at (2,4)  {$6$};
	\node[anchor = west] at (1,3)  {$7$};
	\node[anchor = east] at (-1,3)  {$8$};
	\node[anchor = east] at (0,2)  {$9$};
	\node[anchor = west] at (2,2)  {$10$};
	\node[anchor = west] at (1,1)  {$11$};
	\node[anchor = east] at (-1,1)  {$12$};
	\node[anchor = north] at (0,0)  {$13$};
	
	\node[] at (1,6)  {$\varepsilon$};	
	\node[] at (0,5)  {$L$};
	\node[] at (1,4)  {$R$};
	\node[] at (0,3)  {$L$};
	\node[] at (1,2)  {$R$};
	\node[] at (0,1)  {$L$};
\end{scope}
\begin{scope}[xshift=250, yshift=0]
	\pvx[] at (5,7)  {};
	\pvx[] at (6,6)  {};	
	\pvx[] at (4,6)  {};
	\pvx[] at (5,5)  {};	
	\pvx[] at (3,5)  {};
	\pvx[] at (4,4)  {};
	\pvx[] at (2,4)  {};
	\pvx[] at (3,3)  {};
	\pvx[] at (1,3)  {};
	\pvx[] at (2,2)  {};
	\pvx[] at (0,2)  {};
	\pvx[] at (1,1)  {};
	\pvx[] at (-1,1)  {};
	\pvx[] at (0,0)  {};
	
	\draw[line width = 1pt] (0,0) -- (-1,1) -- (0,2) -- (1,1) -- (0,0);
	\draw[line width = 1pt] (0,2) -- (1,3) -- (2,2) -- (1,1);
	\draw[line width = 1pt] (1,3) -- (2,4) -- (3,3) -- (2,2);
    \draw[line width = 1pt] (2,4) -- (3,5) -- (4,4) -- (3,3);	
    \draw[line width = 1pt] (3,5) -- (4,6) -- (5,5) -- (4,4);
    \draw[line width = 1pt] (4,6) -- (5,7) -- (6,6) -- (5,5);	
	\node[anchor = south] at (5,7)  {$0$};
	\node[anchor = east] at (4,6)  {$1$};
	\node[anchor = east] at (3,5)  {$4$};
	\node[anchor = east] at (2,4)  {$6$};
	\node[anchor = east] at (1,3)  {$8$};
	\node[anchor = east] at (0,2)  {$10$};
	\node[anchor = east] at (-1,1)  {$12$};	
	\node[anchor = west] at (6,6)  {$2$};
	\node[anchor = west] at (5,5)  {$3$};	
	\node[anchor = west] at (4,4)  {$5$};	
	\node[anchor = west] at (3,3)  {$7$};	
	\node[anchor = west] at (2,2)  {$9$};	
	\node[anchor = west] at (1,1)  {$11$};
	\node[anchor = north] at (0,0)  {$13$};		
	\node[] at (5,6)  {$\varepsilon$};	
	\node[] at (4,5)  {$L$};	
	\node[] at (3,4)  {$L$};
	\node[] at (2,3)  {$L$};
	\node[] at (1,2)  {$L$};
	\node[] at (0,1)  {$L$};
\end{scope}
\end{tikzpicture}
\end{center}
\caption{The snake poset $S_5 = P(\varepsilon LRLRL)$ and the ladder poset $\mL_5 = P(\varepsilon LLLLL)$.}
\label{fig.snake_and_ladder}
\end{figure}

In this article, we consider the generalized snake posets in two contexts.
In the next subsection, we consider the order polytope of the generalized snake posets, $\calO(P(\bw))$. 
More precisely, we give a recursive formula for the volume and obtain tight lower and upper bounds for the volumes of $\calO(P(\bw))$ when $\bw$ is of fixed length.
In the remaining sections, we study the order polytope of a related poset $Q_{\bw}$, which is the poset of join-irreducibles of the generalized snake poset.

\subsection{Volume of the order polytope of generalized snake posets}
Recall that the volume of an order polytope $\calO(P)$ is determined by the number of linear extensions of the poset $P$. 
Thus, to study the volume of $\calO(P(\bw))$ we consider the recursive structure of the poset of upper order ideals of $P(\bw)$.  
Because of the definition of the generalized snake poset $P(\bw)$, the minimal element of $J(P(\bw))$ is $\widehat0= \langle 2n+3 \rangle= P(\bw)$ and the maximal element is $\widehat1 = \emptyset$.
\begin{lemma}\label{lem.JPw}
 Let $\bw=w_0w_1\cdots w_n$ be a generalized snake word.
 If $k\geq0$ is the largest index such that $w_k\neq w_n$, then $J(P(\bw))=$
 $$J(P(w_0w_1\cdots w_{n-1})) \cup 
    \left\{ \langle 2n+3\rangle, \langle 2n+2\rangle, \langle 2n+2,2k+2 \rangle \right\} \cup 
    \left\{\langle 2n+2, 2k+2i+1\rangle \right\}_{i=1}^{n-k}.$$ 
\end{lemma}

\begin{proof}
First note that $P(\bw) = P(w_0w_1\cdots w_{n-1}) \cup \{2n+2,2n+3\}$, where $2n+3 \prec 2n+1$, $2n+3\prec 2n+2$, and $2n+2\prec 2n$ or $2n+2 \prec 2n-1$.
One can see that $J(P(w_0w_1\cdots w_{n-1}))$ is contained in $J(P(\bw))$.
The added elements $2n+3$ and $2n+2$ generate the upper order ideals $\langle 2n+3\rangle$ and $\langle 2n+2\rangle$, respectively.
Since $2n+3$ is comparable with every other element of $P(\bw)$, it is not in the minimal generating set of any other upper order ideal. 
The only elements of $P(w_0w_1\cdots w_{n-1})$ which are not comparable with $2n+2$ are $2k+2$ and $\{2k+2i+1\}_{i=1}^{n-k}$.
Hence, each pair $\{2n+2,2k+2\}$ and $\{2n+2,2k+2i+1\}_{i=1}^{n-k}$ generates an upper order ideal of $P(\bw)$.
Since $2n+1 \prec \cdots \prec 2k+5 \prec 2k+3 \prec 2k+2$, no additional minimal generating sets of upper order ideals are possible. 
\end{proof}

\begin{remark}\label{rem:J(P)construction}
Thus, we see that $J(P(\bw))$ can be constructed by adding a chain of $n-k+3$ elements to the bottom of $J(P(w_0w_1\cdots w_{n-1}))$.
In the Hasse diagram for $J(P(\bw))$, this corresponds to drawing a strip of $n-k+1$ squares.
See Figure~\ref{fig:PandJ(P)} for an illustration.

Notice that in the strip of the $n-k+1$ newly added squares in $J(P(\bw))$, the lowest square (consisting of the four elements $\langle 2n+2,2n+1\rangle, \langle 2n+1\rangle, \langle 2n+2, 2n-1\rangle,  \langle 2n,2n-1\rangle$) lies directly below the topmost square of $J(P(\bw))$ (consisting of the four elements $\langle 0\rangle, \langle 1\rangle, \langle 2\rangle, \langle1,2\rangle$).
Hence from Lemma~\ref{lem.JPw}, we see that the Hasse diagram of $J(P(\bw))$ contains exactly $n$ squares which are lined up directly below the topmost square of $J(P(\bw))$.
We will refer to these squares as the \emph{central squares} of $J(P(\bw))$.
Swapping every letter from $R$ to $L$ and vice versa in $\bw$ corresponds to reflecting $J(P(\bw))$ about this central line of squares.
\end{remark}

\begin{figure}[th!]
\begin{center}
\begin{tikzpicture}
\begin{scope}[xshift=0, yshift=0, scale=0.5]

	\draw[line width = 1pt, color=red] (1,1) -- (0,0) -- (-1,1) -- (0,2);

    \pvx[] at (3,7)  {};
    \pvx[] at (2,6)  {};
	\pvx[label=right:{\tiny$2k$}] at (4,6)  {};
	\pvx[label=right:{\tiny$2k+2$}] at (5,5)  {};	
	\pvx[label=left:{\tiny$2k+1$}] at (3,5)  {};
	\pvx[label=right:{\tiny$2k+3$}] at (4,4)  {};
	\pvx[label=left:{\tiny$2k+4$}] at (2,4)  {};
	\pvx[label=right:{\tiny$2k+5$}] at (3,3)  {};
	\pvx[label=left:{\tiny$2n$}] at (0,2)  {};
	\pvx[label=right:{\tiny$2n+1$}] at (1,1)  {};
	\pvx[label=left:{\tiny\textcolor{red}{$2n+2$}}, color=red] at (-1,1)  {};
	\pvx[label=right:{\tiny\textcolor{red}{$2n+3$}}, color=red] at (0,0)  {};
	
	\draw[line width = 1pt] (1,1) -- (0,2);
	\draw[line width = 1pt] (2,4) -- (3,5) -- (4,4) -- (3,3) -- (2,4);	
    \draw[line width = 1pt] (3,5) -- (4,6) -- (5,5) -- (4,4);
    \draw[line width = 1pt] (3,5) -- (2,6) -- (3,7) -- (4,6);    
    \draw[line width = 1pt] (2,6) -- (1.5,6.5);        
    \draw[line width = 1pt] (3,7) -- (2.5,7.5);  
    \draw[line width = 1pt] (0,2) -- (0.5,2.5);  
    \draw[line width = 1pt] (1,1) -- (1.5,1.5);      
    \draw[line width = 1pt] (1.5,3.5) -- (2,4);  
    \draw[line width = 1pt] (2.5,2.5) -- (3,3);          
	
	\node[] at (3,6)  {\scriptsize $w_{k-1}$};	
	\node[] at (4,5)  {\scriptsize $w_k$};	
	\node[] at (3,4)  {\scriptsize $w_{k+1}$};
	\node[] at (1.3,2.5)  {$\udots$};
	\node[] at (1.4,7.4)  {$\Ddots$};
	\node[] at (1,2)  {};
	\node[] at (0,1)  {\scriptsize \textcolor{red}{$w_n$}};
\end{scope}

\begin{scope}[xshift=160, yshift=120, scale=0.5]
    \draw[line width = 1pt, color = red] (1,-4) -- (0,-5) -- (1.5,-6.5);   
    \draw[line width = 1pt, color = red] (1,-6) -- (2,-5);   
    \draw[line width = 1pt, color = red] (2.5,-7.5) -- (5,-10) -- (6,-9); 
    \draw[line width = 1pt, color = red] (3,-8) -- (4,-7);       
    \draw[line width = 1pt, color = red] (4,-9) -- (5,-8);
    \draw[line width = 1pt, color = red] (5,-10) -- (6,-11);
    \pvx[label=left:{\tiny$\langle 2k, 2k-1\rangle$}] at (5,0)  {};
    \pvx[label=left:{\tiny$\langle 2k+1\rangle$}] at (4,-1)  {};    
    \pvx[] at (6,-1)  {};
    \pvx[label=right:{\tiny$\langle 2k+1, 2k+2\rangle$}] at (5,-2)  {};
    \pvx[] at (2,-3)  {};
    \pvx[] at (3,-4)  {};    
    \pvx[label=left:{\tiny$\langle 2n\rangle$}] at (1,-4)  {};
    \pvx[] at (2,-5)  {};
    \pvx[label=left:{\tiny\textcolor{red}{$\langle 2n+2\rangle$}}, color=red] at (0,-5)  {};
    \pvx[label=left:{\tiny\textcolor{red}{$\langle 2n+2, 2k+2\rangle$}}, color=red] at (1,-6)  {};
    \pvx[label=left:{\tiny\textcolor{red}{$\langle 2n+2, 2n-3\rangle$}}, color=red] at (3,-8)  {};    
    \pvx[] at (4,-7)  {};
    \pvx[] at (5,-6)  {};
    \pvx[label=left:{\tiny\textcolor{red}{$\langle 2n+2, 2n-1\rangle$}}, color=red] at (4,-9)  {};
    \pvx[label=right:{\tiny$\langle 2n, 2n-1\rangle$}] at (5,-8)  {};  
    \pvx[label=left:{\tiny\textcolor{red}{$\langle 2n+2, 2n+1\rangle$}}, color=red] at (5,-10)  {};    
    \pvx[label=right:{\tiny$\langle 2n+1\rangle$}] at (6,-9)  {};   
    \pvx[] at (6,-7)  {};    
    \pvx[label=left:{\tiny\textcolor{red}{$\langle 2n+3\rangle$}}, color=red] at (6,-11)  {};    
    
	\draw[line width = 1pt] (5,0) -- (4,-1) -- (5,-2) -- (6,-1) -- (5,0);
	\draw[line width = 1pt] (1,-4) -- (2,-3) -- (3,-4) -- (2,-5) -- (1,-4);
	\draw[line width = 1pt] (5,-8) -- (4,-7) -- (5,-6) -- (6,-7) -- (5,-8) -- (6,-9);	    
    
    \draw[line width = 1pt] (5,-2) -- (4.5,-2.5);
    \draw[line width = 1pt] (5,-2) -- (5.5,-2.5);    
    \draw[line width = 1pt] (3,-4) -- (3.5,-3.5);    
    \draw[line width = 1pt] (2,-3) -- (2.5,-2.5);     
    \draw[line width = 1pt] (3,-4) -- (3.5,-4.5);
    \draw[line width = 1pt] (2,-5) -- (2.5,-5.5);     
    \draw[line width = 1pt] (5,-6) -- (4.5,-5.5);     
    \draw[line width = 1pt] (4,-7) -- (3.5,-6.5);         
    \draw[line width = 1pt] (5,-6) -- (5.5,-5.5); 
    \draw[line width = 1pt] (5,0) -- (4.5,0.5);    
    \draw[line width = 1pt] (6,-1) -- (6.5,-0.5);    
    \draw[line width = 1pt] (5,0) -- (5.5,0.5);  
    \draw[line width = 1pt] (4,-1) -- (3.5,-1.5);      
    

	\node[] at (2.4,-6.7)  {\textcolor{red}{$\Ddots$}};
	\node[] at (3.4,-5.7)  {$\Ddots$};
	\node[] at (6.4,0.6)  {$\udots$};
	\node[] at (3.4,-2.4)  {$\udots$};
	\node[] at (5,-4)  {$\vdots$};
	\node[] at (5,1.5)  {$\vdots$};
\end{scope}
\end{tikzpicture}
\end{center}
\caption{An illustration of Lemma~\ref{lem.JPw}. On the left is a portion of a generalized snake poset $P(\bw)$ and on the right is the corresponding poset of upper order ideals $J(P(\bw))$. To construct $J(P(\bw))$ from $J(P(w_0\cdots w_{n-1})$ is to add $n-k+3$ elements, with cover relations shown in red in the Hasse diagram on the right.
}
\label{fig:PandJ(P)}
\end{figure}
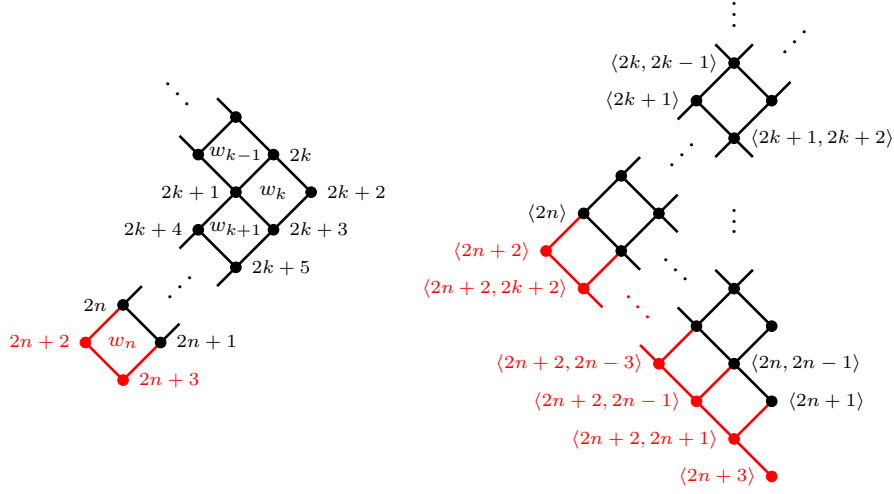

The normalized volume of the order polytope $\mathcal{O}(P(\bw))$ can be computed by a recursive formula involving Catalan numbers.

\begin{definition}
For $m\geq0$, the \emph{$m$-th Catalan number} is $\mathrm{Cat}(m)=\frac{1}{m+1}\binom{2m}{m}$.
\end{definition}
The Catalan number $\mathrm{Cat}(m)$ enumerates \emph{Dyck paths}, which are lattice paths from $(0,0)$ to $(m,m)$ that do not fall below the line $y=x$.

\begin{theorem}
For $n\geq0$, let $\bw=w_0 w_1\cdots w_n$ be a generalized snake word.
If $k\geq0$ is the largest index such that $w_k\neq w_n$, then the normalized volume $v_n$ of $\mathcal{O}(P(\bw))$ is given recursively by 
$$v_n = \mathrm{Cat}(n-k+1)v_{k}
    +\left(\mathrm{Cat}(n-k+2)-2\cdot \mathrm{Cat}(n-k+1)\right)v_{k-1}$$
with $v_{-1}=1$ and $v_0=2$.
\end{theorem}

\begin{proof}
The normalized volume of $\mathcal{O}(P(\bw))$ is the number of linear extensions of $P(\bw)$, and the set of linear extensions of $P(\bw)$ is in bijection with the set of maximal chains in $J(P(\bw))$, so we enumerate the latter.

Let $c(p_1,..,p_j)$ denote the number of maximal chains in $J(P(\bw))$ which contain the elements $p_1,...,p_j$. 
Each maximal chain in $J(P(\bw))$ contains at least one of $\langle 2k, 2k-1\rangle$ or $\langle 2k+1, 2k+2\rangle$, as can be seen in Figure \ref{fig:PandJ(P)}, so the total number of maximal chains in $J(P(\bw))$ is then $c(\emptyset)$ and is given by 
$$ c(\emptyset) = c(\langle 2k+1, 2k+2\rangle) + c(\langle 2k, 2k-1\rangle ) 
	- c(\langle 2k+1, 2k+2\rangle, \langle 2k, 2k-1\rangle). $$ 
Note that $c(\langle 2k+1, 2k+2\rangle)$ is the product of the number of maximal chains in the interval $[\langle 2n+3\rangle, \langle 2k+1,2k+2\rangle]$ and the number of maximal chains in the interval $[\langle 2k+1,2k+2\rangle, \emptyset]$. 
There are $v_{k}$ many maximal chains in $[\langle 2k+1,2k+2\rangle, \emptyset]$, and the maximal chains in $[\langle 2n+3\rangle,\langle 2k+1,2k+2\rangle]$ are counted by $\mathrm{Cat}(n-k+1)$, as they can be viewed as Dyck paths. 
Therefore, $c(\langle 2k+1, 2k+2\rangle) =  \mathrm{Cat}(n-k+1) v_k$. 
Similarly, one sees that $c(\langle 2k, 2k-1\rangle) = \mathrm{Cat}(n-k+2)v_{k-1}$. 

Finally, $c(\langle 2k+1, 2k+2\rangle, \langle 2k, 2k-1\rangle)$ is given by $2 \mathrm{Cat}(n-k+1) v_{k-1}$, as there are two ways to form a maximal chain in $J(P(\bw))$ from a maximal chain in $[\langle 2k,2k-1\rangle, \emptyset]$ and a maximal chain in $[\langle 2n+3\rangle, \langle 2k+1,2k+2\rangle]$. 
Therefore,
$$c(\emptyset) = \mathrm{Cat}(n-k+1)v_k + \mathrm{Cat}(n-k+2)v_{k-1} - 2\mathrm{Cat}(n-k+1)v_{k-1}. $$
\end{proof}

Focusing our attention on the snake poset $S_n=P(\varepsilon LRLR\cdots)$, the letters alternate so we have $n-k=1$ at every step, which leads to the following corollary.

\begin{corollary}
The normalized volume of $\mathcal{O}(S_n)$ with $n\geq 0$ is given recursively by
$$ v_n = 2v_{n-1} + v_{n-2},$$
with $v_{-1}=1$ and $v_0 = 2$.  
These are the Pell numbers.
\qed
\end{corollary}

In the case of the ladder poset $\mL_n=P(\varepsilon LLLL\cdots)$, we have $k=0$ at every step, and hence we have the following well-known result as a corollary.

\begin{corollary}
The normalized volume of $\mathcal{O}(\mL_n)$ with $n\geq 0$ is given by 
$$v_n = \mathrm{Cat}(n+2).$$
\qed
\end{corollary}

We end this section by showing that the normalized volume of an order polytope $\mathcal{O}(P(\bw))$ of a generalized snake poset is bounded above and below by the volume of the order polytope of the ladder poset and the snake poset, respectively.

Let $W_n$ denote the set of generalized snake words of length $n$.
For $i=1,\ldots, n$, define a \emph{swap operation} $f_i: W_n \rightarrow W_n$ by letting $f_i(\bw)$ be the word obtained from $\bw$ by swapping all letters with indices greater than or equal to $i$ to the opposite letter.

\begin{lemma}\label{swap-vol}
Let $\bw$ be a generalized snake word. 
Then
$$\vol(\mathcal{O}(P(f_i(\bw)))) \leq \vol(\mathcal{O}(P(\bw)))$$ 
whenever $w_{i-1}=w_i$ or $i=1$. 
Furthermore, equality occurs only when $i=1$.  
\end{lemma}

\begin{proof}
Consider the maximal chains in $J(P(\bw))$ and for the moment assume  $i\neq 1$. 
Without loss of generality assume that $w_i = w_{i-1} = L$ (the case $w_i = w_{i-1} = R$ is symmetric). 

Let $A = \langle 2i, 2i-1 \rangle$ and $B= \langle 2i-2,2i-3\rangle$.
In light of Remark~\ref{rem:J(P)construction} we see that $J(P(f_i(\bw)))$ consists of a union of the upper order ideal generated by $A$ in $J(P(\bw))$ and a reflected lower order ideal generated by $B$.
Figure \ref{fig:proofByPicture} provides an illustration.

In $J(P(f_i(\bw)))$, all maximal chains contain at least one of $A$ or $B$.
Maximal chains containing $A$ in $J(P(\bw))$ corresponds bijectively to chains containing $A$ in $J(P(f_i(\bw)))$ by reflecting the portion of the chain after $A$. 
Similarly, the chains containing $B$ are in bijection in both posets.
Since $J(P(\bw))$ contains maximal chains that pass through neither $A$ nor $B$, there are strictly more maximal chains in $J(P(\bw))$. 
Hence, $\vol(\mathcal{O}(P(f_i(\bw)))) \le  \vol(\mathcal{O}(P(\bw)))$.

Finally, in the case $i=1$, $P(\bw)$ and $P(f_1(\bw))$ are isomorphic via a reflection, and so their posets of upper order ideals are isomorphic.
\end{proof}

\begin{theorem}\label{thm:minmaxvolumes}
For any generalized snake word $\bw= w_0w_1\cdots w_n$ of length $n$,
$$\vol \mathcal{O}(S_n) \leq 
\vol \mathcal{O}(P(\bw)) \leq 
\vol \mathcal{O}(\mL_n).$$
\end{theorem}

 \begin{proof}
First, we show that $\vol \mathcal{O}(S_n) \leq \vol \mathcal{O}(P(\bw))$.  
Let $2\leq i_1<i_2<\dots<i_k\leq n$ be the set of indices such that $w_{i_j}=w_{i_j-1}$.  
Applying a swap operation at any index strictly smaller then $i_j$ yields a word whose letters indexed by $i_j$ and $i_j-1$ are still the same.   
Then for any $j\in[k]$, the letters indexed by $i_j$ and $i_{j}-1$ are also the same in the word $f_{i_{j-1}}f_{i_{j-2}}\dots f_{i_1}(\bw)$.   
By Lemma~\ref{swap-vol}, we can conclude that 
 \[\vol \mathcal{O} (P (f_{i_{k}}f_{i_{k-1}}\dots f_{i_1}(\bw))) \leq \vol \mathcal{O}(P(\bw)).\]
Moreover, by the construction of $i_j$'s and the definition of the swap operation, no two adjacent letters with indices up to $i_j$ are the same in $f_{i_{j}}f_{i_{j-1}}\dots f_{i_1}(\bw)$.  
This shows that  $P(f_{i_{k}}f_{i_{k-1}}\dots f_{i_1}(\bw))$ equals $P(\varepsilon LRLR \dots)=S_n$ or $P(\varepsilon RLRL \dots)=f_1(S_n)$.  
By Lemma~\ref{swap-vol}, applying $f_1$ does not change the volume of the order polytope, so we conclude that  $\vol \mathcal{O}(S_n) \leq \vol \mathcal{O}(P(\bw))$.

Now, we show the second part of the inequality that $\vol \mathcal{O}(P(\bw)) \leq \vol \mathcal{O}(\mL_n)$.  Let $2\leq i_1<i_2<\dots<i_k\leq n$ be the set of indices such that $w_{i_j}\not=w_{i_j-1}$. Then for any $j\in [k]$ the letters in $f_{i_{j-1}}\dots f_{i_1}(\bw)$ with indices strictly smaller then $i_j$ are the same, but the the letters with indices $i_j$ and $i_j-1$ are different.  Then the letters in $f_{i_j}f_{i_{j-1}}\dots f_{i_1}(\bw)$ with indices $i_j$ and $i_j-1$ are the same and Lemma~\ref{swap-vol} implies that 
\[  \vol \mathcal{O}(P(\bw)) \leq \vol \mathcal{O} (P (f_{i_{k}}f_{i_{k-1}}\dots f_{i_1}(\bw))). \]
Furthermore, by construction all letters in $f_{i_{k}}f_{i_{k-1}}\dots f_{i_1}(\bw)$ are the same so its generalized snake poset equals $P(\varepsilon LLL\dots ) = \mL_n$ or $P(\varepsilon RRR\dots) =f_1(L_n)$.  
By the same reasoning as above we conclude that $\vol \mathcal{O}(P(\bw)) \leq \vol \mathcal{O}(\mL_n)$.
 \end{proof}

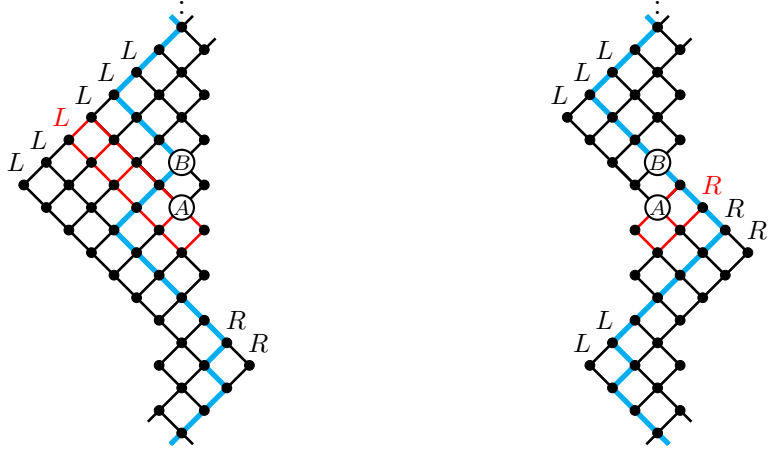
\begin{figure}[ht!]
\begin{center}
\begin{tikzpicture}
\begin{scope}[xshift=0, yshift=0, scale=0.30]

    \draw[line width = 1pt, color = red] (1,-4) -- (0,-5) -- (2,-7) -- (3,-6);   
    \draw[line width = 1pt, color = red] (1,-6) -- (2,-5);   
    \draw[line width = 1pt, color = red] (2,-7) -- (5,-10) -- (6,-9) -- (5,-8); 
    \draw[line width = 1pt, color = red] (3,-8) -- (4,-7);       
    \draw[line width = 1pt, color = red] (4,-9) -- (5,-8);
    \draw[line width = 1pt, color = red] (1,-4) -- (5,-8);

    
    \draw[line width = 1pt] (5,-10) -- (6,-11);
    \draw[line width = 1pt] (-1,-6) -- (8,-15);    
    \draw[line width = 1pt] (-2,-7) -- (7,-16);

    \draw[line width = 1pt] (-2,-7) -- (0,-5);
    \draw[line width = 1pt] (-1,-8) -- (1,-6);
    \draw[line width = 1pt] (0,-9) -- (2,-7);
    \draw[line width = 1pt] (1,-10) -- (3,-8);
    \draw[line width = 1pt] (2,-11) -- (4,-9);
    \draw[line width = 1pt] (3,-12) -- (5,-10);    
    \draw[line width = 1pt] (4,-13) -- (6,-11);
    \draw[line width = 1pt] (5,-14) -- (6,-13);    

    \draw[line width = 1pt] (5,-14) -- (4,-15) -- (6,-17) -- (8,-15);    
    \draw[line width = 1pt] (7,-14) -- (4,-17) -- (5,-18) -- (6,-17);  

    \draw[line width = 1pt] (5,-18) -- (4.5,-18.5);  
    \draw[line width = 1pt] (5,-18) -- (5.5,-18.5);
    \draw[line width = 1pt] (4,-17) -- (3.5,-17.5);

	\draw[line width = 1pt] (5,0) -- (4,-1) -- (5,-2) -- (6,-1) -- (5,0);
	\draw[line width = 1pt] (1,-4) -- (2,-3) -- (3,-4) -- (2,-5) -- (1,-4);
	\draw[line width = 1pt] (5,-8) -- (4,-7) -- (5,-6) -- (6,-7) -- (5,-8);	    
    
    \draw[line width = 1pt] (5,-2) -- (3,-4);
    \draw[line width = 1pt] (5,-2) -- (6,-3) -- (5,-4) -- (6,-5) -- (5,-6);    
    \draw[line width = 1pt] (2,-3) -- (4,-1);     
    \draw[line width = 1pt] (3,-4) -- (4,-5);
    \draw[line width = 1pt] (2,-5) -- (3,-6);     
    \draw[line width = 1pt] (5,-6) -- (4,-5);     
    \draw[line width = 1pt] (4,-7) -- (3,-6);         
    \draw[line width = 1pt] (3,-2) -- (4,-3) -- (5,-4) -- (4,-5) -- (3,-6) -- (4,-7); 
    \draw[line width = 1pt] (5,0) -- (4.5,-0.5);    
    \draw[line width = 1pt] (6,-1) -- (6.5,-0.5);    
    \draw[line width = 1pt] (5,0) -- (5.5,0.5);  
    \draw[line width = 1pt] (4,-1) -- (3.5,-1.5);    
    
    \draw[line width = 1pt, color = red] (1,-4) -- (5,-8);
    
    \draw[line width = 2pt, color = cyan] (4.5,0.5) -- (5,0) -- (2,-3) -- (5,-6) -- (2,-9) -- (4,-11) -- (7,-14) -- (6,-15) -- (7,-16) -- (4.5,-18.5);
    
    \node[style={circle,draw,inner sep=1.3pt, fill=black}] ()  at (5,0)  {};
    \node[style={circle,draw,inner sep=1.3pt, fill=black}] ()  at (4,-1)  {};    
    \node[style={circle,draw,inner sep=1.3pt, fill=black}] ()  at (6,-1)  {};
    \node[style={circle,draw,inner sep=1.3pt, fill=black}] ()  at (5,-2)  {};
    \node[style={circle,draw,inner sep=1.3pt, fill=black}] ()  at (2,-3)  {};
    \node[style={circle,draw,inner sep=1.3pt, fill=black}] ()  at (3,-4)  {};    
    \node[style={circle,draw,inner sep=1.3pt, fill=black}] ()  at (1,-4)  {};
    \node[style={circle,draw,inner sep=1.3pt, fill=black}] ()  at (2,-5)  {};
    \node[style={circle,draw,inner sep=1.3pt, fill=black}] ()  at (0,-5)  {};
    \node[style={circle,draw,inner sep=1.3pt, fill=black}] ()  at (1,-6)  {};
    \node[style={circle,draw,inner sep=1.3pt, fill=black}] ()  at (3,-8)  {};    
    \node[style={circle,draw,inner sep=1.3pt, fill=black}] ()  at (4,-7)  {};
    \node[style={circle,draw,inner sep=1.3pt, fill=black}] ()  at (4,-9)  {};
    
    \node[style={circle,thick,draw,inner sep=0.3pt, fill=white}] ()  at (5,-8)  {{\scriptsize $A$}};  
    \node[style={circle,thick,draw,inner sep=0.3pt, fill=white}] ()  at (5,-6)  {{\scriptsize $B$}};  
    
    \node[style={circle,draw,inner sep=1.3pt, fill=black}] ()  at (5,-10)  {};    
    \node[style={circle,draw,inner sep=1.3pt, fill=black}] ()  at (6,-9)  {};   
    \node[style={circle,draw,inner sep=1.3pt, fill=black}] ()  at (6,-7)  {};    
    \node[style={circle,draw,inner sep=1.3pt, fill=black}] ()  at (6,-11)  {};    
    \node[style={circle,draw,inner sep=1.3pt, fill=black}] ()  at (3,-2)  {};    
    \node[style={circle,draw,inner sep=1.3pt, fill=black}] ()  at (4,-3)  {};    
    \node[style={circle,draw,inner sep=1.3pt, fill=black}] ()  at (4,-5)  {};  
    \node[style={circle,draw,inner sep=1.3pt, fill=black}] ()  at (3,-6)  {};
    \node[style={circle,draw,inner sep=1.3pt, fill=black}] ()  at (2,-7)  {};      
    \node[style={circle,draw,inner sep=1.3pt, fill=black}] ()  at (5,-4)  {};          
    \node[style={circle,draw,inner sep=1.3pt, fill=black}] ()  at (6,-3)  {};          
    \node[style={circle,draw,inner sep=1.3pt, fill=black}] ()  at (6,-5)  {};      

    \node[style={circle,draw,inner sep=1.3pt, fill=black}] ()  at (-1,-6)  {};     
    \node[style={circle,draw,inner sep=1.3pt, fill=black}] ()  at (0,-7)  {};
    \node[style={circle,draw,inner sep=1.3pt, fill=black}] ()  at (1,-8)  {};
    \node[style={circle,draw,inner sep=1.3pt, fill=black}] ()  at (2,-9)  {};
    \node[style={circle,draw,inner sep=1.3pt, fill=black}] ()  at (3,-10)  {};    
    \node[style={circle,draw,inner sep=1.3pt, fill=black}] ()  at (4,-11)  {};
    \node[style={circle,draw,inner sep=1.3pt, fill=black}] ()  at (5,-12)  {};    
    \node[style={circle,draw,inner sep=1.3pt, fill=black}] ()  at (6,-13)  {};    
    
    \node[style={circle,draw,inner sep=1.3pt, fill=black}] ()  at (-2,-7)  {};     
    \node[style={circle,draw,inner sep=1.3pt, fill=black}] ()  at (-1,-8)  {};
    \node[style={circle,draw,inner sep=1.3pt, fill=black}] ()  at (0,-9)  {};    
    \node[style={circle,draw,inner sep=1.3pt, fill=black}] ()  at (1,-10)  {};
    \node[style={circle,draw,inner sep=1.3pt, fill=black}] ()  at (2,-11)  {};
    \node[style={circle,draw,inner sep=1.3pt, fill=black}] ()  at (3,-12)  {};    
    \node[style={circle,draw,inner sep=1.3pt, fill=black}] ()  at (4,-13)  {};
    \node[style={circle,draw,inner sep=1.3pt, fill=black}] ()  at (5,-14)  {};    
    \node[style={circle,draw,inner sep=1.3pt, fill=black}] ()  at (6,-15)  {};    

    \node[style={circle,draw,inner sep=1.3pt, fill=black}] ()  at (7,-14)  {};   
    \node[style={circle,draw,inner sep=1.3pt, fill=black}] ()  at (8,-15)  {};   
    \node[style={circle,draw,inner sep=1.3pt, fill=black}] ()  at (7,-16)  {};       
    \node[style={circle,draw,inner sep=1.3pt, fill=black}] ()  at (4,-15)  {};     
    \node[style={circle,draw,inner sep=1.3pt, fill=black}] ()  at (5,-16)  {};
    \node[style={circle,draw,inner sep=1.3pt, fill=black}] ()  at (6,-17)  {};
    \node[style={circle,draw,inner sep=1.3pt, fill=black}] ()  at (5,-18)  {};
    \node[style={circle,draw,inner sep=1.3pt, fill=black}] ()  at (4,-17)  {};
    
	\node[] at (5,1)  {$\vdots$};
	
    
    \node[anchor = east] at (3.5,-1)  {{\small $L$}};		
    \node[anchor = east] at (2.5,-2)  {{\small $L$}};		
    \node[anchor = east] at (1.5,-3)  {{\small $L$}};		

    \node[anchor = east] at (0.5,-4)  {{\color{red}{\small $L$}}};		
    \node[anchor = east] at (-0.5,-5)  {{\small $L$}};		
    \node[anchor = east] at (-1.5,-6)  {{\small $L$}};		
    
    \node[anchor = west] at (6.5,-13)  {{\small $R$}};		    
    \node[anchor = west] at (7.5,-14)  {{\small $R$}};    
\end{scope}

\begin{scope}[xshift=180, yshift=0, scale=0.30]


    \draw[line width = 1pt, color = red] (7,-8) -- (5,-10) -- (4,-9) -- (5,-8); 
    \draw[line width = 1pt, color = red] (7,-8) -- (6,-7);       
    \draw[line width = 1pt, color = red] (6,-9) -- (5,-8);

	\draw[line width = 1pt] (5,0) -- (4,-1) -- (5,-2) -- (6,-1) -- (5,0);
	\draw[line width = 1pt] (1,-4) -- (2,-3) -- (3,-4) -- (2,-5) -- (1,-4);
	\draw[line width = 1pt] (5,-8) -- (4,-7) -- (5,-6) -- (6,-7) -- (5,-8);	    
    
    \draw[line width = 1pt] (5,-10) -- (4,-11);
    \draw[line width = 1pt] (8,-9) -- (2,-15);    
    \draw[line width = 1pt] (9,-10) -- (3,-16);

    \draw[line width = 1pt] (9,-10) -- (7,-8);
    \draw[line width = 1pt] (8,-11) -- (6,-9);
    \draw[line width = 1pt] (7,-12) -- (5,-10);    
    \draw[line width = 1pt] (6,-13) -- (4,-11);
    \draw[line width = 1pt] (5,-14) -- (4,-13);    

    \draw[line width = 1pt] (5,-14) -- (6,-15) -- (4,-17) -- (2,-15);    
    \draw[line width = 1pt] (3,-14) -- (6,-17) -- (5,-18) -- (4,-17);  

    \draw[line width = 1pt] (5,-18) -- (4.5,-18.5);  
    \draw[line width = 1pt] (5,-18) -- (5.5,-18.5);

    \draw[line width = 1pt] (5,-2) -- (3,-4);
    \draw[line width = 1pt] (5,-2) -- (6,-3) -- (5,-4) -- (6,-5) -- (5,-6);    
    \draw[line width = 1pt] (2,-3) -- (4,-1);     
    \draw[line width = 1pt] (3,-4) -- (4,-5);
    \draw[line width = 1pt] (2,-5) -- (3,-6);     
    \draw[line width = 1pt] (5,-6) -- (4,-5);     
    \draw[line width = 1pt] (4,-7) -- (3,-6);         
    \draw[line width = 1pt] (3,-2) -- (4,-3) -- (5,-4) -- (4,-5) -- (3,-6) -- (4,-7); 
    \draw[line width = 1pt] (5,0) -- (4.5,0.5);    
    \draw[line width = 1pt] (6,-1) -- (6.5,-0.5);    
    \draw[line width = 1pt] (5,0) -- (5.5,0.5);  
    \draw[line width = 1pt] (4,-1) -- (3.5,-1.5);      
    
    \draw[line width = 1pt] (6,-17) -- (6.5,-17.5);  
    
    \draw[line width = 1pt, color = red] (6,-7) -- (5,-8);
    
    \draw[line width = 2pt, color = cyan] (4.5,0.5) -- (5,0) -- (2,-3) -- (5,-6) -- (8,-9) -- (6,-11) -- (3,-14) -- (4,-15) -- (3,-16) -- (5.5,-18.5);
    
    \node[style={circle,draw,inner sep=1.3pt, fill=black}] ()  at (5,0)  {};
    \node[style={circle,draw,inner sep=1.3pt, fill=black}] ()  at (4,-1)  {};    
    \node[style={circle,draw,inner sep=1.3pt, fill=black}] ()  at (6,-1)  {};
    \node[style={circle,draw,inner sep=1.3pt, fill=black}] ()  at (5,-2)  {};
    \node[style={circle,draw,inner sep=1.3pt, fill=black}] ()  at (2,-3)  {};
    \node[style={circle,draw,inner sep=1.3pt, fill=black}] ()  at (3,-4)  {};    
    \node[style={circle,draw,inner sep=1.3pt, fill=black}] ()  at (1,-4)  {};
    \node[style={circle,draw,inner sep=1.3pt, fill=black}] ()  at (2,-5)  {};
    \node[style={circle,draw,inner sep=1.3pt, fill=black}] ()  at (7,-8)  {};    
    \node[style={circle,draw,inner sep=1.3pt, fill=black}] ()  at (4,-7)  {};
    \node[style={circle,draw,inner sep=1.3pt, fill=black}] ()  at (6,-9)  {};
    \node[style={circle,draw,inner sep=1.3pt, fill=black}] ()  at (5,-10)  {};    
    \node[style={circle,draw,inner sep=1.3pt, fill=black}] ()  at (4,-9)  {};   
    \node[style={circle,draw,inner sep=1.3pt, fill=black}] ()  at (6,-7)  {};    
    \node[style={circle,draw,inner sep=1.3pt, fill=black}] ()  at (4,-11)  {};
    
    \node[style={circle,thick,draw,inner sep=0.3pt, fill=white}] ()  at (5,-8)  {{\scriptsize $A$}};
    \node[style={circle,thick,draw,inner sep=0.3pt, fill=white}] ()  at (5,-6)  {{\scriptsize $B$}}; 
    
    \node[style={circle,draw,inner sep=1.3pt, fill=black}] ()  at (3,-2)  {};    
    \node[style={circle,draw,inner sep=1.3pt, fill=black}] ()  at (4,-3)  {};    
    \node[style={circle,draw,inner sep=1.3pt, fill=black}] ()  at (4,-5)  {};  
    \node[style={circle,draw,inner sep=1.3pt, fill=black}] ()  at (3,-6)  {};  
    \node[style={circle,draw,inner sep=1.3pt, fill=black}] ()  at (5,-4)  {};          
    \node[style={circle,draw,inner sep=1.3pt, fill=black}] ()  at (6,-3)  {};          
    \node[style={circle,draw,inner sep=1.3pt, fill=black}] ()  at (6,-5)  {};      

    \node[style={circle,draw,inner sep=1.3pt, fill=black}] ()  at (8,-9)  {};
    \node[style={circle,draw,inner sep=1.3pt, fill=black}] ()  at (7,-10)  {};    
    \node[style={circle,draw,inner sep=1.3pt, fill=black}] ()  at (6,-11)  {};
    \node[style={circle,draw,inner sep=1.3pt, fill=black}] ()  at (5,-12)  {};    
    \node[style={circle,draw,inner sep=1.3pt, fill=black}] ()  at (4,-13)  {};    
    
    \node[style={circle,draw,inner sep=1.3pt, fill=black}] ()  at (9,-10)  {};
    \node[style={circle,draw,inner sep=1.3pt, fill=black}] ()  at (8,-11)  {};
    \node[style={circle,draw,inner sep=1.3pt, fill=black}] ()  at (7,-12)  {};    
    \node[style={circle,draw,inner sep=1.3pt, fill=black}] ()  at (6,-13)  {};
    \node[style={circle,draw,inner sep=1.3pt, fill=black}] ()  at (5,-14)  {};    
    \node[style={circle,draw,inner sep=1.3pt, fill=black}] ()  at (4,-15)  {};    

    \node[style={circle,draw,inner sep=1.3pt, fill=black}] ()  at (3,-14)  {};   
    \node[style={circle,draw,inner sep=1.3pt, fill=black}] ()  at (2,-15)  {};   
    \node[style={circle,draw,inner sep=1.3pt, fill=black}] ()  at (3,-16)  {};       
    \node[style={circle,draw,inner sep=1.3pt, fill=black}] ()  at (6,-15)  {};     
    \node[style={circle,draw,inner sep=1.3pt, fill=black}] ()  at (5,-16)  {};
    \node[style={circle,draw,inner sep=1.3pt, fill=black}] ()  at (4,-17)  {};
    \node[style={circle,draw,inner sep=1.3pt, fill=black}] ()  at (5,-18)  {};
    \node[style={circle,draw,inner sep=1.3pt, fill=black}] ()  at (6,-17)  {};
    

	\node[] at (5,1)  {$\vdots$};


    \node[anchor = east] at (3.5,-1)  {{\small $L$}};		
    \node[anchor = east] at (2.5,-2)  {{\small $L$}};		
    \node[anchor = east] at (1.5,-3)  {{\small $L$}};		

    \node[anchor = west] at (6.5,-7)  {{\color{red}{\small $R$}}};		
    \node[anchor = west] at (7.5,-8)  {{\small $R$}};		
    \node[anchor = west] at (8.5,-9)  {{\small $R$}};		
    
    \node[anchor = east] at (3.5,-13)  {{\small $L$}};		    
    \node[anchor = east] at (2.5,-14)  {{\small $L$}};

\end{scope}
\end{tikzpicture}
\end{center}
\caption{On the left is a snippet of a poset $J(P(\bw))$ where $\bw$ contains the sequence $\cdots RLLL{\color{red}L}LLRRL\cdots$. 
On the right is the corresponding snippet of $J(P(f_i(\bw)))$, where $i$ is the index of the red $L$ in $\bw$. 
The corresponding portion in $f_i(\bw)$ is $\cdots RLLL{\color{red}R}RRLLR\cdots$. 
The blue paths demonstrate the bijective correspondence between maximal chains in $J(P(\bw))$ through $A$ or $B$ with maximal chains in $J(P(f_i(\bw)))$.}
\label{fig:proofByPicture}
\end{figure}

\section{A combinatorial interpretation of circuits}\label{sec:circuits} 

In the remainder of this article, we study the properties of the order polytope of a poset $Q_{\bw}$ whose lattice of filters is a generalized snake poset.

Define $\hatP(\bw)$ to be the generalized snake poset $P(\bw)$ with $\widehat0$ and $\widehat1$ adjoined, and when $\bw$ is clear from context we write $\hatP$.  
Given $\bw=w_0w_1\cdots w_n$,  $\hatP=\hatP(\bw)$ is a distributive lattice with order $2n+6$ because $\hatP$ does not contain a copy of the smallest non-modular lattice with five elements and does not contain a sublattice isomorphic to a three-element antichain with a $\widehat{0}$ and $\widehat{1}$ added.
Let $Q_{\bw}=\mathrm{Irr}_\wedge(\hatP)$ denote the poset of meet-irreducibles of $\hatP$. 
Heuristically, $\Irr_\wedge(\hatP)$ is obtained from $\hatP$ by removing $\widehat1$, and every vertex which is at the bottom of a bounded face in the Hasse diagram. 
See Figure~\ref{fig:graph_poset_meet}.
By the fundamental theorem of finite distributive lattices, $\hatP \cong J(Q_{\bw})$, where $J(Q_{\bw})$ is the lattice of filters of $Q_{\bw}$, ordered by reverse inclusion.

We construct a graph $G=G(\bw)$ associated to $\hatP=\hatP(\bw)$ as follows. 
If $\bw=w_0w_1\cdots w_n$, the vertex set of $G$ is $V(G)=\{w_0,w_1,\ldots, w_n\}$. 
The edge set of $G$ is given by
$$E(G) = \{(w_i,w_{i+1}) \mid i=0,\ldots, n-1\} \cup \{(w_i,w_{i+2}) \mid \hbox{ if $w_iw_{i+1}w_{i+2} = LLR$ or $RRL$}\}.$$
In other words, $G$ consists of the path of length $n$ on the vertices $w_0, \ldots, w_n$, with a $3$-cycle for each turn $LLR$ or $RRL$ in $\bw$. See Figure~\ref{fig:graph_poset_meet}.
We denote the set of nonempty connected induced subgraphs of $G(\bw)$ by $\calG(\bw)$.

The Hasse diagram of $\hatP(\bw)$ can be embedded on the plane so that its edges are non-crossing where each bounded face of the embedded Hasse diagram has degree $4$ given by the length of the cycle bounding the face.
We call these bounded faces the {\em squares} of $\hatP(\bw)$.

There is a one-to-one correspondence between the squares of $\hatP(\bw)$ and the letters of $\bw$ by realizing $G=G(\bw)$ as follows.
Consider each square in the Hasse diagram $\mathrm{Hasse}(\hatP)$ as a vertex, then form an edge between squares when they intersect in the plane, as shown in Figure~\ref{fig:graph_poset_meet}.
To each vertex $w_i$ of $G$, we denote by $\Sq(w_i)$ the four elements of $\hatP$ contained in the $4$-cycle which bounds the face of $\mathrm{Hasse}(\hatP)$ corresponding to $w_i$.



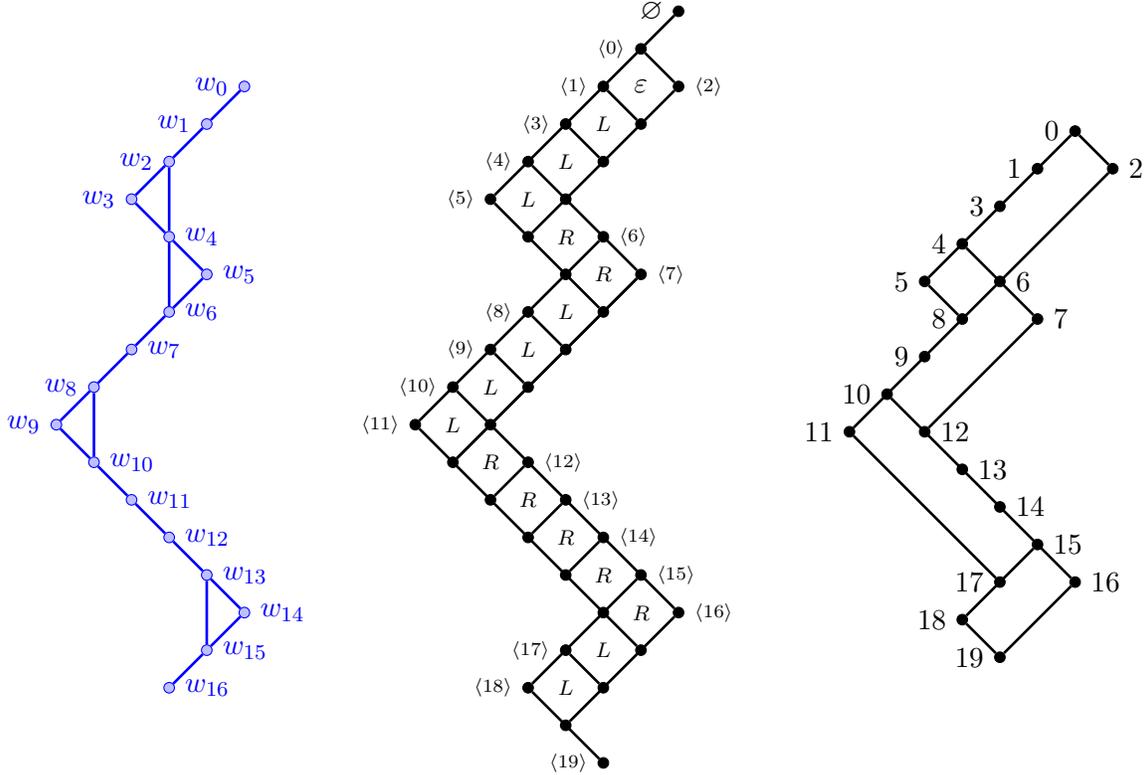
\begin{figure}[ht!] 
\begin{center}
\begin{tikzpicture}[scale=.5]
\begin{scope}[xshift=-300, yshift=0] 
	\gvx[label=left:$\textcolor{blue}{w_0}$](c00) at (1,18) {};	
	\gvx[label=left:$\textcolor{blue}{w_1}$](c01) at (0,17) {};
	\gvx[label=left:$\textcolor{blue}{w_2}$](c02) at (-1,16) {};
	\gvx[label=left:$\textcolor{blue}{w_3}$](c03) at (-2,15) {};
	\gvx[label=right:$\textcolor{blue}{w_4}$](c04) at (-1,14) {};
	\gvx[label=right:$\textcolor{blue}{w_5}$](c05) at (0,13) {};
	\gvx[label=right:$\textcolor{blue}{w_6}$](c06) at (-1,12) {};
	\gvx[label=right:$\textcolor{blue}{w_7}$](c07) at (-2,11) {};
	\gvx[label=left:$\textcolor{blue}{w_8}$](c08) at (-3,10) {};
	\gvx[label=left:$\textcolor{blue}{w_9}$](c09) at (-4,9) {};
	\gvx[label=right:$\textcolor{blue}{w_{10}}$](c10) at (-3,8) {};
	\gvx[label=right:$\textcolor{blue}{w_{11}}$](c11) at (-2,7) {};
	\gvx[label=right:$\textcolor{blue}{w_{12}}$](c12) at (-1,6) {};
	\gvx[label=right:$\textcolor{blue}{w_{13}}$](c13) at (0,5) {};
	\gvx[label=right:$\textcolor{blue}{w_{14}}$](c14) at (1,4) {};
	\gvx[label=right:$\textcolor{blue}{w_{15}}$](c15) at (0,3) {};
	\gvx[label=right:$\textcolor{blue}{w_{16}}$](c16) at (-1,2) {};
	
	\draw[line width=1pt, color=blue] (c16)--(c15)--(c14)--(c13)--(c12)--(c11)--(c10)--(c09)--(c08)--(c07)--(c06)--(c05)--(c04)--(c03)--(c02)--(c01)--(c00);
	\draw[line width=1pt, color=blue] (c02)--(c04);
	\draw[line width=1pt, color=blue] (c04)--(c06);
	\draw[line width=1pt, color=blue] (c08)--(c10);
	\draw[line width=1pt, color=blue] (c13)--(c15);
\end{scope}
\begin{scope}[xshift=0, yshift=0] 
	\pvx[label=left:$\emptyset$](emp) at (2,20) {};
	\pvx[label=left:\tiny$\langle0\rangle$](v00) at (1,19) {};
	\pvx[label=left:\tiny$\langle1\rangle$](v01) at (0,18) {};
	\pvx[label=right:\tiny$\langle2\rangle$](v02) at (2,18) {};
	\pvx[label=left:\tiny$\langle3\rangle$](v03) at (-1,17) {};
	\pvx[label=left:\tiny$\langle4\rangle$](v04) at (-2,16) {};
	\pvx[label=left:\tiny$\langle5\rangle$](v05) at (-3,15) {};			
	\pvx[label=right:\tiny$\langle6\rangle$](v06) at (0,14) {};
	\pvx[label=right:\tiny$\langle7\rangle$](v07) at (1,13) {};	
	\pvx[label=left:\tiny$\langle8\rangle$](v08) at (-2,12) {};
	\pvx[label=left:\tiny$\langle9\rangle$](v09) at (-3,11) {};	
	\pvx[label=left:\tiny$\langle10\rangle$](v10) at (-4,10) {};
	\pvx[label=left:\tiny$\langle11\rangle$](v11) at (-5,9) {};
	\pvx[label=right:\tiny$\langle12\rangle$](v12) at (-2,8) {};
	\pvx[label=right:\tiny$\langle13\rangle$](v13) at (-1,7) {};
	\pvx[label=right:\tiny$\langle14\rangle$](v14) at (0,6) {};
	\pvx[label=right:\tiny$\langle15\rangle$](v15) at (1,5) {};
	\pvx[label=right:\tiny$\langle16\rangle$](v16) at (2,4) {};
	\pvx[label=left:\tiny$\langle17\rangle$](v17) at (-1,3) {};
	\pvx[label=left:\tiny$\langle18\rangle$](v18) at (-2,2) {};
	\pvx[label=left:\tiny$\langle19\rangle$](v19) at (0,0) {};
	\pvx[](w0102) at (1,17) {};
	\pvx[](w0103) at (0,16) {};
	\pvx[](w0104) at (-1,15) {};
	\pvx[](w0105) at (-2,14) {};
	\pvx[](w0506) at (-1,13) {};
	\pvx[](w0507) at (0,12) {};
	\pvx[](w0708) at (-1,11) {};
	\pvx[](w0709) at (-2,10) {};
	\pvx[](w0710) at (-3,9) {};
	\pvx[](w0711) at (-4,8) {};
	\pvx[](w1112) at (-3,7) {};
	\pvx[](w1113) at (-2,6) {};
	\pvx[](w1114) at (-1,5) {};
	\pvx[](w1115) at (0,4) {};
	\pvx[](w1116) at (1,3) {};	
	\pvx[](w1617) at (0,2) {};
	\pvx[](w1618) at (-1,1) {};
	
	\draw[line width=1pt] (v19)--(v18);
	\draw[line width=1pt] (v18)--(v15);
	\draw[line width=1pt] (v16)--(v10);
	\draw[line width=1pt] (v11)--(v06);
	\draw[line width=1pt] (v07)--(v04);
	\draw[line width=1pt] (v02)--(v00);
	\draw[line width=1pt] (v05)--(emp);
	\draw[line width=1pt] (w1618)--(v16);
	\draw[line width=1pt] (w1617)--(v17);	
	\draw[line width=1pt] (w1114)--(v14);
	\draw[line width=1pt] (w1113)--(v13);
	\draw[line width=1pt] (w1112)--(v12);	
	\draw[line width=1pt] (w1116)--(v11);
	\draw[line width=1pt] (w0711)--(v07);
	\draw[line width=1pt] (w0709)--(v09);	
	\draw[line width=1pt] (w0708)--(v08);	
	\draw[line width=1pt] (w0711)--(v07);	
	\draw[line width=1pt] (w0507)--(v05);
	\draw[line width=1pt] (w0105)--(v02);
	\draw[line width=1pt] (w0103)--(v03);
	\draw[line width=1pt] (w0102)--(v01);			
	
	\node[](c01) at (1,18) {$\varepsilon$};	
	\node[](c02) at (0,17) {\scriptsize$L$};
	\node[](c03) at (-1,16) {\scriptsize$L$};
	\node[](c04) at (-2,15) {\scriptsize$L$};
	\node[](c05) at (-1,14) {\scriptsize$R$};
	\node[](c06) at (0,13) {\scriptsize$R$};
	\node[](c07) at (-1,12) {\scriptsize$L$};
	\node[](c08) at (-2,11) {\scriptsize$L$};
	\node[](c09) at (-3,10) {\scriptsize$L$};
	\node[](c10) at (-4,9) {\scriptsize$L$};
	\node[](c11) at (-3,8) {\scriptsize$R$};
	\node[](c12) at (-2,7) {\scriptsize$R$};
	\node[](c13) at (-1,6) {\scriptsize$R$};
	\node[](c14) at (0,5) {\scriptsize$R$};
	\node[](c15) at (1,4) {\scriptsize$R$};
	\node[](c16) at (0,3) {\scriptsize$L$};
	\node[](c17) at (-1,2) {\scriptsize$L$};
\end{scope}
\begin{scope}[xshift=300, yshift=80] 
	\pvx[label=left:$0$](v00) at (2,14) {};
	\pvx[label=left:$1$](v01) at (1,13) {};
	\pvx[label=right:$2$](v02) at (3,13) {};
	\pvx[label=left:$3$](v03) at (0,12) {};
	\pvx[label=left:$4$](v04) at (-1,11) {};
	\pvx[label=left:$5$](v05) at (-2,10) {};
	\pvx[label=right:$6$](v06) at (0,10) {};					
	\pvx[label=right:$7$](v07) at (1,9) {};
	\pvx[label=left:$8$](v08) at (-1,9) {};	
	\pvx[label=left:$9$](v09) at (-2,8) {};	
	\pvx[label=left:$10$](v10) at (-3,7) {};
	\pvx[label=left:$11$](v11) at (-4,6) {};
	\pvx[label=right:$12$](v12) at (-2,6) {};
	\pvx[label=right:$13$](v13) at (-1,5) {};
	\pvx[label=right:$14$](v14) at (0,4) {};
	\pvx[label=right:$15$](v15) at (1,3) {};
	\pvx[label=left:$17$](v17) at (0,2) {};
	\pvx[label=right:$16$](v16) at (2,2) {};
	\pvx[label=left:$18$](v18) at (-1,1) {};
	\pvx[label=left:$19$](v19) at (0,0) {};

	\draw[line width=1pt] (v19)--(v16);
	\draw[line width=1pt] (v19)--(v18);
	\draw[line width=1pt] (v18)--(v15);
	\draw[line width=1pt] (v16)--(v10);
	\draw[line width=1pt] (v17)--(v11);
	\draw[line width=1pt] (v11)--(v02);
	\draw[line width=1pt] (v12)--(v07);
	\draw[line width=1pt] (v07)--(v04);
	\draw[line width=1pt] (v08)--(v05);
	\draw[line width=1pt] (v05)--(v00);
	\draw[line width=1pt] (v02)--(v00);
\end{scope}
\end{tikzpicture}
\end{center}
\caption{In the center is the lattice $\hatP(\bw)$ for $\bw=\varepsilon L^3R^2L^4R^5L^2$. Its poset of meet-irreducibles $Q_{\bw}=\mathrm{Irr}_\wedge(\hatP)$ is shown to the right, and the associated graph $G(\bw)$ is shown to the left.} \label{fig:graph_poset_meet}
\end{figure}


\begin{remark}
The volume of $\calO(Q_{\bw})$ equals the number of maximal chains in $\hatP(\bw)$ or, equivalently, in $P(\bw)$. 
By \cite[Section 4]{Propp}, maximal chains in $P(\bw)$ are in bijection with perfect matchings of the Hasse diagram of $P(\bw^*)$, where $\bw^*$ denotes the dual of $\bw$.
Informally, $P(\bw^*)$ is obtained from $P(\bw)$ by replacing three consecutive squares that form a ladder by three squares that form a bend and vice versa. 
Perfect matchings of $P(\bw^*)$ have been extensively studied because they play an important role in the theory of cluster algebras and their total number can be computed via explicit formulas involving continued fractions \cite[Theorem 3.4]{SchifflerCanakci} or certain admissible sequences \cite[Theorem 4.6]{BFGST}. 

Alternatively, we may view the Hasse diagram of $P(\bw)$ as a skew partition $\lambda/\mu$, and maximal chains in $P(\bw)$ correspond to partitions contained in $\lambda/\mu$.  
For example in Figure~\ref{fig:graph_poset_meet}, $\lambda=(10,7,7,3,3,3,3,3)$ and $\mu=(6,6,2,2,2,2,2)$.
Given a partition $\lambda=(\lambda_1,\ldots, \lambda_k)$, the number of partitions contained inside $\lambda$ is given by $D(\lambda) = \det(\binom{\lambda_j+1}{j-i+1})_{1\leq i,j\leq k}$, thus the number of maximal chains in $P(\bw)$ can be computed by $D(\lambda) -D(\mu)$.
\end{remark}

Next, we study the circuits of the order polytope $\calO(Q_{\bw})$.
Understanding this for arbitrary words $\bw$ is a challenge, therefore we instead restrict our attention in this section to the following set of words.

\begin{definition}\label{def:V}
Let $\mathcal{V}$ denote the subset of words which do not contain the substring $LRL$ or $RLR$.
\end{definition}

Theorem~\ref{thm.bijection} shows that for $\bw\in\mathcal{V}$, circuits in the vertices of $\mathcal{O}(Q_{\bw})$ have a combinatorial interpretation as the nonempty connected induced subgraphs of the graph $G(\bw)$.

\begin{lemma} Let $\bw \in \mathcal{V}$ be a generalized snake word of length $n$. 
The poset $Q_{\bw}$ has order $n+4$. 
\end{lemma}
\begin{proof}
A non-$\widehat{1}$ element of $\hatP$ is meet-irreducible if and only if it is not the minimum element in a square of $\hatP(\bw)$.
There are $n+1$ squares, thus
$|\Irr_\wedge(\hatP)| = 2n+6-1-(n+1) = n+4. $
\end{proof}

\begin{lemma}\label{lem.mindepset01} 
Let $\bw \in \mathcal{V}$ be a generalized snake word of length $n$. 
A circuit of the vertex set of $\calO(Q_{\bw})$ cannot contain the zero vector $\bv_\emptyset$ or $\bv_{Q_{\bw}}=(1,1,\ldots,1)$.
\end{lemma}

\begin{proof}
Observe that $\bv_\emptyset$ is the zero vector in $\R^{|Q_{\bw}|}$, so it cannot be in a minimal dependent set. Also, $\bv_{Q_\bw}$ is the vector of all ones in $\R^{|Q_{\bw}|}$, and in particular is the only vertex whose $|Q_{\bw}|$-th coordinate is nonzero, so it also cannot be in a minimal dependent set.
\end{proof}

\begin{theorem}\label{thm.bijection}
Let $\bw \in \mathcal{V}$ be a generalized snake word of length $n$. 
There exists a bijection $\Gamma:\calG(\bw) \rightarrow \calC(Q_{\bw})$ between the set $\calG(\bw)$ of nonempty connected induced subgraphs of $G(\bw)$ and the set $\calC(Q_{\bw})$ of circuits of the vertex set of the order polytope $\calO(Q_{\bw})$.
\end{theorem}

\begin{proof}
Let ${\bw}=w_0w_1\dots w_n$ and let $H\in \calG(\bw)$.
For a filter $A$ of $Q_{\bw}$, we consider $A$ to be a point in $J(Q_{\bw}) \cong \hatP$.
We say $A$ is \emph{compatible} with $H$ if $A$ is an element of an odd number of squares in $\{\Sq(w_i):w_i\in H\}$.
Define $\Gamma(H)$ to be the set of all $\bv_A$ such that $A$ is compatible with $H$.
See Figure~\ref{fig.embedH} for an illustration.

\begin{figure}[ht!]
\begin{center}
\begin{tikzpicture}[scale=.5]
\begin{scope}[xshift=0, yshift=0] 
	\pvx[](emp) at (2,20) {};
	\pvx[](v00) at (1,19) {};
	\pvx[label=left:{\tiny$\textcolor{red}{+}$}, color=red](v01) at (0,18) {};
	\pvx[](v02) at (2,18) {};
	\pvx[](v03) at (-1,17) {};
	\pvx[](v04) at (-2,16) {};
	\pvx[label=left:{\tiny$\textcolor{red}{-}$}, color=red](v05) at (-3,15) {};	
	\pvx[label=right:{\tiny$\textcolor{red}{-}$}, color=red](v06) at (0,14) {};
	\pvx[](v07) at (1,13) {};	
	\pvx[](v08) at (-2,12) {};
	\pvx[](v09) at (-3,11) {};	
	\pvx[label=left:{\tiny$\textcolor{red}{+}$}, color=red](v10) at (-4,10) {};
	\pvx[](v11) at (-5,9) {};
	\pvx[](v12) at (-2,8) {};
	\pvx[](v13) at (-1,7) {};
	\pvx[](v14) at (0,6) {};
	\pvx[](v15) at (1,5) {};
	\pvx[](v16) at (2,4) {};
	\pvx[](v17) at (-1,3) {};
	\pvx[](v18) at (-2,2) {};
	\pvx[](v19) at (0,0) {};
	\pvx[label=right:{\tiny$\textcolor{red}{-}$}, color=red](w0102) at (1,17) {};
	\pvx[](w0203) at (0,16) {};
	\pvx[label=right:{\tiny$\textcolor{red}{+}$}, color=red](w0204) at (-1,15) {};
	\pvx[](w0205) at (-2,14) {};
	\pvx[](w0506) at (-1,13) {};
	\pvx[label=right:{\tiny$\textcolor{red}{+}$}, color=red](w0507) at (0,12) {};
	\pvx[](w0708) at (-1,11) {};
	\pvx[](w0709) at (-2,10) {};
	\pvx[label=right:{\tiny$\textcolor{red}{-}$}, color=red](w0710) at (-3,9) {};
	\pvx[](w0711) at (-4,8) {};
	\pvx[](w1112) at (-3,7) {};
	\pvx[](w1113) at (-2,6) {};
	\pvx[](w1114) at (-1,5) {};
	\pvx[](w1115) at (0,4) {};
	\pvx[](w1116) at (1,3) {};	
	\pvx[](w1617) at (0,2) {};
	\pvx[](w1618) at (-1,1) {};
	
	\draw[line width=1pt] (v19)--(v18);
	\draw[line width=1pt] (v18)--(v15);
	\draw[line width=1pt] (v16)--(v10);
	\draw[line width=1pt] (v11)--(v06);
	\draw[line width=1pt] (v07)--(v04);
	\draw[line width=1pt] (v02)--(v00);
	\draw[line width=1pt] (v05)--(emp);
	\draw[line width=1pt] (w1618)--(v16);
	\draw[line width=1pt] (w1617)--(v17);	
	\draw[line width=1pt] (w1114)--(v14);
	\draw[line width=1pt] (w1113)--(v13);
	\draw[line width=1pt] (w1112)--(v12);	
	\draw[line width=1pt] (w1116)--(v11);
	\draw[line width=1pt] (w0711)--(v07);
	\draw[line width=1pt] (w0709)--(v09);	
	\draw[line width=1pt] (w0708)--(v08);	
	\draw[line width=1pt] (w0711)--(v07);	
	\draw[line width=1pt] (w0507)--(v05);
	\draw[line width=1pt] (w0205)--(v02);
	\draw[line width=1pt] (w0203)--(v03);
	\draw[line width=1pt] (w0102)--(v01);			
	
	\gvx[](c01) at (0,17) {};
	\gvx[](c02) at (-1,16) {};
	\gvx[](c03) at (-2,15) {};
	\gvx[](c04) at (-1,14) {};
	\gvx[](c06) at (-1,12) {};
	\gvx[](c07) at (-2,11) {};
	\gvx[](c08) at (-3,10) {};
	
	\node[] (d01) at (0,17.4) {\tiny $\textcolor{blue}{w_1}$};	
	\node[] (d02) at (-1.1,16.4) {\tiny $\textcolor{blue}{w_2}$};	
	\node[] (d03) at (-2.1,15.4) {\tiny $\textcolor{blue}{w_3}$};	
	
	\node[] (d04) at (-0.5,14) {\tiny $\textcolor{blue}{w_4}$};
	\node[] (d06) at (-0.5,12) {\tiny $\textcolor{blue}{w_6}$};

	\node[] (d07) at (-2.1,11.4) {\tiny $\textcolor{blue}{w_7}$};	
	\node[] (d08) at (-3.1,10.4) {\tiny $\textcolor{blue}{w_8}$};

	\draw[line width=1pt, color=blue] (c01)--(c02);
	\draw[line width=1pt, color=blue] (c02)--(c03);
	\draw[line width=1pt, color=blue] (c03)--(c04);
	\draw[line width=1pt, color=blue] (c04)--(c06);
	\draw[line width=1pt, color=blue] (c06)--(c07);
	\draw[line width=1pt, color=blue] (c07)--(c08);
	\draw[line width=1pt, color=blue] (c02)--(c04);
\end{scope}
\end{tikzpicture}
\end{center}
\caption{The connected induced subgraph $H$ of $G(\bw)$, depicted in blue, is embedded in $\mathrm{Hasse}(\hatP)$.  The elements in the circuit $\Gamma(H)$ are depicted in red.}
\label{fig.embedH}
\end{figure}
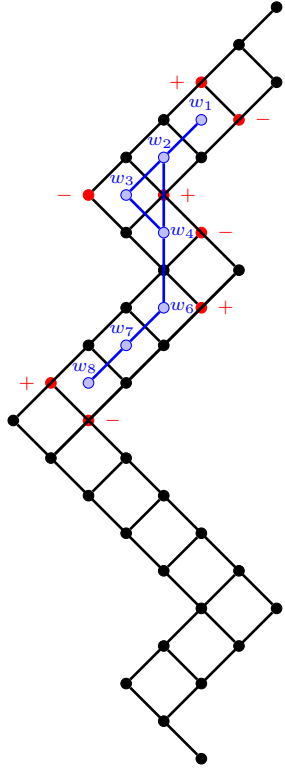


We shall show by induction on $|H|$ that $\Gamma(H)$ is a minimal dependent set of vertices of $\calO(Q_{\bw})$.
First, suppose $H=\{w_k\}$ is a single vertex of $G$ for some $k=0,\ldots, n$.  
Then for some filter $A$ of $Q_{\bw}$ and distinct incomparable elements $x,y\in Q_{\bw}$, we have
\[
\Gamma(H)= \{\bv_B:B\in \Sq(w_k)\} = \{\bv_A, \bv_{A\cup \{x\}}, \bv_{A\cup \{y\}}, \bv_{A\cup\{x,y\}} \} \, .
\]
As
\begin{align*}
\bv_{A\cup\{x\}}&= \bv_A +\be_x,\\
\bv_{A\cup\{y\}}&= \bv_A +\be_y,\\
\bv_{A\cup\{x,y\}}&= \bv_A +\be_x+\be_y,
\end{align*}
then
$\bv_A - \bv_{A\cup \{x\}} - \bv_{A\cup \{y\}}+ \bv_{A\cup\{x,y\}} =\mathbf{0},$
and $\Gamma(H)$ is a circuit.

Next, suppose $H=\{w_{i_1},\ldots, w_{i_k}\}$ is a connected induced subgraph of $G$, with the assumption that $i_1< \cdots < i_k$ and $k\geq 2$.
Observe that $H$ contains the path $w_{i_1},\ldots, w_{i_k}$.
We will inductively assign signs to the squares of $\hatP$ that contain vertices of $H$.
Start by defining $\sgn(\Sq(w_{i_1}))=1$. For $j=2,\ldots, k$,
\[
\sgn(\Sq(w_{i_{j}})) = 
\begin{cases}
\ \,\,\,\sgn(\Sq(w_{i_{j-1}})) \, , &\hbox{if $i_j-i_{j-1}=1$},\\
-\sgn(\Sq(w_{i_{j-1}})), &\hbox{if $i_j-i_{j-1}=2$} \, .
\end{cases} \]
We note that since $H$ is a connected subgraph, then $\Sq(w_{i_j}) \cap \Sq(w_{i_{j+1}})\neq \emptyset$.
If 
$$\Sq(w_{i_j}) = \left\{A_{i_j}, A_{i_j}\cup\{x_{i_j}\}, A_{i_j}\cup\{y_{i_j}\}, A_{i_j}\cup\{x_{i_j},y_{i_j}\}\right\},$$ 
then 
$$\sigma_{i_j}:=\bv_{A_{i_j}} - \bv_{A_{i_j}\cup\{x_{i_j}\}} - \bv_{A_{i_j}\cup\{y_{i_j}\}} + \bv_{A_{i_j}\cup\{x_{i_j},y_{i_j}\}}= \mathbf{0} $$
is an affine dependence relation on the vertices of $\Sq(w_{i_j})$.  Thus,

\begin{equation}\label{eq:11}
\sum_{j=1}^k \sgn(\Sq(w_{i_j}))\cdot \sigma_{i_j} = \mathbf{0}. \, 
\end{equation}

Because of the definition of the $\sgn$ function, the terms which occur in the left hand side of this expression with nonzero coefficient are indexed precisely by the elements of $\hatP$ which are compatible with $H$, i.e., $\Gamma(H)$.
Thus, $\Gamma(H)$ is a dependent set with an affine dependence relation (\ref{eq:11}).

Having shown $\Gamma(H)$ is dependent, it remains to show that $\Gamma(H)$ is also minimal, i.e., a circuit.
We will use induction on $|H|=k$.
If $k=1$, then it is straightforward to verify that $\Sq(w_{i_1})$ is a circuit.
Assume that $k>1$.
We consider three cases.
For the first case, suppose that $i_{k-1}+1=i_k$ and $(w_{i_{k-2}},w_{i_{k}})$ is not an edge in $H$.
Thus, it follows that for some filter $A$ and elements $a,b,c$, we have
\[
\Sq(w_{i_{k-1}})=\{A,A\cup\{b\},A\cup\{a\},A\cup \{a,b\}\}
\]
and
\[
\Sq(w_{i_{k}})=\{A\cup\{b\},A\cup\{a,b\},A\cup\{b,c\},A\cup \{a,b,c\}\}\, .
\]

\begin{center}
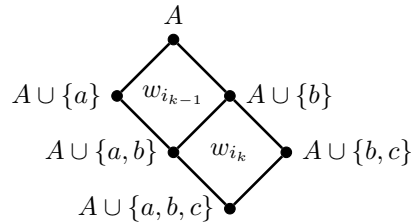
\begin{figure}[h!]
\begin{tikzpicture}[scale=0.75]\	
    \pvx[label=above:\small{$A$}](v) at (0,20) {};
	\pvx[label=left:\small{$A\cup \{a\}$}](va) at (-1,19) {};
	\pvx[label=right:\small{$A\cup \{b\}$}](vb) at (1,19) {};
	\pvx[label=left:\small{$A\cup \{a,b\}$}](vab) at (0,18) {};
	\pvx[label=right:\small{$A\cup \{b,c\}$}](vbc) at (2,18) {};
	\pvx[label=left:\small{$A\cup \{a,b,c\}$}](vabc) at (1,17) {};

	\draw[line width=1pt] (v)--(vb);
	\draw[line width=1pt] (v)--(va);		
	\draw[line width=1pt] (va)--(vab);
	\draw[line width=1pt] (vb)--(vab);
	\draw[line width=1pt] (vb)--(vbc);
	\draw[line width=1pt] (vab)--(vabc);
	\draw[line width=1pt] (vbc)--(vabc);
	
	\node[](wik-1) at (0,19) {\small{$w_{i_{k-1}}$}};
	\node[](wik) at (1,18) {\small{$w_{i_{k}}$}};
\end{tikzpicture}
\caption{The first case, with $i_{k-1}+1 = i_k$ and $(w_{i_{k-2}},w_{i_{k}})$ is not an edge in $H$.}
\end{figure}
\end{center} 

In $\Gamma(H)$, the only vectors supported on the $c$-coordinate are $\bv_{A\cup\{b,c\}}$ and $\bv_{A\cup\{a,b,c\}}$.
If we restrict the vectors in $\{\bv_C:C\in \Gamma(H)\}$ to the coordinates in $A\cup\{a,b\}$, then we obtain the vectors $\{\bv_C:C\in \Gamma(H\setminus w_{i_k})\}$.
By induction, these vectors form a circuit with a unique minimal dependence where the coefficients of $\bv_{A\cup\{b\}}$ and $\bv_{A\cup\{a,b\}}$ are equal and opposite in sign.
Thus, this is the only potential dependence (up to scaling) for $\{\bv_C:C\in \Gamma(H)\}$, where the coefficients of $\bv_{A\cup\{b\}}$ and $\bv_{A\cup\{a,b\}}$  become the coefficients of $\bv_{A\cup\{b,c\}}$ and $\bv_{A\cup\{a,b,c\}}$ respectively.
It is immediate that this choice of coefficients is a dependence with all non-zero coefficients, and thus $\Gamma(H)$ is a circuit.

For the second case, suppose that $i_{k-1}+2=i_k$, i.e., the last edge in $H$ is the vertical edge of a triangle in $G$.
Thus, it follows that for some filter $A$ and elements $a,b,c,d$, we have
\[
\Sq(w_{i_{k-1}})=\{A,A\cup\{b\},A\cup\{a\},A\cup \{a,b\}\}
\]
and
\[
\Sq(w_{i_{k}})=\{A\cup\{a,b\},A\cup\{a,b,c\},A\cup\{a,b,d\},A\cup \{a,b,c,d\}\}\, .
\]

\begin{center}
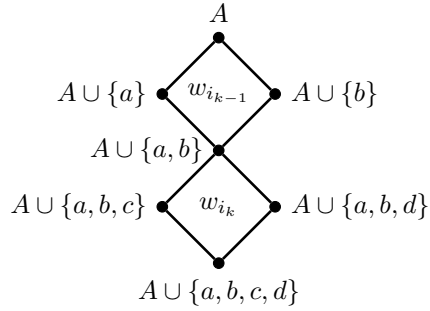
\begin{figure}[H]
\begin{tikzpicture}[scale=0.75]\	
    \pvx[label=above:\small{$A$}](v) at (0,20) {};
	\pvx[label=left:\small{$A\cup \{a\}$}](va) at (-1,19) {};
	\pvx[label=right:\small{$A\cup \{b\}$}](vb) at (1,19) {};
	\pvx[label=left:\small{$A\cup \{a,b\}$}](vab) at (0,18) {};
	\pvx[label=right:\small{$A\cup \{a,b,d\}$}](vabd) at (1,17) {};
	\pvx[label=left:\small{$A\cup \{a,b,c\}$}](vabc) at (-1,17) {};
	\pvx[label=below:\small{$A\cup \{a,b,c,d\}$}](vabcd) at (0,16) {};

	\draw[line width=1pt] (v)--(vb);
	\draw[line width=1pt] (v)--(va);		
	\draw[line width=1pt] (va)--(vab);
	\draw[line width=1pt] (vb)--(vab);
	\draw[line width=1pt] (vab)--(vabc);
	\draw[line width=1pt] (vab)--(vabd);
	\draw[line width=1pt] (vabc)--(vabcd);
	\draw[line width=1pt] (vabd)--(vabcd);
	
	\node[](wik-1) at (0,19) {\small{$w_{i_{k-1}}$}};
	\node[](wik) at (0,17) {\small{$w_{i_{k}}$}};
\end{tikzpicture}
\caption{Case two, with $i_{k-1}+2=i_k$.}
\end{figure}
\end{center} 
In $\Gamma(H)$, the only vectors supported on the $d$-coordinate are $\bv_{A\cup\{a,b,c,d\}}$ and $\bv_{A\cup\{a,b,d\}}$.
Thus, in any dependence for $\Gamma(H)$, these two vectors have coefficients that are equal in magnitude and opposite in sign.
The only vectors supported on the $c$-coordinate are $\bv_{A\cup\{a,b,c\}}$ and $\bv_{A\cup\{a,b,c,d\}}$.
Thus, these vectors have coefficients that are equal in magnitude and opposite in sign in any dependence for $\Gamma(H)$.
Summing these three vectors with these equal and opposite coefficients yields the vector $\bv_{A\cup\{a,b\}}$ (scaled by the same coefficient).
Thus, any dependence on $\Gamma(H)$ arises from a dependence on $\Gamma(H\setminus w_{i_k})$.
By induction, this set is a circuit with a unique minimal dependence.
This unique minimal dependence induces a unique minimal dependence on $\Gamma(H)$, for which every coefficient is non-zero.

For the third case, suppose that $i_{k-1}+1=i_k$ and $(w_{i_{k-2}},w_{i_{k}})$ is an edge in $H$, i.e., that $w_{i_{k-2}}$, $w_{i_{k-1}}$, and $w_{i_{k}}$ form a triangle in $H$.
Thus, it follows that for some filter $A$ and elements $a,b,c,d$, we have
\[
\Sq(w_{i_{k-2}})=\{A,A\cup\{b\},A\cup\{a\},A\cup \{a,b\}\}
\]
and
\[
\Sq(w_{i_{k-1}})=\{A\cup\{b\},A\cup\{a,b\},A\cup \{b,c\},A\cup\{a,b,c\}\}
\]
and
\[
\Sq(w_{i_{k}})=\{A\cup\{a,b\},A\cup\{a,b,c\},A\cup\{a,b,d\},A\cup \{a,b,c,d\}\}\, .
\]

\begin{center}
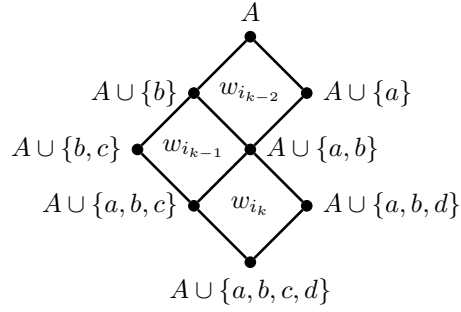
\begin{figure}[H]
\begin{tikzpicture}[scale=0.75]\	
    \pvx[label=above:\small{$A$}](v) at (0,20) {};
	\pvx[label=right:\small{$A\cup \{a\}$}](va) at (1,19) {};
	\pvx[label=left:\small{$A\cup \{b\}$}](vb) at (-1,19) {};
	\pvx[label=right:\small{$A\cup \{a,b\}$}](vab) at (0,18) {};	\pvx[label=left:\small{$A\cup \{b,c\}$}](vbc) at (-2,18) {};
	\pvx[label=right:\small{$A\cup \{a,b,d\}$}](vabd) at (1,17) {};
	\pvx[label=left:\small{$A\cup \{a,b,c\}$}](vabc) at (-1,17) {};
	\pvx[label=below:\small{$A\cup \{a,b,c,d\}$}](vabcd) at (0,16) {};

	\draw[line width=1pt] (v)--(vb);
	\draw[line width=1pt] (v)--(va);		
	\draw[line width=1pt] (va)--(vab);
	\draw[line width=1pt] (vb)--(vab);
	\draw[line width=1pt] (vab)--(vabc);
	\draw[line width=1pt] (vab)--(vabd);
	\draw[line width=1pt] (vabc)--(vabcd);
	\draw[line width=1pt] (vabd)--(vabcd);
	\draw[line width=1pt] (vb)--(vbc);
	\draw[line width=1pt] (vbc)--(vabc);
	
	\node[](wik-2) at (0,19) {\small{$w_{i_{k-2}}$}};	
	\node[](wik-1) at (-1,18) {\small{$w_{i_{k-1}}$}};
	\node[](wik) at (0,17) {\small{$w_{i_{k}}$}};
\end{tikzpicture}
\caption{Case three, with $i_{k-1}+1=i_k$ and $(w_{i_{k-2}},w_{i_{k}})$ is an edge in $H$.}
\end{figure}
\end{center} 
In $\Gamma(H)$, the only vectors supported on the $d$-coordinate are $\bv_{A\cup\{a,b,c,d\}}$ and $\bv_{A\cup\{a,b,d\}}$.
Thus, in any dependence for $\Gamma(H)$, these two vectors have coefficients that are equal in magnitude and opposite in sign.
Because $A\cup\{a,b,c\}$ is not compatible with $H$, the only non-zero vectors supported on the $c$-coordinate are $\bv_{A\cup\{b,c\}}$and $\bv_{A\cup\{a,b,c,d\}}$.
Thus, these vectors have coefficients that are equal in magnitude and opposite in sign in any dependence for $\Gamma(H)$.
Summing these three vectors with these equal and opposite coefficients yields the vector $\bv_{A\cup\{b\}}$ (scaled by the same coefficient).
Thus, any dependence on $\Gamma(H)$ arises from a dependence on $\Gamma(H\setminus \{w_{i_{k-1}},w_{i_{k}}\})$.
By induction, this set is a circuit with a unique minimal dependence.
This unique minimal dependence induces a unique minimal dependence on $\Gamma(H)$, for which every coefficient is non-zero.
\color{black}

Having established that $\Gamma$ is well-defined, we next show that $\Gamma$ is injective.
Suppose $\Gamma(H) = \Gamma(K)$ but $H\neq K$.
Since $H$ and $K$ are induced subgraphs, then this means $V(H)\neq V(K)$.
Without loss of generality, suppose $w_m\in H$ but $w_m\notin K$.  
Since $H$ and $K$ are each connected, then $w_m$ must occur either at one of the ends of the main path ($w_{i_1},\ldots, w_{i_k}$) of $H$, or is a corner of a triangle in $H$, where $w_{i}$ is a corner if $w_{i-1}$ and $w_{i+1}\in H$. 
In either case, this implies $\Gamma(H)\backslash \Gamma(K)\neq \emptyset$, a contradiction.

To see that $\Gamma$ surjective, we induct on the length of $\bw$.
If $\bw=w_0w_1\cdots w_n$, we define $\ell(\bw)=n$.
Consider $\ell(\bw)=0$ so that $\bw=\varepsilon$. See Figure~\ref{fig:base_case}.

\begin{center}
\begin{figure}[H]
\begin{tikzpicture}[scale=.5]
\begin{scope}[xshift=0, yshift=0] 
	\pvx[label=left:\tiny$\emptyset$](emp) at (2,20) {};
	\pvx[label=left:\tiny{$\langle0\rangle$}](v00) at (1,19) {};
	\pvx[label=left:\tiny{$\langle1\rangle$}](v01) at (0,18) {};
	\pvx[label=right:\tiny{$\langle2\rangle$}](v02) at (2,18) {};
	\pvx[label=right:\tiny{$\langle3\rangle$}](v03) at (0,16) {};
	\pvx[label=right:\tiny{$\langle1,2\rangle$}](w0102) at (1,17) {};

	\draw[line width=1pt] (v02)--(v00);
	\draw[line width=1pt] (w0102)--(v01);
	\draw[line width=1pt] (v02)--(v03);
	\draw[line width=1pt] (v01)--(emp);			
	
	\node[](c01) at (1,18) {$\varepsilon$};	
	\node[] at (-3.5,18) {$\hatP=\hatP(\varepsilon):$};	
\end{scope}
\begin{scope}[xshift=350, yshift=140] 
	\pvx[label=above:$0$](v00) at (2,14) {};
	\pvx[label=left:$1$](v01) at (1,13) {};
	\pvx[label=right:$2$](v02) at (3,13) {};
	\pvx[label=below:$3$](v03) at (2,12) {};

	\draw[line width=1pt] (v03)--(v02)--(v00)--(v01)--(v03);	
	\node[] at (-2.75,13) {$Q_{\varepsilon}=\Irr_\wedge(J):$};
\end{scope}
\end{tikzpicture}
\caption{The base case with $\bw = \varepsilon$.}\label{fig:base_case}
\end{figure}
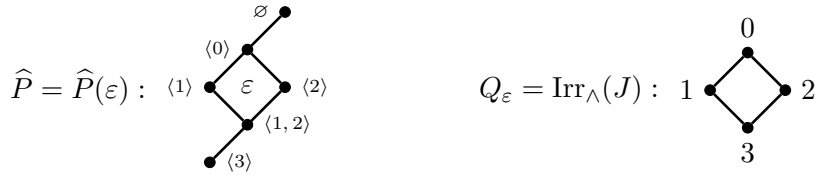
\end{center}

The only circuit arises from $\Sq(\varepsilon)$:
$$\bv_{\langle 0 \rangle} - \bv_{\langle 1 \rangle} - \bv_{\langle 2 \rangle} + \bv_{\langle 1,2 \rangle} = (1,0,0,0) - (1,1,0,0) - (1,0,1,0) + (1,1,1,0) = (0,0,0,0). $$

Suppose the map $\Gamma:\calG(\bw) \rightarrow \calC(Q_{\bw})$ is surjective for all $\bw$ where $\ell(\bw) \leq n-1$.  

Now, let $\bu = w_0\cdots w_{n-1} \in \calV$ and $\bw= \bu w_n \in \calV$.
Suppose $\Sq(w_n)=\{ A,B,C,D \}$ so that $A\prec B$, $A\prec C$, $B\prec D$, $C\prec D$, and $C$ and $B$ are incomparable.
Without loss of generality, there are two cases to consider; $w_{n-1} = w_n = R$, or $w_{n-1} = L$ and $w_{n}= R$. These cases are shown in Figure \ref{fig:surjInductionCases}.

\begin{center}
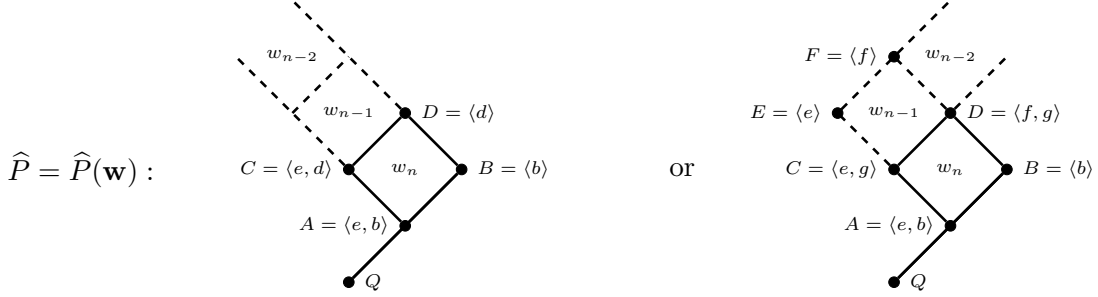
\begin{figure}[H]
\begin{tikzpicture}[scale=.75]
\begin{scope}[xshift=0, yshift=0] 
	\pvx[label=right:\tiny{$D=\langle d\rangle$}](v00) at (1,19) {};
	\pvx[label=left:\tiny{$C=\langle e,d\rangle$}](v01) at (0,18) {};
	\pvx[label=right:\tiny{$B=\langle b\rangle$}](v02) at (2,18) {};
	\pvx[label=left:\tiny{$A=\langle e,b\rangle$}](w0102) at (1,17) {};
	\pvx[label=right:\tiny$Q$](v03) at (0,16) {};

	\draw[line width=1pt] (v02)--(v00);
	\draw[line width=1pt] (w0102)--(v01);
	\draw[line width=1pt] (v02)--(v03);
	\draw[line width=1pt] (v02)--(v03);
	\draw[line width=1pt] (v01)--(v00);			
	\draw[dashed, line width=1pt] (v00)--(-1,21);
	\draw[dashed, line width=1pt] (v01)--(-2,20);
	\draw[dashed, line width=1pt] (-1,19)--(0,20);
	
	\node[](c01) at (1,18) {\tiny$w_n$};	
	\node[](c02) at (0,19) {\tiny$w_{n-1}$};	
	\node[](c03) at (-1,20) {\tiny$w_{n-2}$};		
	\node[] at (-4.75,18) {$\hatP =\hatP(\bw):$};	
\end{scope}
\begin{scope}[xshift=275, yshift=0] 
	\pvx[label=left:\tiny{$F=\langle f\rangle$}] at (0,20) {};
	\pvx[label=left:\tiny{$E=\langle e\rangle$}] at (-1,19) {};
	\pvx[label=right:\tiny{$D=\langle f,g\rangle$}](v00) at (1,19) {};
	\pvx[label=left:\tiny{$C=\langle e,g\rangle$}](v01) at (0,18) {};
	\pvx[label=right:\tiny{$B=\langle b\rangle$}](v02) at (2,18) {};
	\pvx[label=left:\tiny{$A=\langle e,b\rangle$}](w0102) at (1,17) {};
	\pvx[label=right:\tiny$Q$](v03) at (0,16) {};

	\draw[line width=1pt] (v02)--(v00);
	\draw[line width=1pt] (w0102)--(v01);
	\draw[line width=1pt] (v02)--(v03);
	\draw[line width=1pt] (v02)--(v03);
	\draw[line width=1pt] (v01)--(v00);			
	\draw[dashed, line width=1pt] (v00)--(0,20);
	\draw[dashed, line width=1pt] (v01)--(-1,19)--(1,21);
	\draw[dashed, line width=1pt] (v00)--(2,20);
	
	\node[](c01) at (1,18) {\tiny$w_n$};	
	\node[](c02) at (0,19) {\tiny$w_{n-1}$};
	\node[](c03) at (1,20) {\tiny$w_{n-2}$};	
	\node[] at (-3.75,18) {or};	
\end{scope}
\end{tikzpicture}
\caption{The two cases in the induction step.}
\label{fig:surjInductionCases}
\end{figure}
\end{center}

In both cases, $A$ and $C$ are not meet-irreducible, while $B$ is meet-irreducible.
In the case $w_n=w_{n-1}$, $D$ is meet-irreducible,  so we have $B=\langle b\rangle$, $D=\langle d \rangle$, $A=\langle e,b\rangle$ and $C=\langle e,d\rangle$ for some $b,d,e\in Q$.
In the case $w_n \neq w_{n-1}$, $D$ is not meet-irreducible.
If $\Sq(w_{n-1})= \{C,D,E,F \}$ with $E\prec F$, then we have $B=\langle b\rangle$, $E=\langle e\rangle$, $F=\langle f\rangle$, $A=\langle e,b\rangle$, $C= \langle e,d\rangle$ and $D=\langle f,d\rangle$ for some $b,d,e,f\in Q$. See the picture on the left in Figure \ref{fig:surjInductionCases}. 

Let $T =\Irr_\wedge(\hatP(\bu))$.
If $\gamma \in \calC(Q_{\bw})$ is a circuit that does not contain $\bv_A$ or $\bv_B$, then $\gamma$ is a circuit of the vertex set of $\calO(T)$.
By the induction hypothesis, there exists a connected induced subgraph $H\in\calG(\bu)\subset \calG(\bw)$ such that $\Gamma(H) = \gamma$.

We shall show that for any circuit $\gamma$ that contains $\bv_A$ or $\bv_B$, there exists a connected induced subgraph $H\subseteq G$ such that $\Gamma(H) = \gamma$.
First, observe that $\bv_A$ or $\bv_B$ are the only vertices in $\calO(Q_{\bw})$ whose $b$-th coordinate is nonzero, aside from $\bv_{Q_{\bw}}$ (which by Lemma~\ref{lem.mindepset01} we know cannot be contained in any circuit).
Thus, if one of these vectors is in $\gamma$, then they must both be in $\gamma$.
This also implies that any dependency relation involving $\bv_A$ and $\bv_B$ must be of the form $\alpha(\bv_A - \bv_B) + R =\mathbf{0}$ for some nonzero $\alpha\in \R$ and where $R$ is a linear combination of the elements of $\gamma \smallsetminus \{\bv_A, \bv_B\}$.
Second, since $\bv_A\in \gamma$ and $A=\langle e,b\rangle$, then $\gamma$ must contain a vertex of the form $\bv_{\langle e \rangle}$ or $\bv_{\langle e,x\rangle}$ for some $x\in Q_{\bw}$.

Consider the case $w_n=w_{n-1} = R$, so that $C=\langle e,d\rangle$ and $D=\langle d \rangle$. If $\bv_C\in \gamma$, then $\bv_D\in \gamma$ as well, since these are the only vertices in $\calO(T)$ whose $d$-th coordinate is nonzero.
Then by minimality $\gamma$ must be the circuit $\{A,B,C,D \}$ as
$$\bv_A -\bv_B -\bv_C +\bv_D = \mathbf{0}. $$
In this case $H=\{w_n\}$ and $\Gamma(H) = \gamma$.

Otherwise, $\gamma$ does not contain $\bv_C$ nor $\bv_D$,
so suppose $\gamma$ gives rise to a minimal dependence relation of the form
$$\alpha\bv_A - \alpha\bv_B  + R = \mathbf{0}$$
for some nonzero $\alpha\in \R$, and $R$ is a nonzero linear combination of vertices of $\calO(T)$ that does not involve $\bv_A$ or $\bv_B$.
As $\bv_A - \bv_B = \bv_C-\bv_D,$
then
$$ \alpha\bv_C - \alpha\bv_D  + R = \mathbf{0}$$
is another dependence relation which is minimal, because the initial dependence relation was minimal.
This new dependence relation consists of a set of vertices $\gamma'\subseteq\calO(T)$, so by the induction hypothesis, there is a connected induced subgraph $H'\in \calG(\bu)$ such that $\Gamma(H') = \gamma'$.  
From this, it follows that the connected induced subgraph $H= H'\cup \{w_n\} \in \calG(\bw)$ satisfies $\Gamma(H) = \gamma$.

Next, we consider the case  $w_{n-1} = L$ and $w_n = R$.
We have $C=\langle e,g\rangle$, $D=\langle f,g\rangle$, $E=\langle e\rangle $ and $F=\langle f\rangle$. See the picture on the right in Figure \ref{fig:surjInductionCases}. 
We consider four cases; in each we will find an $H$ with $\Gamma(H)=\gamma$, concluding the proof.

\begin{enumerate}
    \item[(i)] Case $\bv_C,\bv_D\in \gamma$: This implies $\gamma$ contains $\{A,B,C,D\}$, which is a circuit.
Thus, it must be that $\gamma=\{A,B,C,D\}$ and hence $\Gamma(\{w_n\})=\gamma$.
    \item[(ii)] Case $\bv_C\in\gamma$, $\bv_D\notin \gamma$: 
Since $\bv_A=\bv_F+\be_g+\be_e+\be_b$, $\bv_B=\bv_F+\be_g+\be_b$, and $\bv_C=\bv_F+\be_g+\be_e$, then $\gamma$ must contain $\bv_F$ since $\bv_F$ is the only other vertex in $\calO(T)$ whose $f$-th entry is nonzero.
In this case, $\gamma$ gives rise to a minimal dependence relation of the form
$$\alpha(\bv_A -\bv_B-\bv_C+\bv_F) +R =\mathbf{0} $$
for some nonzero $\alpha\in\R$, and $R$ is a nonzero linear combination of vertices of $\calO(T)$ that does not involve $\bv_A,\bv_B,\bv_C,\bv_D,\bv_E$, or $\bv_F$.
As $\bv_A -\bv_B-\bv_C = -\bv_D$, then
$$\alpha(\bv_F-\bv_D) + R = \mathbf{0} $$
is another minimal dependence relation consisting of vertices $\gamma'\subseteq \calO(T)$ so there exists $H'\in\calG(\bu)$ such that $\Gamma(H')=\gamma'$, and it follows that $\Gamma(H'\cup \{w_n\}) =\gamma$.
    \item[(iii)] Case $\bv_C\notin \gamma$, $\bv_F \in \gamma$: By a similar analysis as above, it must be that $\gamma$ is the circuit $\{A,B,E,F\}$, since \[\bv_A - \bv_B - \bv_E + \bv_F =\mathbf{0},\]
and $H$ is the subgraph induced on $\{w_{n-1},w_n\}$, where $\Gamma(H) = \gamma$.
    \item[(iv)] Case $\bv_C,\bv_F\notin \gamma$: Since $\bv_C\notin\gamma$, we must have $\bv_E\in \gamma$ as $\bv_A = \bv_B+\be_e$ and $\bv_E$ is the only other vertex in $\calO(T)$ whose $e$-th coordinate is nonzero.
Since $A,B,E < F$, then the $f$-th entry of each of $\bv_A,\bv_B,\bv_E$ are nonzero.
This implies that $\bv_D$ is in $\gamma$ since the only other vertex in $\calO(T)$ whose $f$-th entry is nonzero is $\bv_F$.

Thus, $\gamma$ gives rise to a minimal dependence relation of the form
$$\alpha(\bv_A - \bv_B + \bv_D -\bv_E) + R =\mathbf{0} $$
for some nonzero $\alpha\in \R$, and $R$ is a nonzero linear combination of vertices of $\calO(T)$ that does not involve $\bv_A,\bv_B,\bv_C,\bv_D,\bv_E$, or $\bv_F$.
As $\bv_A -\bv_B+\bv_D=\bv_C$, then
$$\alpha(\bv_C -\bv_E) + R =\mathbf{0} $$
is another minimal dependence relation consisting of vertices $\gamma'\subseteq \calO(T)$ so there exists $H'\in \calG(\bu)$ such that $\Gamma(H')=\gamma'$, and $\Gamma(H'\cup\{w_n\}) = \gamma$.

\end{enumerate}
\end{proof}

\begin{remark}
Theorem~\ref{thm.bijection} does not hold for a generalized snake word $\bw$ outside of $\mathcal{V}$.  Computational evidence suggests that the size of $\calG(\bw)$ is an upper bound for the number of circuits of $\calO(Q_{\bw})$.
\end{remark}

Next, we obtain a number of corollaries about the structure of the circuits in the vertex set of $\calO(Q_{\bw})$.

\begin{corollary}\label{cor:circuit}
Let $\bw\in \calV$.  A circuit $Z$ with partition $(Z_{+}, Z_{-})$ in the vertex set of $\calO(Q_{\bw})$ has an affine dependence relation of the form \[\sum\limits_{j \in Z_+}  \mathbf{v}_j = \sum\limits_{j \in Z_-} \mathbf{v}_j.\]  
In particular, $|{Z_-}|=|{Z_+}|$.
\end{corollary}

\begin{proof}
By proof of Theorem~\ref{thm.bijection}, a circuit $Z=\Gamma(H)$ for some nonempty connected induced subgraph $H$ of $G(\bw)$.  Moreover, an affine dependence relation for $Z$ is given in (\ref{eq:11}), where by construction every vertex of $Z$ appears with coefficient $\pm1$. This shows that there is an affine dependence relation for $Z$ as in the statement of the corollary.  Furthermore, we conclude that $|{Z_-}|=|{Z_+}|$, because the dependence is affine. 
\end{proof}

\begin{corollary}
Let $H=\{w_{i_1},\ldots, w_{i_k}\}$ be a connected induced subgraph of $G$ induced by the subword $w_{i_1}\cdots w_{i_k}$ of $\bw = w_0\cdots w_{n}\in\calV$ such that $i_1<\cdots <i_k$.
Suppose $H'=H \cup\{w_{i_j}\}$ is a connected induced subgraph of $G$ such that $i_k < i_j$.
Then
\begin{enumerate}
\item[(a)] If $w_{i_j} = w_{i_k}$, then $|\Gamma(H')|=|\Gamma(H)|$.
\item[(b)] If $w_{i_j} \neq w_{i_k}$, then $|\Gamma(H')|=|\Gamma(H)|+2$.
\end{enumerate}
In the case where $H = \{\varepsilon\}$, $|\Gamma(H')|=|\Gamma(H)|$. 
Thus, the smallest circuits in the vertex set of $\calO(Q_{\bw})$ have four vertices.
The largest circuits have $4+2t$ vertices where $t$ is the number of turns (an occurrence of $LLR$ or $RRL$) in $\bw$.
\end{corollary}

Using the bijection of Theorem~\ref{thm.bijection}, we can recursively compute the number of circuits in the vertex set of $\calO(Q_{\bw})$.

\begin{corollary}
Let $\bu=w_0\cdots w_{n-1}\in\calV$ and $\bw=\bu w_n\in\calV$.
Let $N_{k}$ be the number of connected induced subgraphs of $G(\bu)$ that contain $w_k$ but not $w_{k+1}$.
Then $|\calG(\varepsilon)| = 1$, $|\calG(\varepsilon w_1)| = 3$, and
\begin{enumerate}
\item[(a)] If $w_n = w_{n-1}$, then $|\calG(\bw)|=|\calG(\bu)|+N_{n-1}+1$.
\item[(b)] If $w_n \neq w_{n-1}$, then $|\calG(\bw)|=|\calG(\bu)|+N_{n-1}+N_{n-2}+1$.
\end{enumerate}
\end{corollary}

\begin{proof}\
\begin{enumerate}
\item[(a)] If $w_n=w_{n-1}$, then $\deg_{G(\bw)} w_n =1$.  
Thus for any connected induced subgraph $H$ of $G(\bw)$ that contains $w_n$, the connected induced subgraph $H'=H\backslash \{w_n\}$ contains $w_{n-1}$.
From the proof of Theorem~\ref{thm.bijection}, every circuit supported on the squares corresponding to $H$ has a corresponding circuit supported on the squares corresponding to $H'$.
The claim follows as $N_{n-1}$ counts the subgraphs of $G(\bu)$ which contain $w_{n-1}$ and $1$ counts the subgraph $\{w_n\}$. 
\item[(b)] 
If $w_n=w_{n-1}$, then $\deg_{G(\bw)} w_n =2$. 
Thus for any connected induced subgraph $H$ of $G(\bw)$ that contains $w_n$, the connected induced subgraph $H'=H\backslash \{w_n\}$ contains at least one of $w_{n-2}$ and $w_{n-1}$.
From the proof of Theorem~\ref{thm.bijection}, every circuit supported on the squares corresponding to $H$ has a corresponding circuit supported on the squares corresponding to $H'$.
The claim follows as $N_{n-2}$ is the number of connected induced subgraphs of $G(\bu)$ which contain $w_{n-2}$ but not $w_{n-1}$, $N_{n-1}$ is the number of connected induced subgraphs of $G(\bu)$ which contain $w_{n-1}$, and $1$ counts the subgraph $\{w_n\}$. 
\end{enumerate}
\end{proof}

\begin{remark}
When $\bw =  \varepsilon RRLLRRLL\ldots$, the poset $Q_{\bw}=P(\varepsilon RLRLRL\ldots) = S_k$ is the snake poset.  
The number of circuits of the order polytope of the snake poset is equal to the number of nonempty connected induced subgraphs of the graph $TS_{2k+1}$, defined as follows.
For odd $n$, let $TS_n$ denote the graph on $n$ vertices formed by taking the the path graph on $n$ vertices $P_{n}$ and adding edges $(2i-1, 2i+1)$ for $i = 1, \dots, (n-1)/2$.
The graph $TS_n$ is called a \emph{triangular snake graph}; see~\cite{GallianSurvey} and the references therein for additional information about triangular snakes.
\end{remark}

The properties of circuits imply the following results regarding triangulations of $\cO(Q_{\bw})$. 

\begin{lemma}\label{lem:unimodular}
Let $\bw\in\calV$.  If two triangulations of the polytope $\cO(Q_{\bw})$ are connected by a flip, then they have the same number of simplices.
\end{lemma}

\begin{proof}
 Let $\mathcal{T}_1,\mathcal{T}_2$ be a pair of triangulations of $\cO(Q_{\bw})$ that differ by a flip at circuit $Z$.  Then $\mathcal{T}_1,\mathcal{T}_2$ are related as in Theorem~\ref{thm:flip}. By Corollary~\ref{cor:circuit}, we have  $|{Z_-}|=|{Z_+}|$, which implies that the two triangulations $\mathcal{T}_Z^-, \mathcal{T}_Z^+$ of $Z$ have the same number of simplices. Therefore, because by definition the link $\mathrm{L}$ is disjoint from both $\mathcal{T}_Z^-, \mathcal{T}_Z^+$, we see that $\mathcal{T}_1=\mathcal{T}_Z^+\ast \mathrm{L}$ and $\mathcal{T}_2=\mathcal{T}_Z^-\ast \mathrm{L}$ also have the same number of simplices.  
 This shows that the number of simplices in a triangulation does not change when performing a flip. 
\end{proof}


\begin{theorem}\label{thm:unimodular}
For $\bw\in\calV$, every vertex of the secondary polytope of $\cO(Q_{\bw})$ is a unimodular triangulation.
Thus, every triangulation of $\cO(Q_{\bw})$ is unimodular.
\end{theorem}

\begin{proof}
The vertices of a secondary polytope of $\cO(Q_{\bw})$ correspond to regular triangulations of $\cO(Q_{\bw})$, see \cite[Theorem 5.1.9]{DeLoeraRambauSantos}.  Moreover, all regular triangulations are connected by flips \cite[Theorem 5.3.1]{DeLoeraRambauSantos}. 
In particular, because the canonical triangulation of $\cO(Q_{\bw})$ is regular and unimodular, then Lemma~\ref{lem:unimodular} implies that all regular triangulations of  $\cO(Q_{\bw})$ are also unimodular. 
Finally, it is a straightforward exercise to show that for a lattice polytope $P$, the following conditions are equivalent: all full-dimensional simplices are unimodular; all triangulations are unimodular; all regular triangulations are unimodular; all placing triangulations are unimodular.
Hence, all triangulations are unimodular.
\end{proof}


\section{Flips and a twist action on triangulations} \label{sec:flips_twists}

In this section we will take a deeper look at the 1-skeleton of the secondary polytope. Starting from the canonical triangulation of $\calO(Q_\bw)$ we will see that for a length $k$ word there are exactly $k+1$ flips, where a single flip corresponds to a local move along an edge in the flip graph. As a consequence, we fully determine the flip graph of regular triangulations in the special case of the ladder. 
We will also introduce the notion of twists which act globally by inducing automorphisms on the flip graph.

\subsection{Theorems regarding twists, flips, and triangulations}

Using the notation from Section~\ref{sec:circuits}, let $\bw$ be a generalized snake word in $\mathcal{V}$ and consider the associated poset $Q_{\bw}$.
In this section, our goal is to prove four theorems about flips of regular triangulations for $\cO(Q_{\bw})$.
We state the four theorems below; all undefined terms and proofs will be given in later subsections.
First, we classify the flips that can be made from the canonical triangulation of $\cO(Q_{\bw})$.

\begin{theorem}\label{thm:flipsfromcanonical}
Let $\bw\in \mathcal{V}$ have length $k$.
The canonical triangulation of $\cO(Q_{\bw})$ admits exactly $k+1$ flips.
\end{theorem}

As an application, we determine the flip graph of regular triangulations for the special case of a ladder.
When $\bw=\varepsilon L^{n-1}$, $\hatP\backslash\{\hat0,\hat1\}$ is the product of a $(n+1)$-chain and a $2$-chain.  
Thus the next result is a rephrasing of the well-known result that the secondary polytope of the Cartesian product of an $n$-simplex and $1$-simplex is an $n$-dimensional permutahedron~\cite[Section 16.7.1]{GOT}.

\begin{theorem}\label{thm.Cayleygraph}
Let $\mathbf{w}= \varepsilon L^{n-1}$, and $Q_\mathbf{w}= \mathrm{Irr}_\wedge (\hatP(\mathbf{w}))$.
The flip graph of triangulations of $\calO(Q_\mathbf{w})$ is the Cayley graph of the symmetric group $\mathfrak{S}_{n+1}$ with the simple transpositions as the generating set.
\end{theorem}

Third, we introduce the following group.
Let $\hatP=\hatP(\bw)$ be defined as in the previous section. 
We can then think of $\hatP$ as being made up of $\widehat{0}$, $\widehat{1}$, and ladders $\mL^1, \dots, \mL^t$ for $t\geq 1$ defined as follows.
Given the vertices $w_0,\ldots,w_k$ of $G(\bw)$, let $w_{i_1}$ be the first index such that there is an edge from $w_{i_1}$ to $w_{i_1+2}$.
Then $\mL^1$ is the ladder in $\hatP$ induced by the elements of $\cup_{j=0}^{i_1+1}\Sq(w_j)$.
Let $w_{i_2}$ be the next vertex where there is an edge from $w_{i_2}$ to $w_{i_2+2}$.
Then $\mL^2$ is the ladder in $\hatP$ induced by the elements of $\cup_{j=i_1+1}^{i_2+1}\Sq(w_j)$.
Inductively define $\mL^i$ in a similar fashion.
Note that by definition these ladders are disjoint except that $\mL^i\cap \mL^{i+1}$ is a single square corresponding to a corner box in $\hatP$. That is, $\mL^i\cap \mL^{i+1}$ comes from the underlined letter $\dots R\underline{R} L \dots$ or $\dots L\underline{L} R \dots$ in the expression for $\bw$.
Moreover, we index the ladders so that $y$, the top element of $\mL^1$,  is covered by $\widehat{1}$ in $\hatP$. That is, $y \prec \widehat{1}.$  
Since $\bw$ avoids subwords $LRL$ and $RLR$, each $\mL^i$, for $1<i<t$, consists of at least three squares and $\mL^1, \mL^t$ consist of at least two squares, except for the case where $\bw=\varepsilon$, in which case we have one square and one ladder.
For example, in Figure~\ref{fig.embedH} the poset $\hatP$ consists of five ladders $\mL^1, \dots, \mL^5$ made up of $4, 3, 5, 6,$ and $3$ boxes respectively.  

Let $V_0$ denote the set of vertices of $\hatP$. 
Next, we define a collection of certain permutations on elements of $V_0$.  
Consider the ladder $\mL^i$ for $i\in [t]$ in the poset $\hatP$. 
Then $\mL^i$ has the following structure up to a reflection of $\hatP$ in a vertical axis.
Label the vertices of $\mL^i$ as $x_1, \dots, x_s$ for some even integer $s$ as in Figure~\ref{fig.twist}. In the case where $\bw = \varepsilon$, we resolve the ambiguity of the labeling by choosing the convention that the left and right elements in the antichain of the square have labels $x_2$ and $x_3$ respectively.

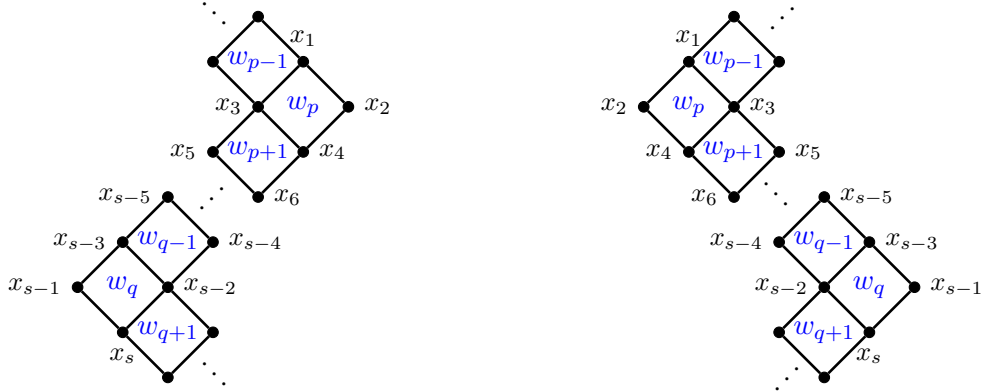
\begin{figure}
\begin{center}
\begin{tikzpicture}[scale=.6]
\begin{scope}[xshift=0, yshift=0] 
	
	\pvx[label=above:{\small$x_1$}](v01) at (1,19) {};
	\pvx[label=left:{\small$x_3$}](v03) at (0,18) {};
	\pvx[label=right:{\small$x_2$}](v02) at (2,18) {};
	\pvx[label=right:{\small$x_4$}](v04) at (1,17) {};
	\pvx[label=left:{\small$x_5$}](v05) at (-1,17) {};
	\pvx[label=right:{\small$x_6$}](v06) at (0,16) {};
	
	\node[] (d2)  at (-1.2,16)  {$\udots$};
	\node[] (d2)  at (-1,11.9)  {$\Ddots$};
	\node[] (d2)  at (-1,19.9)  {$\Ddots$};
	
	\node[] (w1)  at (0,19)  {\color{blue} $w_{p-1}$};
	\node[] (w2)  at (1,18)  {\color{blue} $w_{p}$};
	\node[] (w3)  at (0,17)  {\color{blue} $w_{p+1}$};
	\node[] (w4)  at (-2,15)  {\color{blue} $w_{q-1}$};
	\node[] (w5)  at (-3,14)  {\color{blue} $w_{q}$};
	\node[] (w6)  at (-2,13)  {\color{blue} $w_{q+1}$};
	
	\pvx[label=left:{\small$x_{s-5}$}](v07) at (-2,16) {};
	\pvx[label=left:{\small$x_{s-3}$}](v08) at (-3,15) {};
	\pvx[label=left:{\small$x_{s-1}$}](v09) at (-4,14) {};
	\pvx[label=right:{\small$x_{s-4}$}](v10) at (-1,15) {};
 	\pvx[label=right:{\small$x_{s-2}$}](v11) at (-2,14) {};
	\pvx[label=below:{\small$x_{s}$}](v12) at (-3,13) {};
		
	\pvx[](v13) at (-2,12) {};	
	\pvx[](v14) at (-1,13) {};	
	\pvx[](v15) at (0,20) {};	
	\pvx[](v16) at (-1,19) {};	
	\draw[line width=1pt] (v01)--(v02);
	\draw[line width=1pt] (v03)--(v04);		
	\draw[line width=1pt] (v05)--(v06);
	\draw[line width=1pt] (v01)--(v03);
	\draw[line width=1pt] (v03)--(v05);
	\draw[line width=1pt] (v02)--(v04);
	\draw[line width=1pt] (v04)--(v06);
	
	\draw[line width=1pt] (v07)--(v08);
	\draw[line width=1pt] (v08)--(v09);
	\draw[line width=1pt] (v10)--(v11);
	\draw[line width=1pt] (v11)--(v12);
	\draw[line width=1pt] (v07)--(v10);
	\draw[line width=1pt] (v08)--(v11);
	\draw[line width=1pt] (v09)--(v12);
	
	\draw[line width=1pt] (v13)--(v14);
	\draw[line width=1pt] (v13)--(v12);
	\draw[line width=1pt] (v14)--(v11);
	\draw[line width=1pt] (v15)--(v16);
	\draw[line width=1pt] (v15)--(v01);
	\draw[line width=1pt] (v16)--(v03);

\end{scope}
\begin{scope}[xshift=300, yshift=0, yscale=1,xscale=-1] 
	
	\pvx[label=above:{\small$x_1$}](v01) at (1,19) {};
	\pvx[label=right:{\small$x_3$}](v03) at (0,18) {};
	\pvx[label=left:{\small$x_2$}](v02) at (2,18) {};
	\pvx[label=left:{\small$x_4$}](v04) at (1,17) {};
	\pvx[label=right:{\small$x_5$}](v05) at (-1,17) {};
	\pvx[label=left:{\small$x_6$}](v06) at (0,16) {};
	
	\node[] (d2)  at (-0.9,15.9)  {$\Ddots$};
	\node[] (d2)  at (-1,12)  {$\udots$};
	\node[] (d2)  at (-0.9,20.0)  {$\udots$};
	
	\node[] (w1)  at (0,19)  {\color{blue} $w_{p-1}$};
	\node[] (w2)  at (1,18)  {\color{blue} $w_{p}$};
	\node[] (w3)  at (0,17)  {\color{blue} $w_{p+1}$};
	\node[] (w4)  at (-2,15)  {\color{blue} $w_{q-1}$};
	\node[] (w5)  at (-3,14)  {\color{blue} $w_{q}$};
	\node[] (w6)  at (-2,13)  {\color{blue} $w_{q+1}$};
	
	\pvx[label=right:{\small$x_{s-5}$}](v07) at (-2,16) {};
	\pvx[label=right:{\small$x_{s-3}$}](v08) at (-3,15) {};
	\pvx[label=right:{\small$x_{s-1}$}](v09) at (-4,14) {};
	\pvx[label=left:{\small$x_{s-4}$}](v10) at (-1,15) {};
 	\pvx[label=left:{\small$x_{s-2}$}](v11) at (-2,14) {};
	\pvx[label=below:{\small$x_{s}$}](v12) at (-3,13) {};
		
	\pvx[](v13) at (-2,12) {};	
	\pvx[](v14) at (-1,13) {};	
	\pvx[](v15) at (0,20) {};	
	\pvx[](v16) at (-1,19) {};	
	\draw[line width=1pt] (v01)--(v02);
	\draw[line width=1pt] (v03)--(v04);		
	\draw[line width=1pt] (v05)--(v06);
	\draw[line width=1pt] (v01)--(v03);
	\draw[line width=1pt] (v03)--(v05);
	\draw[line width=1pt] (v02)--(v04);
	\draw[line width=1pt] (v04)--(v06);
	
	\draw[line width=1pt] (v07)--(v08);
	\draw[line width=1pt] (v08)--(v09);
	\draw[line width=1pt] (v10)--(v11);
	\draw[line width=1pt] (v11)--(v12);
	\draw[line width=1pt] (v07)--(v10);
	\draw[line width=1pt] (v08)--(v11);
	\draw[line width=1pt] (v09)--(v12);
	
	\draw[line width=1pt] (v13)--(v14);
	\draw[line width=1pt] (v13)--(v12);
	\draw[line width=1pt] (v14)--(v11);
	\draw[line width=1pt] (v15)--(v16);
	\draw[line width=1pt] (v15)--(v01);
	\draw[line width=1pt] (v16)--(v03);

	\end{scope}	
\end{tikzpicture}
\end{center}
\caption{Ladder $\mL^i$ in $\hatP$ containing boxes with labels $w_p, \dots, w_q$, where $w_p < w_{p+1} < \cdots < w_{q}$. The left (right) represents the case where $w_q = L$ ($w_q=R$).}
\label{fig.twist}
\end{figure}

\begin{definition}
\label{def.twist}
Given a ladder $\mL^i$, define $\tau_i \in \mathfrak{S}_{|V_0|}$ to be the permutation of $V_0$ such that for $v\in V_0$, 
\[
\tau_i(v) = \begin{cases} 
x_{j-1}, 
    & \text{if } v=x_{j} \text{ and } j\in [s] \text{ is even,}\\ 
x_{j+1}, 
    & \text{if } v=x_{j} \text{ and } j\in [s] \text{ is odd,} \\ 
v, & \text{otherwise.} 
\end{cases}
\]
\end{definition}

Hence, $\tau_i$ acts on $V_0$ by reflecting the vertices of $\mL^i$ across a diagonal and fixing the remaining vertices. 
The next lemma says that the set of $\tau_i$ for $i\in[t]$ generate a commutative subgroup of $\mathfrak{S}_{|V_0|}$. 

\begin{lemma}\label{lem:twist1}
For all $\tau_i, \tau_r \in \mathfrak{S}_{|V_0|}$, the following properties hold.
\begin{itemize}
\item[(a)] $\tau_i^2=1$
\item[(b)] $\tau_i\tau_r=\tau_r\tau_i$
\end{itemize}
\end{lemma}

\begin{proof}
Part~(a) follows directly from the definition of $\tau_i$.
Part~(b) is straightforward when $|r-i|\geq 2$ because the ladders $\mL^r, \mL^i$ have no vertices in common. 
The case $r=i$ follows from part (a), so it suffices to consider the case $r=i+1$.  
Moreover, it is enough to check the action of the $\tau$'s on the vertices of the square in $\mL^i\cap \mL^{i+1}$. 
Label the vertices of this square as $x_a, x_b, x_c, x_d$ and consider the computations $\tau_i\tau_{i+1}$ and $\tau_{i+1}\tau_i$, as shown in Figure~\ref{fig.tauscommute}.  
This shows that $\tau_i\tau_{i+1}=\tau_{i+1}\tau_i$ as desired. 
\end{proof}

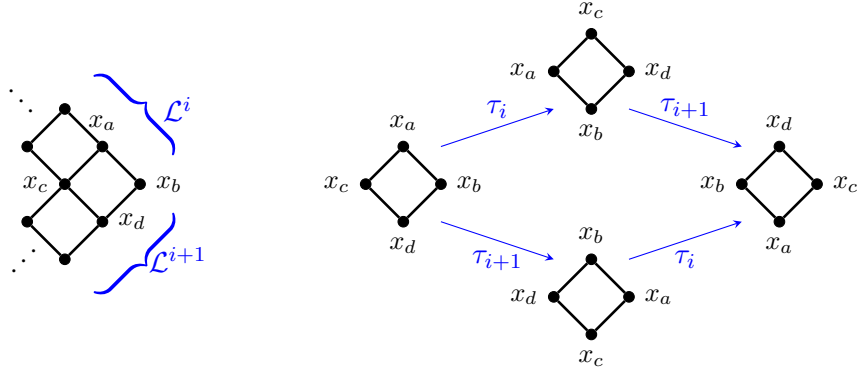
\begin{figure}[ht!]
\centering
\begin{tikzpicture}[scale=.5]
\begin{scope}[xshift=0, yshift=0] 
	
	\pvx[label=above:{\small$x_a$}](v01) at (-1,19) {};
	\pvx[label=left:{\small$x_c$}](v03) at (-2,18) {};
	\pvx[label=right:{\small$x_b$}](v02) at (0,18) {};
	\pvx[label=right:{\small$x_d$}](v04) at (-1,17) {};
	\pvx[](v05) at (-3,17) {};
	\pvx[](v06) at (-2,16) {};
	
	\node[](d2)  at (-3.3,16)  {$\udots$};
	\node[] (d2)  at (-3.2,20)  {$\Ddots$};
	
	\pvx[](v15) at (-2,20) {};	
	\pvx[](v16) at (-3,19) {};	
	\draw[line width=1pt] (v01)--(v02);
	\draw[line width=1pt] (v03)--(v04);		
	\draw[line width=1pt] (v05)--(v06);
	\draw[line width=1pt] (v01)--(v03);
	\draw[line width=1pt] (v03)--(v05);
	\draw[line width=1pt] (v02)--(v04);
	\draw[line width=1pt] (v04)--(v06);
	
	\draw[line width=1pt] (v15)--(v16);
	\draw[line width=1pt] (v15)--(v01);
	\draw[line width=1pt] (v16)--(v03);
	
\node[rotate=-135, scale=1.3] (d2)  at (0,20)  {\color{blue} $\Bigg \{ $};	
\node[] (d2)  at (1,20)  {\color{blue} $\mL^i$};	
\node[rotate=135, scale=1.3] (d2)  at (0,16)  {\color{blue} $\Bigg \{ $};
\node[] (d2)  at (1,16)  {\color{blue} $\mL^{i+1}$};		

\pvx[label=above:{\small$x_a$}](v11) at (7,19) {};
	\pvx[label=left:{\small$x_c$}](v13) at (6,18) {};
	\pvx[label=right:{\small$x_b$}](v12) at (8,18) {};
	\pvx[label=below:{\small$x_d$}](v14) at (7,17) {};

\draw[blue, -stealth] (8,19) -- (11,20)  node[midway, above] {\color{blue} $\tau_i$};
\draw[blue,  -stealth] (8,17) -- (11,16)  node[midway, below] {\color{blue} $\tau_{i+1}$};
\draw[blue, -stealth] (13,20) -- (16,19)  node[midway, above] {\color{blue} $\tau_{i+1}$};
\draw[blue, -stealth] (13,16) -- (16,17)  node[midway, below] {\color{blue} $\tau_{i}$};

\draw[line width=1pt] (v11)--(v12);
	\draw[line width=1pt] (v13)--(v14);		
\draw[line width=1pt] (v12)--(v14);
	\draw[line width=1pt] (v13)--(v11);	
	
	\pvx[label=above:{\small$x_c$}](v21) at (12,22) {};
	\pvx[label=left:{\small$x_a$}](v23) at (11,21) {};
	\pvx[label=right:{\small$x_d$}](v22) at (13,21) {};
	\pvx[label=below:{\small$x_b$}](v24) at (12,20) {};

\draw[line width=1pt] (v21)--(v22);
	\draw[line width=1pt] (v23)--(v24);		
\draw[line width=1pt] (v22)--(v24);
	\draw[line width=1pt] (v23)--(v21);

\pvx[label=above:{\small$x_b$}](v31) at (12,16) {};
	\pvx[label=left:{\small$x_d$}](v33) at (11,15) {};
	\pvx[label=right:{\small$x_a$}](v32) at (13,15) {};
	\pvx[label=below:{\small$x_c$}](v34) at (12,14) {};

\draw[line width=1pt] (v31)--(v32);
	\draw[line width=1pt] (v33)--(v34);		
\draw[line width=1pt] (v32)--(v34);
	\draw[line width=1pt] (v33)--(v31);		
	
\pvx[label=above:{\small$x_d$}](v41) at (17,19) {};
	\pvx[label=left:{\small$x_b$}](v43) at (16,18) {};
	\pvx[label=right:{\small$x_c$}](v42) at (18,18) {};
	\pvx[label=below:{\small$x_a$}](v44) at (17,17) {};

\draw[line width=1pt] (v41)--(v42);
	\draw[line width=1pt] (v43)--(v44);		
\draw[line width=1pt] (v42)--(v44);
	\draw[line width=1pt] (v43)--(v41);	
	
\end{scope}
\end{tikzpicture}
\caption{The action of $\tau_i$ and $\tau_{i+1}$ commute.}
\label{fig.tauscommute}
\end{figure}

\begin{definition}\label{def.twistgroup}
Let $\frakT(\bw)$ denote the subgroup of $\mathfrak{S}_{|V_0|}$ generated by the set of the $\tau_i$'s.
We call $\frakT(\bw)$ {\it the twist group of} $Q_\bw$.
Elements of $\frakT(\bw)$ are called \emph{twists} and the elements $\tau_i$ are called \emph{elementary twists}.
\end{definition}

Note that by Lemma~\ref{lem:twist1}, $\frakT(\bw) = \langle\tau_i \mid i\in[t] \rangle$ is isomorphic to $\mathbb{Z}_2^t$.
The fact that elementary twists commute will be an important factor in several proofs in this section.
As the next theorem demonstrates, the twist group acts on the component of the flip graph of triangulations of $\cO(Q_{\bw})$ containing the canonical triangulation, and flips are preserved by twists.
A priori, a simplex $\sigma$ in the triangulation $T$ after twisting becomes a collection of vertices $\tau(\sigma)$ of $\cO(Q_{\bw})$ that may or may not also form a simplex.
Hence, a twist  $\tau(\cT)$ of a triangulation $\cT$ is a collection of subsets of vertices obtained by applying the twist $\tau$ to every simplex in $\cT$, so  $\tau(\cT)$ is not necessarily a triangulation. However, in the case when twisting results in a triangulation, the following theorem says that twists and flips behave well with each other. 

Recall that if $Z$ is a circuit in $\cO(Q_{\bw})$ and $\cT$ is a triangulation of $\cO(Q_{\bw})$ that admits a flip using $Z$, then $\cT=\cT_Z^+$ and $\cT_Z^-$ are the triangulations related by flips at $Z$.

\begin{theorem}\label{thm:square}
Let $\bw\in \mathcal{V}$, $Q_\mathbf{w} = \mathrm{Irr}_\wedge (\hatP(\mathbf{w}))$, and let $\cT$ and $\tau(\cT)$ be two triangulations of $\cO(Q_\mathbf{w})$ where $\tau$ is a twist. 
If $\cT=\cT_Z^+$ can be flipped at circuit $Z$ and $\tau(\cT_Z^+)=\tau(\cT_Z^+)_{\tau(Z)}^+$, then $\tau(\cT_Z^+)_{\tau(Z)}^-=\tau(\cT_Z^-)$. In other words, the following diagram commutes. 
\[
\xymatrix@C=50pt{\cT_Z^+\ar[r]^-{\text{flip in } Z} \ar[d]_-{\text twist}& \cT_Z^- \ar[d]^-{\text twist} \\
\tau(T_Z^+)=\tau(\cT_Z^+)_{\tau(Z)}^+ \ar[r]^-{\text{flip in } \tau(Z)} & \tau(\cT_Z^+)_{\tau(Z)}^-=\tau(\cT_Z^-)}
\] 
\end{theorem}

\begin{corollary}
Let $\bw\in \mathcal{V}$, $Q_\mathbf{w} = \mathrm{Irr}_\wedge (\hatP(\mathbf{w}))$, and let $\cT$ and $\tau(\cT)$ be two triangulations of $\cO(Q_\mathbf{w})$ where $\tau$ is a twist. 
Then $\cT$ and $\tau(\cT)$ admit the same number of flips.
\label{cor:kFlips}
\end{corollary}

\begin{proof}
As will be seen later, a twist $\tau$ is an involution on $J(Q_\bw)$ that yields an involution on circuits of $\cO(Q_\bw)$. 
Thus, we can apply $\tau$ to $\tau(\cT)$ and recover $\cT$.
This shows that there is a bijective correspondence between flips from $\cT$ and flips from $\tau(\cT)$.
\end{proof}

Lastly, we show that twists of regular triangulations lead to regular triangulations.

\begin{theorem} \label{thm:regularity}
Let $\mathbf{w}\in \mathcal{V}$, $Q_\mathbf{w} = \mathrm{Irr}_\wedge (\hatP(\mathbf{w}))$, and let $\cT_\mathbf{w}$ be the canonical triangulation of $\mathcal{O}(Q_\mathbf{w})$. 
Then $\cT_\mathbf{w}$ is a regular triangulation with height function $\omega(x_i) = 2^{\rho(x_i)}$ defined in Definition~\ref{def:canonicalheight}.
Furthermore, for any twist $\tau$, $\tau(\cT_\mathbf{w})$ is a regular triangulation with the corresponding twisted height function.
\end{theorem}

For each of these theorems, we have dedicated one subsection that follows to their proof.

\subsection{Proof of Theorem~\ref{thm:flipsfromcanonical}}

For the canonical triangulation of $\cO(Q_{\bw})$, simplices correspond to maximal chains in $J(Q_\mathbf{w})$.
In order for a circuit $Z$ to be supported on a flip in the canonical triangulation, then either $\mathcal{T}_+$ or $\mathcal{T}_-$ must be a subcomplex of the canonical triangulation.
Consider the circuits corresponding to a single square $\Sq(w_i)=\{A,A\cup\{a\},A\cup\{b\},A\cup\{a,b\}\}$ in $J(Q_\mathbf{w})$. 
If we assign the left $A \cup \{a\}$ and right $A \cup \{b\}$ vertices of the square to be have negative signs in the circuit, and we assign positive signs to the upper vertex $A$ and the lower vertex $A\cup\{a,b\}$, then the canonical triangulation contains $\mathcal{T}_-$, where $\mathcal{T}_-$ consists of the two triangles each formed by the positive vertices and one of the negative vertices in $\Sq(w_i)$.
Further, for each of these triangles, the link in the canonical triangulation is the set of all chains contained in 
$\{x\in J(Q_\mathbf{w}) \mid x < A\cup\{a,b\} \} \cup \{x\in J(Q_\mathbf{w}) \mid x > A \}$,
and thus $\mathcal{T}_-$ can be flipped at $\Sq(w_i)$.
However, $\mathcal{T}_+$ is not contained in the canonical triangulation, because there is not a maximal chain in $J(Q_\mathbf{w})$ that runs through the two negative vertices, which form an antichain in $J(Q_\mathbf{w})$.
Thus, for each of the $k+1$ squares in $J(Q_\mathbf{w})$, we can flip the corresponding circuit.

If we have a circuit in $J(Q_\mathbf{w})$ that does not come from a square, then we consider two cases.
If the circuit is of size four, then the corresponding connected induced subgraph has vertices $w_{i_1},w_{i_1+1},\ldots,w_{i_1+r}$.
In this case, because there are elements of $J(Q_\mathbf{w})$ strictly between the elements of the circuit and they lie in different ranks, the links are different for the two faces in $\mathcal{T}_-$.
For example, suppose that $\Sq(w_{i_1})$ consists of the filters $A$, $A\cup\{a\}$, $A\cup\{c\}$, and $A\cup\{a,c\}$.
Suppose also that $\Sq(w_{i_1+r})$ consists of the filters $B$, $B\cup\{b\}$, $B\cup\{c\}$, and $B\cup \{b,c\}$.
See Figure~\ref{fig.links} for an example.
Then the filters $A,A\cup \{c\},B\cup\{b,c\}$ form a triangle in $\mathcal{T}_-$, as does $A,B\cup \{b\},B\cup\{b,c\}$.
Note that $B\cup\{c\}$ is in the link of the first triangle but not the second.
A similar argument holds in general, and thus this circuit does not support a flip.

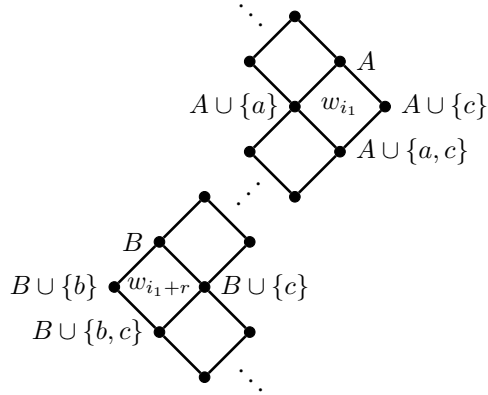
\begin{figure}[ht!]
\begin{center}
\begin{tikzpicture}[scale=.6]
\begin{scope}[xshift=0, yshift=0] 
    \pvx[label=right:{\small$A$}](v01) at (16,19) {};
	\pvx[label=left:{\small$A\cup \{a\}$}](v03) at (15,18) {};
	\pvx[label=right:{\small$A\cup \{c\}$}](v02) at (17,18) {};
	\pvx[label=right:{\small$A\cup \{a,c\}$}](v04) at (16,17) {};
	\pvx[](v05) at (14,17) {};
	\pvx[](v06) at (15,16) {};
	
	\node[] (d2)  at (13.8,16.1)  {$\udots$};
	\node[] (d2)  at (14,11.75)  {$\Ddots$};
	\node[] (d2)  at (14,19.85)  {$\Ddots$};
	
	\pvx[](v07) at (13,16) {};
	\pvx[label=left:{\small$B$}](v08) at (12,15) {};
	\pvx[label=left:{\small$B\cup \{b\}$}](v09) at (11,14) {};
	\pvx[](v10) at (14,15) {};
 	\pvx[label=right:{\small$B\cup \{c\}$}](v11) at (13,14) {};
	\pvx[label=left:{\small$B\cup \{b,c\}$}](v12) at (12,13) {};
		
	\pvx[](v13) at (13,12) {};	
	\pvx[](v14) at (14,13) {};	
	\pvx[](v15) at (15,20) {};	
	\pvx[](v16) at (14,19) {};	
	\draw[line width=1pt] (v01)--(v02);
	\draw[line width=1pt] (v03)--(v04);		
	\draw[line width=1pt] (v05)--(v06);
	\draw[line width=1pt] (v01)--(v03);
	\draw[line width=1pt] (v03)--(v05);
	\draw[line width=1pt] (v02)--(v04);
	\draw[line width=1pt] (v04)--(v06);
	
	\draw[line width=1pt] (v07)--(v08);
	\draw[line width=1pt] (v08)--(v09);
	\draw[line width=1pt] (v10)--(v11);
	\draw[line width=1pt] (v11)--(v12);
	\draw[line width=1pt] (v07)--(v10);
	\draw[line width=1pt] (v08)--(v11);
	\draw[line width=1pt] (v09)--(v12);
	
	\draw[line width=1pt] (v13)--(v14);
	\draw[line width=1pt] (v13)--(v12);
	\draw[line width=1pt] (v14)--(v11);
	\draw[line width=1pt] (v15)--(v16);
	\draw[line width=1pt] (v15)--(v01);
	\draw[line width=1pt] (v16)--(v03);

    \node[] (d2)  at (12,14)  {\small $w_{i_1+r}$};
	\node[] (d2)  at (16,18)  {\small $w_{i_1}$};
	
	\end{scope}
\end{tikzpicture}
\end{center}
\caption{A circuit failing the link condition.}
\label{fig.links}
\end{figure}

For a circuit of size greater than four, the word defining the corresponding induced subgraph either contains a subword of the form $w_{i_j}w_{i_j+2}$ or contains a turn, i.e., contains one of $RRL$, $LLR$, $RLL$, or $LRR$.
Consider the first case, where a subword of the form $w_{i_j}w_{i_j+2}$ is present.
See Figure~\ref{fig.embedH} for an example with $w_{4}w_{6}$.
In this case, it is straightforward to verify that there exist at least two incomparable pairs in the circuit, where the elements of each pair share the same sign in the circuit and each pair has a different sign.
Thus, both $\mathcal{T}_+$ and $\mathcal{T}_-$ have simplices containing antichains, and hence the canonical triangulation does not contain any triangulations arising from circuits of this type.

Next consider the case where the induced subgraph contains a turn, i.e., contains one of $RRL$, $LLR$, $RLL$, or $LRR$.
Again see Figure~\ref{fig.embedH} for an example with $w_{2}w_{3}w_{4}$.
In this case, the existence of a turn implies that there are two antichains $\{A,B\}$ and $\{B,C\}$ where each filter has the same sign in the circuit.
Again, this implies that both $\mathcal{T}_+$ and $\mathcal{T}_-$ have simplices containing antichains, and hence the canonical triangulation does not contain any triangulations arising from circuits of this type.

Thus, the canonical triangulation admits only flips in circuits formed by the squares, and each of the resulting triangulations is distinct, from which the result follows.

\subsection{Proof of Theorem~\ref{thm.Cayleygraph}}

Let $\bw=\varepsilon L^{n-1}$.
Then $\hatP(\bw)\backslash\{\hat0,\hat1\}$ is the ladder with $n$ squares, and the graph $G(\bw)$ associated to $\hatP(\bw)$ is the path graph with $n$ vertices.  
By Theorem~\ref{thm.bijection}, the circuits of the vertex set of the order polytope $\mathcal{O}(Q_\mathbf{w})$ are in bijection with the nonempty connected induced subgraphs of $G(\bw)$, so in this context, every circuit has exactly $4$ elements, corresponding to some nonempty connected subgraph of the path graph with $n$ vertices.  
Since the circuits never contain $\hat0$ or $\hat1$, then without loss of generality we only need to concern ourselves with the remaining $2(n+1)$ vertices of $\hatP(\bw)\backslash\{\hat0,\hat1\}$, which we label  as $[n+1]\cup[(n+1)']$. 
See Figure~\ref{fig.sigmaladder}.
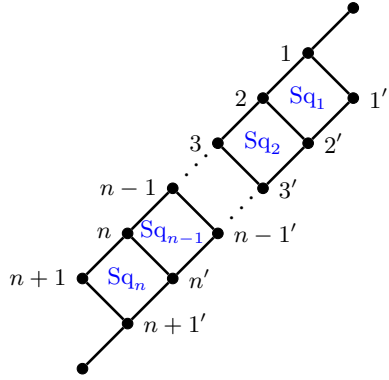
\begin{figure}[ht!]
\begin{center}
\begin{tikzpicture}[scale=.6]
\begin{scope}[xshift=0, yshift=0] 
 	\pvx[](v00) at (2,20) {};
 	\pvx[label=left:{\footnotesize$1$}](v01) at (1,19) {};
 	\pvx[label=left:{\footnotesize$2$}](v03) at (0,18) {};
	\pvx[label=left:{\footnotesize$3$}](v05) at (-1,17) {};
 	\pvx[label=left:{\footnotesize$n-1$}](v07) at (-2,16) {};
 	\pvx[label=left:{\footnotesize$n$}](v08) at (-3,15) {};
 	\pvx[label=left:{\footnotesize$n+1$}](v09) at (-4,14) {};
 	\pvx[label=right:{\footnotesize$1'$}](v02) at (2,18) {};
 	\pvx[label=right:{\footnotesize$2'$}](v04) at (1,17) {};
 	\pvx[label=right:{\footnotesize$3'$}](v06) at (0,16) {};
 	\pvx[label=right:{\footnotesize$n-1'$}](v10) at (-1,15) {};
  	\pvx[label=right:{\footnotesize$n'$}](v11) at (-2,14) {};
 	\pvx[label=right:{\footnotesize$n+1'$}](v12) at (-3,13) {};
 	\pvx[](v13) at (-4,12) {};
 	\node[]  at (-1.65,16.55)  {$\udots$};
 	\node[]  at (-.65,15.6)  {$\udots$};
	
 	\node[] at (1,18)  {\color{blue} \footnotesize$\Sq_1$};
 	\node[] at (0,17)  {\color{blue} \footnotesize$\Sq_2$};
 	\node[] at (-2,15)  {\color{blue} \footnotesize$\Sq_{n-1}$};
 	\node[] at (-3,14)  {\color{blue} \footnotesize$\Sq_n$};
	
 	\draw[line width=1pt] (v00)--(v01)--(v02);
 	\draw[line width=1pt] (v03)--(v04);		
 	\draw[line width=1pt] (v05)--(v06);
 	\draw[line width=1pt] (v01)--(v03);
 	\draw[line width=1pt] (v03)--(v05);
 	\draw[line width=1pt] (v02)--(v04);
 	\draw[line width=1pt] (v04)--(v06);
	
 	\draw[line width=1pt] (v07)--(v08);
 	\draw[line width=1pt] (v08)--(v09);
 	\draw[line width=1pt] (v10)--(v11);
 	\draw[line width=1pt] (v11)--(v12);
 	\draw[line width=1pt] (v07)--(v10);
 	\draw[line width=1pt] (v08)--(v11);
 	\draw[line width=1pt] (v09)--(v12)--(v13);
\end{scope}
\end{tikzpicture}
\end{center}
\caption{The $n$-ladder $J(Q_\mathbf{w})=\hatP(\mathbf{w})$ of Theorem~\ref{thm.Cayleygraph}.
 }
\label{fig.sigmaladder}
\end{figure}

There are $\binom{n+1}{2}$ circuits; explicitly, the circuits are $Z_{i,j} = \{ i,i', j, j'\}$ for $1\leq i<j\leq n+1$.  
In particular, it follows from Theorem~\ref{thm:flipsfromcanonical} that the circuits which support a flip in the canonical triangulation $\cT_\mathbf{w}$ are $Z_{i,i+1}$ for $i=1,\ldots, n$.



Next, we define maps on the labels of $\hatP(\bw)\backslash\{\hat0,\hat1\}$. 
Given $1\leq i \leq n$, if the labels on the four vertices of the $i$-th square $\Sq_i$ are $a,a',b,b'$, then $\pi_i$ swaps $a$ with $b$, and $a'$ with $b'$.
In other words, $\pi_i$ permutes places, not values. 
See Figure~\ref{fig.pi}. 
Compare this to the twist maps which permute values, not places.
It is clear from this definition that $\pi_i$ acts as a simple transposition on the labels of $\hatP(\bw)\backslash\{\hat0,\hat1\}$ so that $\pi_i^2=1$, and moreover, $\pi_1,\ldots,\pi_n$ generate the symmetric group $\mathfrak{S}_{n+1}$.

\begin{figure}[ht!]
\begin{center}
\begin{tikzpicture}[scale=.6]
\begin{scope}[xshift=0, yshift=0] 
	\pvx[](u01) at (1,19) {};
	\pvx[label=left:{\footnotesize$a$}](u02) at (0,18) {};
	\pvx[label=left:{\footnotesize$b$}](u03) at (-1,17) {};
	\pvx[](u04) at (-2,16) {};

	\pvx[](l01) at (2,18) {};
	\pvx[label=right:{\footnotesize$a'$}](l02) at (1,17) {};
	\pvx[label=right:{\footnotesize$b'$}](l03) at (0,16) {};
	\pvx[](l04) at (-1,15) {};
	
	\node[]  at (-2.75,15.5)  {$\udots$}; \node[]  at (-1.75,14.5)  {$\udots$};
	\node[]  at (1.35,19.6)  {$\udots$}; \node[]  at (2.35,18.6)  {$\udots$};
	
	\node[] at (1,18)  {\color{blue} \footnotesize$\Sq_{i-1}$};
	\node[] at (0,17)  {\color{blue} \footnotesize$\Sq_i$};
	\node[] at (-1,16)  {\color{blue} \footnotesize$\Sq_{i+1}$};
	
	\draw[line width=1pt] (u01)--(u04);
	\draw[line width=1pt] (l01)--(l04);
	\draw[line width=1pt] (u01)--(l01);
	\draw[line width=1pt] (u02)--(l02);
	\draw[line width=1pt] (u03)--(l03);
	\draw[line width=1pt] (u04)--(l04);	
\end{scope}

\begin{scope}[xshift=100, yshift=0] 
	\node[] at (0,17) {$\longrightarrow$};
	\node[] at (0,17.5) {$\pi_i$};
\end{scope}

\begin{scope}[xshift=200, yshift=0] 
	\pvx[](u01) at (1,19) {};
	\pvx[label=left:{\footnotesize$b$}](u02) at (0,18) {};
	\pvx[label=left:{\footnotesize$a$}](u03) at (-1,17) {};
	\pvx[](u04) at (-2,16) {};

	\pvx[](l01) at (2,18) {};
	\pvx[label=right:{\footnotesize$b'$}](l02) at (1,17) {};
	\pvx[label=right:{\footnotesize$a'$}](l03) at (0,16) {};
	\pvx[](l04) at (-1,15) {};
	
	\node[]  at (-2.75,15.5)  {$\udots$}; \node[]  at (-1.75,14.5)  {$\udots$};
	\node[]  at (1.35,19.6)  {$\udots$}; \node[]  at (2.35,18.6)  {$\udots$};
	
	\node[] at (1,18)  {\color{blue} \footnotesize$\Sq_{i-1}$};
	\node[] at (0,17)  {\color{blue} \footnotesize$\Sq_i$};
	\node[] at (-1,16)  {\color{blue} \footnotesize$\Sq_{i+1}$};
	
	\draw[line width=1pt] (u01)--(u04);
	\draw[line width=1pt] (l01)--(l04);
	\draw[line width=1pt] (u01)--(l01);
	\draw[line width=1pt] (u02)--(l02);
	\draw[line width=1pt] (u03)--(l03);
	\draw[line width=1pt] (u04)--(l04);	
\end{scope}
\end{tikzpicture}
\end{center}
\caption{The action of $\pi_i$ on the labels of $\hatP(\varepsilon L^{n-1})$.}
\label{fig.pi}
\end{figure}
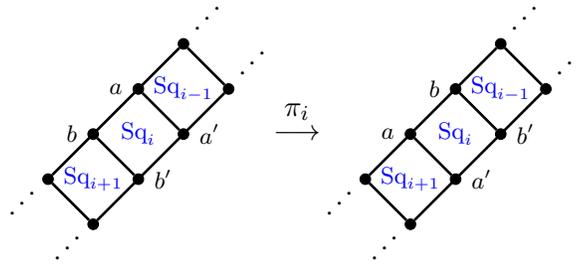

\begin{lemma}\label{lem.2}
Let $\cT_\bw$ denote the canonical triangulation of $\calO(Q_\bw)$. 
Let $\mathcal{U}_i$ denote the triangulation of $\calO(Q_\bw)$ that differs from $\cT_\bw$ by the flip supported at the circuit $Z_{i,i+1}$ for $1 \leq i < n$. 
Then the simplices of $\mathcal{U}_i$ are the maximal chains of the poset $\pi_i \cdot \hatP(\bw)$.
Thus, $\mathcal{U}_i$ is a canonical triangulation of an order polytope, and hence is regular.
\end{lemma}

\begin{proof}
Let
$$Z_{i,i+1} = ((Z_{i,i+1})_+, (Z_{i,i+1})_-) = ( \{i, (i+1)'\}, \{i+1, i'\})$$
be the oriented circuit so that the canonical triangulation $\cT_\bw$ contains $\cT_Z^-$ but not $\cT_Z^+$.  
The $n+1$ maximal chains in $\hatP(\bw)$ are of the form
$$C_i: 1, 2, \ldots, i, i', \ldots, n', (n+1)' $$
for $i=1,\ldots,n+1$. 
To make a flip from $\cT_\bw$ to $\mathcal{U}_i$, each maximal chain in $\hatP(\bw)$ containing $\{i,i+1, (i+1)'\}$ has that triplet replaced by $\{i,i+1, i'\}$, and every maximal chain in $\hatP(\bw)$ containing $\{i,i', (i+1)'\}$ has that triplet replaced by $\{i+1, i', (i+1)'\}$.  
So the flip affects only the two chains $C_i$ and $C_{i+1}$, where effectively, $(i+1)'$ is replaced by $i'$, and $i$ is replaced by $i'$.
Thus the simplices of $\mathcal{U}_i$ are precisely the maximal chains in the poset $\pi_i\cdot\hatP(\bw)$.
\end{proof}

We will denote the triangulation $\mathcal{U}_i$ by $\pi_i\cT_\bw$.

\begin{lemma}\label{lem.3}
Let $\cT$ be a regular triangulation of $\calO(Q_\bw)$ whose simplices are maximal chains in a labeled $n$-ladder poset $\hatP_\mathcal{T}$.
There are exactly $n$ regular triangulations that differ from $\mathcal{T}$ by a circuit flip, and each of these triangulations have simplices which are the maximal chains in $\pi_i \hatP_\mathcal{T}$ for $i=1,\ldots, n$.
\end{lemma}

\begin{proof} Proceed by induction on the number of flips away from the canonical triangulation $\cT_\bw$. 
The base case follows from Lemma~\ref{lem.2}.
Suppose $\mathcal{T} = \pi_{i_\ell} \cdots \pi_{i_1} \mathcal{T}_c$ for some (reduced) sequence of transpositions $\pi_{i_\ell},\ldots, \pi_{i_1}\in \mathfrak{S}_{n+1}$.
Let $\pi = \pi_{i_\ell} \cdots \pi_{i_1} $.
The circuits of $\mathcal{T}$ are then 
$$Z_{\pi^{-1}(j), \pi^{-1}(k)} = \{\pi^{-1}(j), \pi^{-1}(j'), \pi^{-1}(k), \pi^{-1}(k') \}$$ for $1\leq j < k\leq n+1$.  
In particular, for $k=1,\ldots,n$, the four elements of the circuit $Z_{\pi^{-1}(k), \pi^{-1}(k+1)}$ are labels on the square $\Sq_{\pi^{-1}(k)}$ of the poset $\hatP_\mathcal{T}$, and these are the only circuits which support a flip as they are the only ones whose links are the same for the two faces in $\mathcal{T}_Z^-$.  So a flip in this circuit yields a triangulation $\pi \mathcal{T}$ whose simplices are maximal chains in the poset $\pi\cdot \hatP_\mathcal{T}$.
\end{proof}

Since the transpositions $\pi_1,\ldots,\pi_n$ generate $\mathfrak{S}_{n+1}$, the proof of Theorem~\ref{thm.Cayleygraph} now follows.

\subsection{Proof of Theorem~\ref{thm:square}}

We use the notation from Definition~\ref{def.twistgroup}.
We can naturally extend the action of $\frakT(\bw)$ on $V_0$ to the action of $\frakT(\bw)$ on subsets of $V_0$.
The following lemma states that the twist group also acts on the circuits of the vertices of $\cO(Q_\bw)$. 

\begin{lemma}\label{lem:twist2}
Let $Z = (Z_+, Z_-)$ be a circuit on $V_0$ and $\tau\in \frakT(\bw)$. Then $\tau(Z):= (\tau(Z_+), \tau(Z_-))$ is also a circuit on $V_0$.  
\end{lemma}

\begin{proof}
Because $\frakT(\bw)$ is abelian, it suffices to show that $\tau_k(Z)$ is a circuit for every $k$.  
We use the labeling of the squares and vertices of $\mL^k$ given in Figure~\ref{fig.twist}. 
Also, by Theorem~\ref{thm.bijection} we have $Z = \Gamma(H_{\bf w'})$, where $H_{\bf w'}$ is an induced connected subgraph of $G(\bf w)$ corresponding to a subword ${\bf w'}$ of ${\bf w}$.  

Observe that if $Z$ does not contain any vertex of $\mL^k$ then $\tau_k(Z)=Z$ and the lemma holds. 
If every vertex of $Z$ is also a vertex of $\mL^k$, then $Z$ consists of four vertices and we have $Z=(\{x_i, x_j\}, \{x_{i+1}, x_{j-1}\})$ up to interchanging $Z_+, Z_-$ for some $i$ odd and $j$ even.
Then,
\[\tau_k(Z) = (\{x_{i+1}, x_{j-1}\}, \{x_{i}, x_{j}\})=(Z_-, Z_+)\] 
which is the same circuit as $Z$.   
It remains to consider the case where $Z$ contains vertices both in $\mL^k$ and outside of $\mL^k$.  
In this case, it must be that $Z=\Gamma(H_{\bf w'})$ where $\bf w'$ contains one or both of $w_{p-1}$ and $w_{q+1}$. 

If ${\bf w'}$ contains $w_{p-1}$ and no letter in $\mL^k$, then $\tau_k$ applied to $Z$ replaces $x_1, x_3$ with $x_2, x_4$ respectively. 
Hence, $\tau_k(Z)= \Gamma(H_{{\bf w'}w_p})$ is again a circuit.  
Similar computation holds if ${\bf w'}$ contains $w_{q-1}$ and no letter in $\mL^k$.   
Since $\tau_k^2=1$ by Lemma~\ref{lem:twist1}(a), this also resolves the case when $\bf w'$ ends in $w_{p-1}w_p$ or starts with $w_{q}w_{q+1}$.   

The cases when ${\bf w'}$ contains $w_{p-1}, w_{p+1}$ but not $w_{q+1}$ or ${\bf w'}$ contains $w_{q-1}, w_{q+1}$ but not $w_{p-1}$ follow similarly to the case when  $\bf w'$ contains both $w_{p-1}$ and $w_{q+1}$; therefore, we only provide a detailed proof for the latter case.
Suppose that $\bf w'$ contains both $w_{p-1}$ and $w_{q+1}$.  
Then ${\bf w'}$ contains $w_{p+1}, \dots, w_{q-1}$ and it may or may not contain each of $w_p, w_q$. 
We treat the situation when ${\bf w'}$ contains $w_q$ but not $w_p$, and the other possibilities follow similarly.
Up to interchanging $Z_+$ and $Z_-$, we have the following situation where $x_1, x_{s-2}\in Z_-$ and $x_{s-1}, x_4\in Z_+$ and no other vertex in $\mL^k$ appears in the circuit $Z$.
This is depicted in the left-hand side of Figure~\ref{fig.flipsofcircuits}.

In this case, 
\[ 
\tau_k(Z_-)=\left(Z_-\setminus\{x_1, x_{s-2}\}\right)\cup \{x_2, x_{s-3}\}
\]
and
\[ 
\tau_k(Z_+)=\left(Z_+\setminus\{x_4, x_{s+1}\}\right) \cup \{x_3, x_s\} \, .
\]
In particular, $\tau_k(Z) = \Gamma(H_{\bf{w}''})$ where ${\bf{w}''}$ is obtained from $\bf{w}'$ by adding $w_p$ and removing $w_q$.   
Therefore, $\tau_k(Z)$ is a circuit. 
The remaining cases are proved in a similar fashion.
This completes the proof of the lemma.
\end{proof}

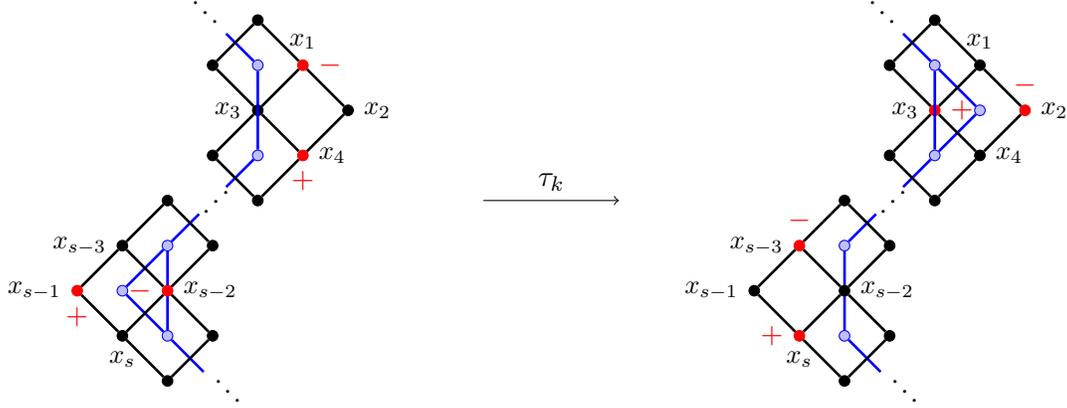
\begin{figure}
\centering
\begin{tikzpicture}[scale=.6]
\begin{scope}[xshift=0, yshift=0] 
	
	\pvx[label=above:{\small$x_1$}, label=right:{ \color{red} $-$}, color=red](v01) at (1,19) {};
	\pvx[label=left:{\small$x_3$}](v03) at (0,18) {};
	\pvx[label=right:{\small$x_2$}](v02) at (2,18) {};
	\pvx[label=right:{\small$x_4$}, label=below:{ \color{red} $+$}, color=red](v04) at (1,17) {};
	\pvx[](v05) at (-1,17) {};
	\pvx[](v06) at (0,16) {};

	\node[] (d2)  at (-1.1,16)  {$\udots$};
	\node[] (d2)  at (-0.7,11.6)  {$\Ddots$};
	\node[] (d2)  at (-1.2,20)  {$\Ddots$};

	\gvx[](c01) at (0,19) {};
	\gvx[](c02) at (0,17) {};
	\gvx[](c03) at (-2,15) {};
	\gvx[](c04) at (-3,14) {};
	\gvx[](c05) at (-2,13) {};
	\draw[blue, line width=1pt] (c01)--(c02);
	\draw[blue, line width=1pt] (-1.2,12.2)--(c05);
	\draw[blue, line width=1pt] (-.7,19.7)--(c01);
	\draw[blue, line width=1pt] (-.7,16.3)--(c02);
	\draw[blue, line width=1pt] (-1.3,15.7)--(c03);
	\draw[blue, line width=1pt] (c03)--(c04);
	\draw[blue, line width=1pt] (c04)--(c05);
	\draw[blue, line width=1pt] (c03)--(c05);
	
	\pvx[](v07) at (-2,16) {};
	\pvx[label=left:{\small$x_{s-3}$}](v08) at (-3,15) {};
	\pvx[label=left:{\small$x_{s-1}$}, label=below:{ \color{red} $+$}, color=red](v09) at (-4,14) {};
	\pvx[](v10) at (-1,15) {};
 	\pvx[label=right:{\small$x_{s-2}$}, label=left:{ \color{red} $-$}, color=red](v11) at (-2,14) {};
	\pvx[label=below:{\small$x_{s}$}](v12) at (-3,13) {};
		
	\pvx[](v13) at (-2,12) {};	
	\pvx[](v14) at (-1,13) {};	
	\pvx[](v15) at (0,20) {};	
	\pvx[](v16) at (-1,19) {};	
	\draw[line width=1pt] (v01)--(v02);
	\draw[line width=1pt] (v03)--(v04);		
	\draw[line width=1pt] (v05)--(v06);
	\draw[line width=1pt] (v01)--(v03);
	\draw[line width=1pt] (v03)--(v05);
	\draw[line width=1pt] (v02)--(v04);
	\draw[line width=1pt] (v04)--(v06);
	
	\draw[line width=1pt] (v07)--(v08);
	\draw[line width=1pt] (v08)--(v09);
	\draw[line width=1pt] (v10)--(v11);
	\draw[line width=1pt] (v11)--(v12);
	\draw[line width=1pt] (v07)--(v10);
	\draw[line width=1pt] (v08)--(v11);
	\draw[line width=1pt] (v09)--(v12);
	
	\draw[line width=1pt] (v13)--(v14);
	\draw[line width=1pt] (v13)--(v12);
	\draw[line width=1pt] (v14)--(v11);
	\draw[line width=1pt] (v15)--(v16);
	\draw[line width=1pt] (v15)--(v01);
	\draw[line width=1pt] (v16)--(v03);

\pvx[label=above:{\small$x_1$}](v01) at (16,19) {};
	\pvx[label=left:{\small$x_3$}, label=right:{ \color{red} $+$}, color=red](v03) at (15,18) {};
	\pvx[label=right:{\small$x_2$}, label=above:{ \color{red} $-$}, color=red](v02) at (17,18) {};
	\pvx[label=right:{\small$x_4$}](v04) at (16,17) {};
	\pvx[](v05) at (14,17) {};
	\pvx[](v06) at (15,16) {};
	
	\node[] (d2)  at (13.9,16)  {$\udots$};
	\node[] (d2)  at (14.3,11.6)  {$\Ddots$};
	\node[] (d2)  at (13.8,20)  {$\Ddots$};

	\gvx[](c01) at (15,19) {};
	\gvx[](c02) at (15,17) {};
	\gvx[](c03) at (13,15) {};
	\gvx[](c00) at (16,18) {};
	\gvx[](c05) at (13,13) {};
	\draw[blue, line width=1pt] (c01)--(c02);
	\draw[blue, line width=1pt] (13.8,12.2)--(c05);
	\draw[blue, line width=1pt] (14.3,19.7)--(c01);
	\draw[blue, line width=1pt] (14.3,16.3)--(c02);
	\draw[blue, line width=1pt] (13.7,15.7)--(c03);
	\draw[blue, line width=1pt] (c00)--(c01);
	\draw[blue, line width=1pt] (c00)--(c02);
	\draw[blue, line width=1pt] (c03)--(c05);
	\pvx[](v07) at (13,16) {};
	\pvx[label=left:{\small$x_{s-3}$}, label=above:{\color{red}$-$}, color=red](v08) at (12,15) {};
	\pvx[label=left:{\small$x_{s-1}$}](v09) at (11,14) {};
	\pvx[](v10) at (14,15) {};
 	\pvx[label=right:{\small$x_{s-2}$}](v11) at (13,14) {};
	\pvx[label=below:{\small$x_{s}$}, label=left:{\color{red}$+$}, color=red](v12) at (12,13) {};
		
	\pvx[](v13) at (13,12) {};	
	\pvx[](v14) at (14,13) {};	
	\pvx[](v15) at (15,20) {};	
	\pvx[](v16) at (14,19) {};	
	\draw[line width=1pt] (v01)--(v02);
	\draw[line width=1pt] (v03)--(v04);		
	\draw[line width=1pt] (v05)--(v06);
	\draw[line width=1pt] (v01)--(v03);
	\draw[line width=1pt] (v03)--(v05);
	\draw[line width=1pt] (v02)--(v04);
	\draw[line width=1pt] (v04)--(v06);
	
	\draw[line width=1pt] (v07)--(v08);
	\draw[line width=1pt] (v08)--(v09);
	\draw[line width=1pt] (v10)--(v11);
	\draw[line width=1pt] (v11)--(v12);
	\draw[line width=1pt] (v07)--(v10);
	\draw[line width=1pt] (v08)--(v11);
	\draw[line width=1pt] (v09)--(v12);
	
	\draw[line width=1pt] (v13)--(v14);
	\draw[line width=1pt] (v13)--(v12);
	\draw[line width=1pt] (v14)--(v11);
	\draw[line width=1pt] (v15)--(v16);
	\draw[line width=1pt] (v15)--(v01);
	\draw[line width=1pt] (v16)--(v03);

\draw[->] (5, 16)--(8,16) node[midway, above] {$\tau_k$};
	
	\end{scope}
\end{tikzpicture}
\caption{Two cases for proof of Lemma~\ref{lem:twist2}.}
\label{fig.flipsofcircuits}
\end{figure}

\begin{proof}[Proof of Theorem~\ref{thm:square}]

Let $\cT=\cT_Z^+$ and $\cT_Z^-$ be triangulations related by a flip at the circuit $Z$.
Let $\cT|_Z$ denote the restriction of $\cT$ to the circuit $Z$.
Since we can flip $\cT$ at $Z$, the links of simplices of $\cT|_Z$ in $\cT$ match. 
By Lemma~\ref{lem:twist2}, the twist $\tau(Z)$ is also a circuit, so the links of $\tau(\cT)|_{\tau(Z)}$ in $\tau(\cT)$ are obtained from the links of $\cT|_Z$ in $\cT$ by applying the permutation $\tau$. 
Thus, they also match. 
Therefore, we obtain two triangulations related by a flip through $\tau(Z)$, denoted $\tau(\cT)=\tau(\cT)_{\tau(Z)}^+$ and $\tau(\cT)_{\tau(Z)}^-$.
It suffices to show that $\tau(\cT)_{\tau(Z)}^-$ and $\tau(\cT_Z^-)$ are the same as sets.  

Every full-dimensional simplex $\sigma$ in $\cT=\cT_Z^+$ that does not contain a simplex supported on $Z$ remains a simplex in $\cT_Z^-$.
Hence $\tau(\cT_Z^-)$ contains $\tau(\sigma)$ as a subset.
Also, $\tau(\sigma)$ is a simplex of $\tau(\cT)=\tau(\cT)_{\tau(Z)}^+$ that does not a simplex supported on $\tau(Z)$, so it remains a simplex in $\tau(\cT)_{\tau(Z)}^-$ after the flip.
Every full-dimensional simplex $\sigma$ in $\cT$ that contains a simplex supported on $Z$ becomes $(\sigma\setminus \{ \sigma_+\})\cup \{\sigma_- \}$ after the flip, for some appropriate pair of subsets $\sigma_+$, $\sigma_-$ obtained from $Z$ by removing a single element in $Z_+$, $Z_-$ respectively.  
Similarly, $\tau(\sigma)$ contains a simplex supported on $\tau(Z)$, and after the flip supported at the circuit $\tau(Z)$ it becomes $(\tau(\sigma)\setminus \{ \tau(\sigma+)\})\cup \{\tau(\sigma_-) \}$ which equals $\tau((\sigma\setminus \{ \sigma_+\})\cup \{\sigma_- \})$.  
\end{proof}


\subsection{Proof of Theorem~\ref{thm:regularity}}

Given a lattice $\hatP =\hatP(\bw)$, let $Q_{\bw}=\mathrm{Irr}_\wedge(\hatP)$ be the poset of meet-irreducibles.
Label the elements of $Q_{\bw}$ as in Figure~\ref{fig:graph_poset_meet}.
Let $\sigma$ be a full-dimensional simplex in the canonical triangulation of the order polytope $\calO(Q_{\bw})$, i.e., $\sigma = \{ \bv_{\emptyset}, \bv_{A_0}, \dots, \bv_{A_m}\}$ is a collection of vertices in $\calO(Q_{\bw})$ that form a maximal chain $\mathbf{0}=\bv_{\emptyset}< \bv_{A_0}< \dots< \bv_{A_m}= \mathds{1}_m$  in $\hatP$, where $m=|Q_{\bw}|$.
Here $\bv_{A_{i}}$ is an indicator vector of $A_{i} \subset Q_{\bw}$.
Since the canonical triangulation is unimodular, we know that every simplex $\sigma$ has normalized volume $1$.
Let $M_\sigma$ denote the matrix with columns $\bv_{A_0}, \dots, \bv_{A_m}$. 
We have the following fact that will be used later in this subsection.

\begin{remark}\label{rem:simplex}
A collection of points $\sigma=\{ \bv_{\emptyset}, \bv_{B_0}, \dots, \bv_{B_m} \}$ in $\calO(Q_{\bw})$ is a simplex of the canonical triangulation of $\calO(Q_{\bw})$ if and only if $\emptyset\subset B_0\subset B_1\subset \cdots \subset B_m$ forms a maximal chain in $\hatP$, which holds if and only if $\det M_{\sigma}\not=0$.
\end{remark}

\begin{figure}
\begin{center}
\begin{tikzpicture}[scale=.6]
\begin{scope}[xshift=0, yshift=0] 

\pvx[label=right:{\small$x_1 = \langle a \rangle$}](v01) at (16,19) {};
	\pvx[label=left:{\small$x_3$}](v03) at (15,18) {};
	\pvx[label=right:{\small$x_2 = \langle a+1 \rangle$}](v02) at (17,18) {};
	\pvx[label=right:{\small$x_4$}](v04) at (16,17) {};
	\pvx[label=left:{\small$x_5 = \langle a+2 \rangle$}](v05) at (14,17) {};
	\pvx[label=right:{\small$x_6$}](v06) at (15,16) {};
	
	\node[] (d2)  at (13.8,16.1)  {$\udots$};
	\node[] (d2)  at (14,11.75)  {$\Ddots$};
	\node[] (d2)  at (14,19.85)  {$\Ddots$};
	
	
	
	\pvx[label=left:{\small$x_{s-5}$}](v07) at (13,16) {};
	\pvx[label=left:{\small$x_{s-3} = \langle a + \frac{s}{2} -2 \rangle$}](v08) at (12,15) {};
	\pvx[label=left:{\small$x_{s-1} = \langle a + \frac{s}{2} -1 \rangle$}](v09) at (11,14) {};
	\pvx[label=right:{\small$x_{s-4}$}](v10) at (14,15) {};
 	\pvx[label=right:{\small$x_{s-2}$}](v11) at (13,14) {};
	\pvx[label=below:{\small$x_{s}$}](v12) at (12,13) {};
		
	\pvx[](v13) at (13,12) {};	
	\pvx[label=right:{\small$ \langle a+\frac{s}{2} \rangle$}](v14) at (14,13) {};	
	\pvx[label=right:{\small$\langle a-1 \rangle$}](v15) at (15,20) {};	
	\pvx[](v16) at (14,19) {};	
	\draw[line width=1pt] (v01)--(v02);
	\draw[line width=1pt] (v03)--(v04);		
	\draw[line width=1pt] (v05)--(v06);
	\draw[line width=1pt] (v01)--(v03);
	\draw[line width=1pt] (v03)--(v05);
	\draw[line width=1pt] (v02)--(v04);
	\draw[line width=1pt] (v04)--(v06);
	
	\draw[line width=1pt] (v07)--(v08);
	\draw[line width=1pt] (v08)--(v09);
	\draw[line width=1pt] (v10)--(v11);
	\draw[line width=1pt] (v11)--(v12);
	\draw[line width=1pt] (v07)--(v10);
	\draw[line width=1pt] (v08)--(v11);
	\draw[line width=1pt] (v09)--(v12);
	
	\draw[line width=1pt] (v13)--(v14);
	\draw[line width=1pt] (v13)--(v12);
	\draw[line width=1pt] (v14)--(v11);
	\draw[line width=1pt] (v15)--(v16);
	\draw[line width=1pt] (v15)--(v01);
	\draw[line width=1pt] (v16)--(v03);

	\end{scope}
\end{tikzpicture}
\end{center}
\caption{Ladder $\mL^k$ in $\hatP$.}
\label{fig.twist2}
\end{figure}

Now we analyze the action of a twist $\tau_k$ on the $0/1$ entries of vertices in $\calO(Q_{\bw})$.
Consider the ladder $\mL^k$ in $\hatP$ with vertices labeled $x_1, \dots, x_s$ as in Figure~\ref{fig.twist2} and let $[a] = \{0,1,2, \dots, a\}$.
Note the inclusion of $0$ in the set $[a]$.
Then, after identifying $B$ with $\bv_B$, we have the following description of the vertices, for some $a'\leq a-1$.
\[
\begin{array}{lll}
x_1 = \bv_{[a]\setminus \{a'\} } = (1^{a'}, 0, 1^{a-a'}, 0, \dots, 0) && 
x_2 = \bv_{[a+1]\setminus \{a'\} } = (1^{a'}, 0, 1^{a-a'+1}, 0, \dots, 0)\\
x_3 = \bv_{[a+1]\setminus\{a+1\} } = (1^{a+1}, 0, \dots, 0) &&
x_4=\bv_{[a+1]} = (1^{a+2}, 0, \dots, 0)\\
x_5=\bv_{[a+2]\setminus\{a+1\}} = (1^{a+1}, 0,1, 0, \dots, 0)&&
x_6=\bv_{[a+2]} = (1^{a+3}, 0, \dots, 0)\\
\vdots && \vdots\\
x_{s-1} = \bv_{[a+\frac{s}{2}-1]\setminus\{a+1\}} = (1^{a+1}, 0,1^{\frac{s}{2}-2} ,0 \dots, 0)&& x_s=\bv_{[a+\frac{s}{2}-1]} = (1^{a+\frac{s}{2}}, 0, \dots, 0)\\
\end{array}
\]

For example in Figure~\ref{fig:graph_poset_meet}, if $k=5$ then $\mL^5$ is the last ladder in $\hatP$ and in this case we have $a=15$ and $a'=11$.
Hence all vertices of $\calO(Q_{\bw})$ in $\mL^k$ are characterized as follows.  

\begin{remark}\label{rem:vertices}
A vertex $\bv_A \in \mL^k$ if and only if $[a]\setminus\{a'\}\subset A$ and $a+\frac{s}{2}, a+\frac{s}{2}+1, \dots, m \not\in A$.  Moreover, for every $\bv_A \in \mL^k$, the vector $\tau_k(\bv_A)$ is obtained from $\bv_A$ by replacing a $0$ or $1$ in position $a+2$ by $1$ or $0$, respectively.  
\end{remark}

The next result says that a twist of a canonical triangulation is again a triangulation.  

\begin{theorem}
Let $\cT_{\bw}$ be the canonical triangulation of $\calO(Q_{\bw})$. 
Then $\tau(\cT_{\bw})$ is also a triangulation of $\calO(Q_{\bw})$ for every $\tau \in \frakT(\bw)$.  
\end{theorem}

\begin{proof}
We prove the theorem in two steps.  
First, we show that a twist of a simplex in $\cT_{\bw}$ is also a simplex. 
Then, we show that no two simplices in $\tau(\cT_{\bw})$ intersect in the interior of $\calO(Q_{\bw})$. 
 Since $\cT_{\bw}$ is a unimodular triangulation and $\tau(\cT_{\bw})$ has the same number of simplices as $\cT_{\bw}$, this implies the desired result that $\tau(\cT_{\bw})$ is also a (unimodular) triangulation of $\calO(Q_{\bw})$.

Let $\sigma = \{ \bv_{\emptyset}, \bv_{A_0}, \dots, \bv_{A_m}\}$ be a simplex in $\cT_{\bw}$, and we claim that $\tau(\sigma)$ is also a simplex.  
Note that $\tau(\bv_{\emptyset})=\bv_{\emptyset}$ and by Remark~\ref{rem:simplex} it suffices to show that $\det M_{\tau(\sigma)}\not=0$. 
First, we address the claim when $\tau=\tau_k$ is an elementary twist.  
Label the vertices $x_i$ of $\mL^k$ by $\bv_A$ for their corresponding subset $A$ of $P$ as discussed earlier, see Figure~\ref{fig.twist2}.
The determinant of $M_\sigma$ is $\pm 1$, and by Remark~\ref{rem:vertices} the two matrices $M_\sigma, M_{\tau_k(\sigma)}$ differ only in row $a+2$.  
Recall that by construction, the entry $M_\sigma$ in position $(i,j)$ equals 1 if and only if $i-1\in A_{j-1}$.
Then, since $\emptyset \subset A_0 \subset \dots \subset A_m$ is a maximal chain, a row $i$ in $M_\sigma$ is of the form $[ 0^{q}, 1^{m+1-q}]$ for some $q\geq 0$ such that $i-1\not\in A_0, \dots, A_{q-1}$ and $i-1 \in A_{q}, \dots, A_m$.   
Let $\text{Row}_M (i)$ denote the $i$-th row of the matrix $M$, so then 
\[
\text{Row}_{M_\sigma}(a+2)=[ 0^{d}, 1^{m+1-d}]
\] 
for some $d$. We will show that $M_{\tau_k(\sigma)}$ is obtained from $M_\sigma$ by a row replacement operation consisting of three rows, as follows. 
Let $\bv_{A_b}, \bv_{A_{b+1}}, \dots, \bv_{A_c}$ be vertices of $\mL^k\cap \sigma$.  
If $i\geq b$ then $a\in A_i$, while if $i<b$ then $a\not\in A_i$.
Hence, 
\[
\text{Row}_{M_\sigma}(a+1)= [0^{b}, 1^{m+1-b}] \, .
\]  
Similarly, if $i\geq c+1$ then $a+s/2\in A_i$ and if $i<c+1$ then $a+\frac{s}{2}\not\in A_i$.  Hence, 
\[
\text{Row}_{M_\sigma}(a+1+\frac{s}{2})= [0^{c+1}, 1^{m-c}] \, .
\]  
Note that $b<d<c$.  By Remark~\ref{rem:vertices}, the $(a+2)$-th row of $M_{\tau_k(\sigma)}$ is obtained from the $(a+2)$-th row of $M_\sigma$ by replacing $0$'s in columns indexed by $\bv_{A_b}, \dots,  \bv_{A_{d-1}}$ by $1$'s, and replacing $1$'s in columns $\bv_{A_d}, \dots,  \bv_{A_{c}}$ by $0$'s.  
Then 
\[
\text{Row}_{M_{\tau_k(\sigma)}}(a+2)=[0^{b}, 1^{d-b}, 0^{c+1-d}, 1^{m-c}]
\]
 while all other rows of $M_{\tau_k(\sigma)}$ remain the same.
 Then the following equation holds. 
\begin{equation}\label{eq:row}
\text{Row}_{M_{\tau_k(\sigma)}}(a+1)-\text{Row}_{M_{\tau_k(\sigma)}}(a+2)+\text{Row}_{M_{\tau_k(\sigma)}}(a+1+\frac{s}{2})= [ 0^{m+1-d},\sigma1^{d}] = \text{Row}_{M_\sigma}(a+2).
\end{equation} 
We have shown that $M_{\tau_k(\sigma)}$ is obtained from $M_\sigma$ by a row replacement operation, so $\det M_{\tau_k(\sigma)}=-\det(M_\sigma) = \pm 1$.  
This shows that $\tau_k(\sigma)$ is also a simplex.  
In the case where $\tau$ is a composition of $\tau_k$'s, the same argument works because rows $a+1$ and $a+1+\frac{s}{2}$ in $M_\sigma$ do not change under twists.  
This completes the proof of the claim that $\tau(\sigma)$ is a simplex.

Now we claim that the images under $\tau$ of two simplices $\sigma_1, \sigma_2$ in the canonical triangulation do not intersect in the interior, that is $\tau(\sigma_1)^{\circ}\cap \tau(\sigma_2)^{\circ}=\emptyset$. 
Suppose to the contrary that there exists some $\tau(p)\in \tau(\sigma_1)^{\circ}\cap \tau(\sigma_2)^{\circ}$.
Let $\sigma_1 =  \{ \bv_{\emptyset}, \bv_{A_0}, \dots, \bv_{A_m}\}$ and $\sigma_2= \{ \bv_{\emptyset}, \bv_{B_0}, \dots, \bv_{B_m}\}$.
Then,
\[
\tau(p)=a_{-1} \mathbf{0} + a_0\tau(\bv_{A_0}) + \dots + a_m\tau(\bv_{A_m}) = b_{-1} \mathbf{0} + b_0\tau(\bv_{B_0}) + \dots + b_m\tau(\bv_{B_m}) 
\]
for some $a_i, b_i>0$ and $\sum_{i=-1}^{m} a_i = \sum_{i=-1}^{m} b_i=1$.   
Now, define
\[
p_A:= a_{-1} \mathbf{0} + a_0 \bv_{A_0} + \dots + a_m \bv_{A_m}
\]
and
\[
p_B:=  b_{-1} \mathbf{0} + b_0 \bv_{B_0} + \dots + b_m \bv_{B_m} \, . 
\]
We will show that $p_A=p_B$ which implies $p_A\in \sigma_1^\circ \cap \sigma_2^\circ = \emptyset$, a contradiction.  
Again we will show this in the case $\tau=\tau_k$, and the general case follows in the same way as in the proof of the first claim.
Since $\tau_k$ only affects $(a+2)$-th index of points in $\sigma_1, \sigma_2$, then $p_A, p_B$ are equal to $\tau_k(p)$ in all positions except for $a+2$. 
Therefore, it suffices to show that $(p_A)_{a+2}=(p_B)_{a+2}$. 
Observe that 
\begin{equation}\label{eq:12}
 \text{Row}_{M_{\tau_k(\sigma_1)}}(a+2) \cdot [a_0, \dots, a_m]=\tau_k(p)_{a+2}= \text{Row}_{M_{\tau_k(\sigma_2)}}(a+2) \cdot [b_0, \dots ,b_m] \, .
\end{equation}

Note that $ \mathbf{0}$ is not a column of $M_\sigma$ by construction, so $a_{-1}, b_{-1}$ are omitted in the equation above. 
Moreover, by equation \eqref{eq:row} we have that $\text{Row}_{M_{\sigma_1}}(a+2)$ (respectively $\text{Row}_{M_{\sigma_2}}(a+2)$) is a linear combination of rows $a+1, a+2,$ and $a+1+s/2$ in $M_{\tau_k(\sigma_1)}$ (respectively $M_{\tau_k(\sigma_2)}$).
Hence, after taking the dot product of both sides of equation \eqref{eq:row} with $[a_0, \dots, a_m]$ and with $[b_0, \dots ,b_m]$, we obtain
\[
(p_A)_{a+2}  = \text{Row}_{M_{\sigma_1}} (a+2) \cdot [a_0, \dots, a_m] = \text{Row}_{M_{\sigma_2}} (a+2) \cdot [b_0, \dots, b_m] = (p_B)_{a+2} \, .
\]
Indeed, this follows because $\tau(p)_i=(p_A)_i=\text{Row}_{M_{\tau_k (\sigma_1)}} (i) \cdot [a_0, \dots, a_m] = \text{Row}_{M_{\tau_k (\sigma_2)}} (i) \cdot [b_0, \dots, b_m]= (p_B)_i$ for all $i\not=a+2$, and for $i=a+2$ we have equation \eqref{eq:12}.

This shows the desired claim, and we obtain a contradiction to $\sigma_1^\circ \cap \sigma_2^\circ = \emptyset$.  
\end{proof}

Next we proceed to show that twisting the canonical triangulation yields regular triangulations. 
Given a point configuration $\mathbf{A} = (\mathbf{p}_1,...,\mathbf{p}_m)$ in $\R^d$ and a height function $\omega: \mathbf{A} \to \R$, let $\mathbf{A}^\omega$ be the lifted point configuration given by $\mathbf{A}^\omega = \begin{pmatrix}
\mathbf{p}_1 & \cdots & \mathbf{p}_m \\
\omega(\bp_1) & \cdots & \omega (\bp_m) \\
\end{pmatrix} $.
Let $\cT(\mathbf{A},\omega)$ denote the regular triangulation of $\mathbf{A}$ induced by the height function $\omega$. 

For a simplex $\sigma$ in a triangulation $\cT$ of $\mathbf{A}$, we use $\mathbf{A}|_{\sigma}$ to denote the matrix whose columns correspond to the vertices of $\sigma$, and similarly for $\mathbf{A}^\omega|_{\sigma}$. 
The main tool we will use is the following theorem.
Recall that a \emph{wall} in a triangulation is a simplex of codimension one that is a face of two maximal simplices.

\begin{theorem}[\cite{DeLoeraRambauSantos}, Theorem 2.3.20] 
\label{thm.2.3.20}
Let $\cT$ be a triangulation of a point configuration $\mathbf{A}\subset \R^d$, and let $\omega : \mathbf{A} \to \mathbb{R}$ be a height function. Then one has $\cT = \cT(\mathbf{A},\omega)$ if and only if the local folding condition holds, i.e., if for every wall $\sigma_0 \in \cT$, with incident full-dimensional simplices $\sigma
_1$ and $\sigma_2$, the point $\bv \in V(\sigma_1) \setminus V(\sigma_2)$ lies above the hyperplane containing $\mathbf{A}^\omega|_{\sigma_2}$
and vice versa. 
\end{theorem}

We use the notation $(M:\mathbf{v})$ to denote the matrix $M$ with the column vector $\mathbf{v}$ appended as the last column.
Further, $\mathrm{sign}\det(M)$ denotes the sign of the determinant of $M$.

\begin{definition}[\cite{DeLoeraRambauSantos}, Definition 5.2.4] 
Let $B=\{\mathbf{p}_{r_1},...,\mathbf{p}_{r_{d+1}}\}$ be a basis for $\mathbf{A}$, i.e. an affinely independent set of full dimension. 
Then for any $\mathbf{p}_j \in \mathbf{A}$, the linear form 
\begin{align*}
    \Psi_{B,\mathbf{p}_j}(\omega) &:= \mathrm{sign} \det(\mathbf{p}_{r_1},...,\mathbf{p}_{r_{d+1}})\cdot \det\begin{pmatrix} \mathbf{p}_{r_1} & \cdots  & \mathbf{p}_{r_{d+1}} & \mathbf{p}_j \\ \omega(\bp_{r_1}) & \cdots & \omega(\bp_{r_{d+1}}) & \omega(\bp_{j}) 
    \end{pmatrix} \\
    &:= \mathrm{sign} \det(\mathbf{A}|_B)\cdot \det(\mathbf{A}^\omega|_B:(\mathbf{p}_j,\omega(\bp_j)))
\end{align*}
is called the {\em folding form} of $\mathbf{p}_j$ with respect to $B$. 
\end{definition}

The folding form can be used to check the local folding condition of Theorem~\ref{thm.2.3.20} since the point $\mathbf{p}_j~\in~\sigma_1\setminus~\sigma_2$ lies above the hyperplane containing $\mathbf{A}^\omega|_{\sigma_2}$ if $\Psi_{\sigma_2,\mathbf{p}_j}(\omega) > 0$.
Let $\mathbf{w} = w_0 w_1\cdots w_k$ be a fixed word.  Let $Q_\mathbf{w}$ be the poset for which $J(Q_\mathbf{w}) = \hatP(\mathbf{w})$.
We will slightly alter the labeling convention used in the previous section for the elements in $Q_\mathbf{w}$. 
Relabel the vertices $0,1,2$ to be $1,2,3$ respectively, relabel vertex $i$ as $i+2$ for $3 \leq i \leq n-1$, and relabel $n$ with $4$. See Figure~\ref{fig.newLabeling} for an example. 
This new labeling may seem somewhat peculiar at first, but it will make an inductive proof for regularity easier for the following reason: when $k> 0$, removing the highest label in $Q_\mathbf{w}$ with the new labeling yields the poset $Q_{\mathbf{w}\setminus\{w_k\}}$, with $J(Q_{\mathbf{w}\setminus\{w_k\}}) = \hatP(w_0\cdots w_{k-1})$.
In other words, removing the highest label in $Q_{\mathbf{w}}$ corresponds with removing the last box in $\hatP(\mathbf{w})$.
We then get the sequence of posets $Q_\mathbf{w}\supset Q_{\mathbf{w}\setminus \{w_k\}} \supset \cdots \supset Q_{w_0}$, where each poset is obtained as a subposet from the previous one by removing the vertex with the highest label. 
The corresponding posets of upper order ideals are then $\hatP(w_0\cdots w_{k-1}w_k)\supseteq \hatP(w_0\cdots w_{k-1})\supseteq \cdots \supseteq \hatP(w_0)$. 
We label vertices of $\mathcal{O}(Q_\mathbf{w})$ by the generators of their corresponding upper order ideals. 
For example, a vertex $\mathbf{v}_{\langle a,b\rangle}$ is the vertex corresponding to the upper order ideal generated by $a$ and $b$ in $Q_\bw$. We label the vertex corresponding to the empty upper order ideal with $\mathbf{v}_\emptyset$.

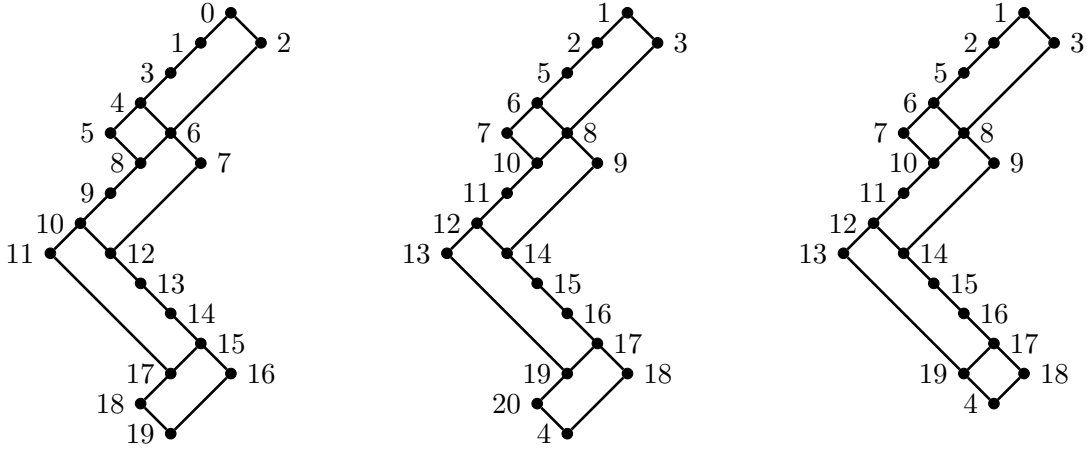
\begin{figure}[h!]
\begin{center}
\begin{tikzpicture}
\begin{scope}[xshift=-150, yshift=0, scale=0.4] 
	\pvx[label=left:$0$](v00) at (2,14) {};
	\pvx[label=left:$1$](v01) at (1,13) {};
	\pvx[label=right:$2$](v02) at (3,13) {};
	\pvx[label=left:$3$](v03) at (0,12) {};
	\pvx[label=left:$4$](v04) at (-1,11) {};
	\pvx[label=left:$5$](v05) at (-2,10) {};
	\pvx[label=right:$6$](v06) at (0,10) {};					
	\pvx[label=right:$7$](v07) at (1,9) {};
	\pvx[label=left:$8$](v08) at (-1,9) {};	
	\pvx[label=left:$9$](v09) at (-2,8) {};	
	\pvx[label=left:$10$](v10) at (-3,7) {};
	\pvx[label=left:$11$](v11) at (-4,6) {};
	\pvx[label=right:$12$](v12) at (-2,6) {};
	\pvx[label=right:$13$](v13) at (-1,5) {};
	\pvx[label=right:$14$](v14) at (0,4) {};
	\pvx[label=right:$15$](v15) at (1,3) {};
	\pvx[label=left:$17$](v17) at (0,2) {};
	\pvx[label=right:$16$](v16) at (2,2) {};
	\pvx[label=left:$18$](v18) at (-1,1) {};
	\pvx[label=left:$19$](v19) at (0,0) {};

	\draw[line width=1pt] (v19)--(v16);
	\draw[line width=1pt] (v19)--(v18);
	\draw[line width=1pt] (v18)--(v15);
	\draw[line width=1pt] (v16)--(v10);
	\draw[line width=1pt] (v17)--(v11);
	\draw[line width=1pt] (v11)--(v02);
	\draw[line width=1pt] (v12)--(v07);
	\draw[line width=1pt] (v07)--(v04);
	\draw[line width=1pt] (v08)--(v05);
	\draw[line width=1pt] (v05)--(v00);
	\draw[line width=1pt] (v02)--(v00);
\end{scope}
\begin{scope}[xshift=0, yshift=0, scale=0.4] 
	\pvx[label=left:$1$](v00) at (2,14) {};
	\pvx[label=left:$2$](v01) at (1,13) {};
	\pvx[label=right:$3$](v02) at (3,13) {};
	\pvx[label=left:$5$](v03) at (0,12) {};
	\pvx[label=left:$6$](v04) at (-1,11) {};
	\pvx[label=left:$7$](v05) at (-2,10) {};
	\pvx[label=right:$8$](v06) at (0,10) {};					
	\pvx[label=right:$9$](v07) at (1,9) {};
	\pvx[label=left:$10$](v08) at (-1,9) {};	
	\pvx[label=left:$11$](v09) at (-2,8) {};	
	\pvx[label=left:$12$](v10) at (-3,7) {};
	\pvx[label=left:$13$](v11) at (-4,6) {};
	\pvx[label=right:$14$](v12) at (-2,6) {};
	\pvx[label=right:$15$](v13) at (-1,5) {};
	\pvx[label=right:$16$](v14) at (0,4) {};
	\pvx[label=right:$17$](v15) at (1,3) {};
	\pvx[label=left:$19$](v17) at (0,2) {};
	\pvx[label=right:$18$](v16) at (2,2) {};
	\pvx[label=left:$20$](v18) at (-1,1) {};
	\pvx[label=left:$4$](v19) at (0,0) {};

	\draw[line width=1pt] (v19)--(v16);
	\draw[line width=1pt] (v19)--(v18);
	\draw[line width=1pt] (v18)--(v15);
	\draw[line width=1pt] (v16)--(v10);
	\draw[line width=1pt] (v17)--(v11);
	\draw[line width=1pt] (v11)--(v02);
	\draw[line width=1pt] (v12)--(v07);
	\draw[line width=1pt] (v07)--(v04);
	\draw[line width=1pt] (v08)--(v05);
	\draw[line width=1pt] (v05)--(v00);
	\draw[line width=1pt] (v02)--(v00);
\end{scope}
\begin{scope}[xshift=150, yshift=0, scale=0.4] 
	\pvx[label=left:$1$](v00) at (2,14) {};
	\pvx[label=left:$2$](v01) at (1,13) {};
	\pvx[label=right:$3$](v02) at (3,13) {};
	\pvx[label=left:$5$](v03) at (0,12) {};
	\pvx[label=left:$6$](v04) at (-1,11) {};
	\pvx[label=left:$7$](v05) at (-2,10) {};
	\pvx[label=right:$8$](v06) at (0,10) {};					
	\pvx[label=right:$9$](v07) at (1,9) {};
	\pvx[label=left:$10$](v08) at (-1,9) {};	
	\pvx[label=left:$11$](v09) at (-2,8) {};	
	\pvx[label=left:$12$](v10) at (-3,7) {};
	\pvx[label=left:$13$](v11) at (-4,6) {};
	\pvx[label=right:$14$](v12) at (-2,6) {};
	\pvx[label=right:$15$](v13) at (-1,5) {};
	\pvx[label=right:$16$](v14) at (0,4) {};
	\pvx[label=right:$17$](v15) at (1,3) {};
	\pvx[label=left:$19$](v17) at (0,2) {};
	\pvx[label=right:$18$](v16) at (2,2) {};
	\pvx[label=left:$4$](v18) at (1,1) {};

	\draw[line width=1pt] (v16)--(v18)--(v17);
	\draw[line width=1pt] (v17)--(v15);
	\draw[line width=1pt] (v16)--(v10);
	\draw[line width=1pt] (v17)--(v11);
	\draw[line width=1pt] (v11)--(v02);
	\draw[line width=1pt] (v12)--(v07);
	\draw[line width=1pt] (v07)--(v04);
	\draw[line width=1pt] (v08)--(v05);
	\draw[line width=1pt] (v05)--(v00);
	\draw[line width=1pt] (v02)--(v00);
\end{scope}
\end{tikzpicture}
\caption{The old labeling of $Q_\mathbf{w}$ (left). The altered labeling for $Q_\mathbf{w}$ (center). Removing the largest label $20$ in $Q_\mathbf{w}$ with the altered labeling forms $Q_{\mathbf{w}\setminus \{w_k\}}$ (right). }
\label{fig.newLabeling}
\end{center}
\end{figure}

As before, we think of $J(Q_\mathbf{w})$ as being made up of $\widehat{0}$, $\widehat{1}$ and ladders $\mL^1,...,\mL^t$. 
Label $\widehat{0}$ and $\widehat{1}$ with $x_0$ and $x_{2k+5}$, respectively. 
In the case that $\bw = \varepsilon$, we resolve the ambiguity as before in the discussion preceding Definition~\ref{def.twist}. 
That is, we treat $\varepsilon$ as an $R$ word. 
Refer to Figure~\ref{fig.baseCase} for an illustration of this labeling. 
In the case that $\bw \neq \varepsilon$, we label the remaining elements of $\hatP(\bw)$ with labels $x_i$, where $1\leq i \leq 2k+4$ in sequential order as follows. 
First label the element covered by $x_0$ with $x_1$, then label the other element of that rung in $\mL^1$ with $x_2$. 

Continue labeling the pairs on the rungs of $\mL^1$ until all elements of $\mL^1$ have labels $x_1\ldots,x_s$ for some even integer $s$ such that the even indexed labels are on one side of $\mL^1$ and the odd indexed labels are on the other. 
Note that this labeling of $\mL^1$ is consistent with that presented in Figure~\ref{fig.twist}.
Note that since $\mL^1$ and $\mL^2$ intersect in a square, the first two rungs of $\mL^2$ have already been assigned labels. 
Then, we proceed by labeling the unlabeled rungs of $\mL^2$ in a similar fashion to the labeling of $\mL_1$. 
That is, we label the first unlabeled rung of $\mL^2$ with $x_{s+1}$ covering $x_{s+2}$ 
and so on until all remaining elements of $\mL^2$ have been labeled. 
Continue labeling in this manner for the remaining ladders until every element of $\hatP(\bw)$ is assigned a label.
The left poset of Figure~\ref{fig:pushingOrder} demonstrates the labeling of $\hatP(\bw)$ when $\bw = \varepsilon RRLLR$.

We introduce the following terminology and notation for the following proof of Theorem~\ref{thm:regularity}. 
The \emph{canonical order} on $V(\mathcal{O}(Q_\mathbf{w}))$ is defined by
\[
x_0,x_2,x_1,x_4,x_3,x_6,x_5,...,x_s,x_{s-1},...,x_{2k+4},x_{2k+3},x_{2k+5} \, .
\]
We call this order the canonical order since it is used to construct a height function yielding the canonical triangulation.
Recall that applying a twist $\tau \in \frakT(\bw)$ to $\hatP(\bw)$ permutes the vertices $V(\mathcal{O}(Q_\mathbf{w}))$. 
Therefore, applying $\tau$ to the canonical order yields a new ordering, which we call the \textit{twisted order} with respect to $\tau$. We denote by $\tau(x_i)$ the label which replaces $x_i$ as a result of the twist.

\begin{definition}\label{def:canonicalheight}
If $x_i$ is the $k$-th element in the canonical order, let $\rho(x_i) = k-1$.
The {\em canonical height function} is the function $\omega:\mathbf{A} \to \mathbb{R}$ given by $\omega(x_i) = 2^{\rho(x_i)}$. 
Furthermore, we define the \textit{twisted height function} $\omega_\tau:\mathbf{A}\to \R$ to be given by $\omega_\tau(x_i) = \omega(\tau(x_i))$. Note that taking $\tau = \mathrm{id}$ gives the canonical height function.
\end{definition}

For example, consider the vertex $x_9$ of $V(\calO(Q_\bw))$ in Figure~\ref{fig:pushingOrder}. 
It appears as the $11$-th element in the canonical order. 
Thus $\omega(x_9) = 2^{\rho(x_9)} = 2^{10}$. 
Applying the twist $\tau_2$ on $V(\calO(Q_\bw))$ gives $\tau_2(x_{10}) = x_{9}$. 
Therefore $\omega_{\tau_2}(x_{10}) = \omega(\tau_2(x_{10})) = \omega(x_{9}) = 2^{10}$.   

\begin{center}
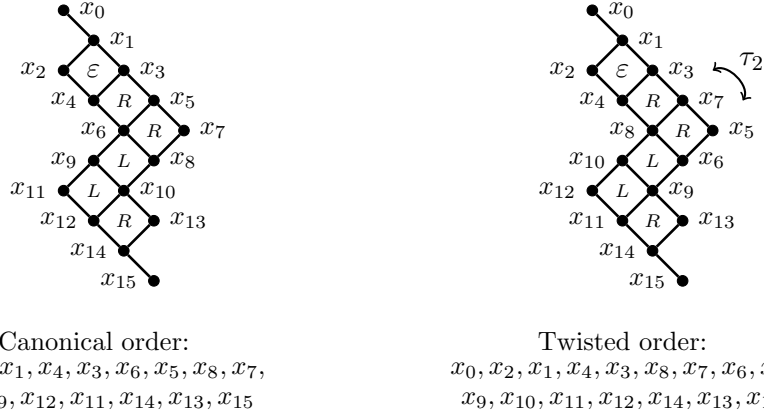
\begin{figure}
\begin{tikzpicture}
\begin{scope}[xshift=0, yshift=0, scale=0.4] 
	\pvx[label=right:{\small $x_0$}](-1) at (-1,11) {};
	\pvx[label=right:{\small $x_1$}](0) at (0,10) {};	
	\pvx[label=left:{\small $x_2$}](1) at (-1,9) {};
	\pvx[label=right:{\small $x_3$}](2) at (1,9) {};
	\pvx[label=left:{\small $x_4$}](3) at (0,8) {};
	\pvx[label=right:{\small $x_5$}](4) at (2,8) {};
	\pvx[label=left:{\small $x_6$}](5) at (1,7) {};
	\pvx[label=right:{\small $x_7$}](6) at (3,7) {};
	\pvx[label=right:{\small $x_8$}](7) at (2,6) {};
	\pvx[label=left:{\small $x_9$}](8) at (0,6) {};
	\pvx[label=right:{\small $x_{10}$}](9) at (1,5) {};	
	\pvx[label=left:{\small $x_{11}$}](10) at (-1,5) {};
	\pvx[label=left:{\small $x_{12}$}](11) at (0,4) {};	
	\pvx[label=left:{\small $x_{14}$}](13) at (1,3) {};
	\pvx[label=right:{\small $x_{13}$}](12) at (2,4) {};
	\pvx[label=left:{\small $x_{15}$}](14) at (2,2) {};	

	\draw[line width=1pt] (-1)--(0);			
	\draw[line width=1pt] (0)--(1);
	\draw[line width=1pt] (0)--(2);			
	\draw[line width=1pt] (1)--(3);	
	\draw[line width=1pt] (2)--(3);	
	\draw[line width=1pt] (2)--(4);
	\draw[line width=1pt] (3)--(5);
	\draw[line width=1pt] (4)--(5);		
	\draw[line width=1pt] (4)--(6);
	\draw[line width=1pt] (5)--(7);
	\draw[line width=1pt] (6)--(7);	
	\draw[line width=1pt] (5)--(8);	
	\draw[line width=1pt] (9)--(7);	
	\draw[line width=1pt] (8)--(9);	
	\draw[line width=1pt] (8)--(10);	
	\draw[line width=1pt] (10)--(11);	
	\draw[line width=1pt] (9)--(11);		
	\draw[line width=1pt] (11)--(13);
	\draw[line width=1pt] (9)--(12);
	\draw[line width=1pt] (12)--(13);		
	\draw[line width=1pt] (14)--(13);
	
	\node[](c01) at (0,9) {$\varepsilon$};
	\node[](c01) at (1,8) {{\tiny $R$}};
	\node[](c01) at (2,7) {{\tiny $R$}};
	\node[](c01) at (1,6) {{\tiny $L$}};	
	\node[](c01) at (0,5) {{\tiny $L$}};
	\node[](c01) at (1,4) {{\tiny $R$}};
	
	\node[](c01) at (0,0) {{\small Canonical order:}};
	\node[](c01) at (0,-1) {{\small $x_0,x_2,x_1,x_4,x_3,x_6,x_5,x_8,x_7,$}};
	\node[](c01) at (0,-2) {{\small $x_{10},x_9, x_{12},x_{11}, x_{14},x_{13},x_{15}$}};
	
\end{scope}
\begin{scope}[xshift=200, yshift=0, scale=0.4] 
	\pvx[label=right:{\small $x_0$}](-1) at (-1,11) {};
	\pvx[label=right:{\small $x_1$}](0) at (0,10) {};	
	\pvx[label=left:{\small $x_2$}](1) at (-1,9) {};
	\pvx[label=right:{\small $x_3$}](2) at (1,9) {};
	\pvx[label=left:{\small $x_4$}](3) at (0,8) {};
	\pvx[label=right:{\small $x_7$}](4) at (2,8) {};
	\pvx[label=left:{\small $x_8$}](5) at (1,7) {};
	\pvx[label=right:{\small $x_5$}](6) at (3,7) {};
	\pvx[label=right:{\small $x_6$}](7) at (2,6) {};
	\pvx[label=left:{\small $x_{10}$}](8) at (0,6) {};
	\pvx[label=right:{\small $x_{9}$}](9) at (1,5) {};	
	\pvx[label=left:{\small $x_{12}$}](10) at (-1,5) {};
	\pvx[label=left:{\small $x_{11}$}](11) at (0,4) {};	
	\pvx[label=left:{\small $x_{14}$}](13) at (1,3) {};
	\pvx[label=right:{\small $x_{13}$}](12) at (2,4) {};
	\pvx[label=left:{\small $x_{15}$}](14) at (2,2) {};	

	\draw[line width=1pt] (-1)--(0);			
	\draw[line width=1pt] (0)--(1);
	\draw[line width=1pt] (0)--(2);			
	\draw[line width=1pt] (1)--(3);	
	\draw[line width=1pt] (2)--(3);	
	\draw[line width=1pt] (2)--(4);
	\draw[line width=1pt] (3)--(5);
	\draw[line width=1pt] (4)--(5);		
	\draw[line width=1pt] (4)--(6);
	\draw[line width=1pt] (5)--(7);
	\draw[line width=1pt] (6)--(7);	
	\draw[line width=1pt] (5)--(8);	
	\draw[line width=1pt] (9)--(7);	
	\draw[line width=1pt] (8)--(9);	
	\draw[line width=1pt] (8)--(10);	
	\draw[line width=1pt] (10)--(11);	
	\draw[line width=1pt] (9)--(11);		
	\draw[line width=1pt] (11)--(13);
	\draw[line width=1pt] (9)--(12);
	\draw[line width=1pt] (12)--(13);		
	\draw[line width=1pt] (14)--(13);
	
	\node[](c01) at (0,9) {$\varepsilon$};
	\node[](c01) at (1,8) {{\tiny $R$}};
	\node[](c01) at (2,7) {{\tiny $R$}};
	\node[](c01) at (1,6) {{\tiny $L$}};	
	\node[](c01) at (0,5) {{\tiny $L$}};
	\node[](c01) at (1,4) {{\tiny $R$}};
	
	\node[](c01) at (4.3,9.3) {{$\tau_2$}};	
	\draw [<->,thick] (3,9) to[out=30,in=60] (4,8);

	\node[](c01) at (0,0) {{\small Twisted order:}};
	\node[](c01) at (0,-1) {{\small $x_0,x_2,x_1,x_4,x_3,x_8,x_7,x_6,x_5,$}};
	\node[](c01) at (0,-2) {{\small $x_9, x_{10}, x_{11},x_{12}, x_{14},x_{13},x_{15}$}};

\end{scope}
\end{tikzpicture}
\caption{An illustration of the canonical order and a twisted order.}
\label{fig:pushingOrder}
\end{figure}
\end{center}

We are now ready to prove Theorem~\ref{thm:regularity}. 

\begin{proof}[Proof of Theorem~\ref{thm:regularity}]
For a fixed word $\mathbf{w}$ of length $k$, let $\cT_\mathbf{w}$ denote the canonical triangulation of $\mathcal{O}(Q_{\mathbf{w}})$.
Let $n$ be the number of elements in $Q_\mathbf{w}$, that is, $n=k+4$. 
Let $\mathbf{A}_n = 
\begin{pmatrix} 
\mathbf{v}_{x_0} & \mathbf{v}_{x_1} & \cdots & \mathbf{v}_{x_{2n-3}} \\ 
1 & 1 & \cdots & 1  
\end{pmatrix}$ 
be the homogenized point configuration associated with the vertices of the order polytope $\mathcal{O}(Q_\mathbf{w})$, where $\bv_{x_i}$ is the vertex associated with $x_i\in \hatP(\bw)$. 
 We will show that $\tau(\cT_\mathbf{w}) = \cT(\mathbf{A}_n,\omega_\tau)$ for all $\tau \in \frakT(\bw)$ by inducting on $n$, or equivalently $k$. We first check the base case when $n=4$ and $k=0$, i.e., when $\mathbf{w} = \varepsilon$. 

\begin{figure}
\centering
\begin{tikzpicture}[scale=.5]
\begin{scope}[xshift=0, yshift=0] 
	\pvx[label=left:\tiny{$x_0 = \emptyset$}](emp) at (1,20.4) {};
	\pvx[label=left:\tiny{$x_1=\langle1\rangle$}](0) at (1,19) {};
	\pvx[label=left:\tiny{$x_2=\langle2\rangle$}](1) at (0,18) {};
	\pvx[label=right:\tiny{$x_3 =\langle3\rangle$}](2) at (2,18) {};
	\pvx[label=right:\tiny{$x_5=\langle4\rangle$}](3) at (1,15.6) {};
	\pvx[label=right:\tiny{$x_4=\langle 2,3\rangle$}](12) at (1,17) {};

	\draw[line width=1pt] (emp)--(0);
	\draw[line width=1pt] (1)--(0);
	\draw[line width=1pt] (2)--(0);
	\draw[line width=1pt] (12)--(2);			
	\draw[line width=1pt] (12)--(1);
	\draw[line width=1pt] (12)--(3);	
	
	\node[](c01) at (1,18) {$\varepsilon$};	
	\node[] at (-5,18) {$\hatP=\hatP(\varepsilon):$};	
\end{scope}
\begin{scope}[xshift=-350, yshift=140] 
	\pvx[label=above:$1$](v00) at (2,14) {};
	\pvx[label=left:$2$](v01) at (1,13) {};
	\pvx[label=right:$3$](v02) at (3,13) {};
	\pvx[label=below:$4$](v03) at (2,12) {};

	\draw[line width=1pt] (v03)--(v02)--(v00)--(v01)--(v03);	
	\node[] at (-3.75,13) {$Q_\varepsilon =\Irr_\wedge(\hatP(\varepsilon)):$};
\end{scope}
\end{tikzpicture}
\caption{Base case in proof of Theorem~\ref{thm:regularity} with $\bw = \varepsilon$.}
\label{fig.baseCase}
\end{figure}
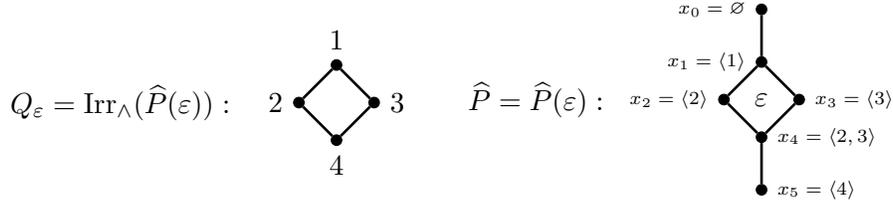

Although we previously chose the convention to treat $\varepsilon$ as an $R$, here it is necessary to also consider the case when $\varepsilon$ is treated as an $L$ as we will induct on the length of the word and a priori, we do not know whether $w_1 = L$ or $w_1 = R$. Both cases, however, are similar, and we therefore check the case with $\varepsilon$ treated as an $R$ word (see Figure~\ref{fig.baseCase}), leaving the other base case for the reader. 

Treating $\varepsilon$ as an $R$, we must check that $\tau(\cT_{\varepsilon}) = \cT(\mathbf{A}_4,\omega_\tau)$ for $\tau \in \frakT(\varepsilon) = \{\mathrm{id},\tau_1\}$ where $\tau_1$ exchanges the pairs $(\langle 1 \rangle, \langle 2\rangle )$ and $(\langle 3\rangle, \langle 2,3\rangle)$ in $\hatP$.
In this case, the canonical order is $x_0$, $x_2$, $x_1$, $x_4$, $x_3$, $x_5$.
We begin with $\tau = \mathrm{id}$ and check that the canonical triangulation is obtained as a regular triangulation with the canonical order and height function $\omega_{\mathrm{id}}:\mathbf{A}_4\to \R$ given by $\omega_{\mathrm{id}}(x_i)=2^{\rho(x_i)}$.
We have
\[
\mathbf{A}_4^{\omega_{\mathrm{id}}} = \begin{blockarray}{ccccccc}
&\mathbf{v}_\emptyset & \mathbf{v}_{\langle 1\rangle} & \mathbf{v}_{\langle 2\rangle} & \mathbf{v}_{\langle 3\rangle} & \mathbf{v}_{\langle 2,3\rangle} & \mathbf{v}_{\langle 4\rangle} \\
\begin{block}{c(cccccc)}
1 & 0 & 1 & 1 & 1 & 1 & 1 \bigstrut[t] \\
2 & 0 & 0 & 1 & 0 & 1 & 1 \\
3 & 0 & 0 & 0 & 1 & 1 & 1 \\
4 & 0 & 0 & 0 & 0 & 0 & 1 \\
5 & 1 & 1 & 1 & 1 & 1 & 1 \\
6 & 1 & 4 & 2 & 16 & 8 & 32 \\
\end{block}
\end{blockarray} \; .
 \]

From $\hatP$, we know there are two simplices in the canonical triangulation $\cT_\varepsilon$ of $\mathcal{O}(Q_\varepsilon)$, namely $\sigma_1 = (\mathbf{v}_\emptyset, \mathbf{v}_{\langle 1\rangle},\mathbf{v}_{\langle 2\rangle},\mathbf{v}_{\langle 2,3\rangle},\mathbf{v}_{\langle 4\rangle})$ and $\sigma_2 = (\mathbf{v}_\emptyset, \mathbf{v}_{\langle 1\rangle},\mathbf{v}_{\langle 3\rangle},\mathbf{v}_{\langle 2,3\rangle},\mathbf{v}_{\langle 4\rangle})$. 
The only wall is $\sigma_0=(\mathbf{v}_\emptyset$, $\mathbf{v}_{\langle 1\rangle}$,$\mathbf{v}_{\langle 2,3\rangle}$,$\mathbf{v}_{\langle 4\rangle})$, and so we compute

\begin{align*}
\Psi_{\sigma_1,\mathbf{v}_{\langle 3\rangle}}(\omega_{\mathrm{id}}) &= \mathrm{sign }\det(\mathbf{A}_4|_{\sigma_1})\cdot \det(\mathbf{A}_4^{\omega_{\mathrm{id}}}|_{\sigma_1}:(\mathbf{v}_{\langle 3\rangle},1,16)) \\
&=
\mathrm{sign }\det\begin{pmatrix}
\mathbf{v}_\emptyset & \mathbf{v}_{\langle 1\rangle} & \mathbf{v}_{\langle 2\rangle} & \mathbf{v}_{\langle 2,3\rangle} & \mathbf{v}_{\langle 4\rangle} \\
1 & 1 & 1 & 1 & 1  \\
\end{pmatrix} 
\cdot \det\begin{pmatrix}
\mathbf{v}_\emptyset & \mathbf{v}_{\langle 1\rangle} & \mathbf{v}_{\langle 2\rangle} & \mathbf{v}_{\langle 2,3\rangle} & \mathbf{v}_{\langle 4\rangle} & \mathbf{v}_{\langle 3\rangle} \\
1 & 1 & 1 & 1 & 1 & 1 \\
1 & 4 & 2 & 8 & 32 & 16 \\
\end{pmatrix} \\
&= 6 > 0
\end{align*}

and similarly

\begin{align*}
\Psi_{\sigma_2,\mathbf{v}_{\langle 2\rangle}}(\omega_{\mathrm{id}}) 
&= \mathrm{sign }\det(\mathbf{A}_4|_{\sigma_2}) \cdot \det(\mathbf{A}_4^{\omega_{\mathrm{id}}}|_{\sigma_2}:(\mathbf{v}_{\langle 2\rangle},1,2)) \\
&= \mathrm{sign} \det\begin{pmatrix}
\mathbf{v}_\emptyset & \mathbf{v}_{\langle 1\rangle} & \mathbf{v}_{\langle 3\rangle} & \mathbf{v}_{\langle 2,3\rangle} & \mathbf{v}_{\langle 4\rangle} \\
1 & 1 & 1 & 1 & 1  \\
\end{pmatrix}
\cdot \det\begin{pmatrix}
\mathbf{v}_\emptyset & \mathbf{v}_{\langle 1\rangle} & \mathbf{v}_{\langle 3\rangle} & \mathbf{v}_{\langle 2,3\rangle} & \mathbf{v}_{\langle 4\rangle} & \mathbf{v}_{\langle 2\rangle} \\
1 & 1 & 1 & 1 & 1 & 1 \\
1 & 4 & 16 & 8 & 32 & 2 \\
\end{pmatrix} \\
&= (-1)(-6) > 0.
\end{align*}
Thus, by the local folding condition in Theorem \ref{thm.2.3.20}, $\mathbf{v}_{\langle 2 \rangle} = \sigma_1\setminus \sigma_2$
lies above the hyperplane containing $\mathbf{A}_4|_{\sigma_2}$, and $\mathbf{v}_{\langle 3 \rangle} = \sigma_2\setminus \sigma_1$
lies above the hyperplane containing $\mathbf{A}_4|_{\sigma_1}$. By Theorem~\ref{thm.2.3.20} we then have $\cT_\varepsilon= \cT(\mathbf{A}_4,\omega_{\mathrm{id}})$. 
Checking that $\tau_1(\cT_\varepsilon) = \cT(\mathbf{A}_4, \omega_{\tau_1})$ is done similarly. Appending the row of heights according to the twisted order determined by $\tau_1$, we have
\[
\mathbf{A}_4^{\omega_{\tau_1}} = \begin{blockarray}{(ccccccc)}
\mathbf{v}_\emptyset & \mathbf{v}_{\langle 1\rangle} & \mathbf{v}_{\langle 2\rangle} & \mathbf{v}_{\langle 3\rangle} & \mathbf{v}_{\langle 2,3\rangle} & \mathbf{v}_{\langle 4\rangle} \\
  1 & 1 & 1 & 1 & 1 & 1 \\
  1 & 2 & 4 & 8 & 16 & 32 \\
\end{blockarray}\, .
 \]

The two simplices in the triangulation $\tau_1(\cT_\varepsilon)$ of $\mathcal{O}(Q_\varepsilon)$ are $\tau_1(\sigma_1) = (\mathbf{v}_\emptyset, \mathbf{v}_{\langle 1\rangle},\mathbf{v}_{\langle 2\rangle},\mathbf{v}_{\langle 3\rangle},\mathbf{v}_{\langle 4\rangle})$ and $\tau_1(\sigma_2) = (\mathbf{v}_\emptyset, \mathbf{v}_{\langle 2\rangle},\mathbf{v}_{\langle 3\rangle},\mathbf{v}_{\langle 2,3\rangle},\mathbf{v}_{\langle 4\rangle})$. 
The wall between them is $\tau_1(\sigma_0)=(\mathbf{v}_\emptyset, \mathbf{v}_{\langle 2\rangle},\mathbf{v}_{\langle 3\rangle},\mathbf{v}_{\langle 4\rangle})$. 
We compute $\Psi_{\tau_1(\sigma_1),\mathbf{v}_{\langle 2,3\rangle}}(\omega_{\tau_1}) = 6 > 0$ and $\Psi_{\tau_1(\sigma_2),\mathbf{v}_{\langle 1\rangle}}(\omega_{\tau_1}) 
= (-1)(-6) > 0$.
Similarly, in the case that $\varepsilon$ is treated as an $L$, the reader may check that $\tau(\cT_\varepsilon) = \cT(\mathbf{A}_4,\omega_\tau)$ for $\tau \in \frakT(\varepsilon) = \{\mathrm{id}, \tau_1\}$. Hence the base case holds.

We now consider the case $\mathbf{w} = w_0 w_1\cdots w_kw_{k+1}$. 
We assume the inductive hypothesis, namely that the local folding condition holds for every wall in $\tau(\cT_{w_0 w_1\cdots w_{\ell}})$ of $\mathcal{O}(Q_{w_0 w_1\cdots w_{\ell}})$ for all $0\leq \ell \leq k$ and for all $\tau \in \frakT(w_0 w_1\cdots w_{\ell})$. 
That is, for any wall $\tau(\sigma_{0})=\tau(\sigma_{1}) \cap \tau(\sigma_{2})$ in $\cT_{w_0 w_1\cdots w_{\ell}}$, we assume that $\Psi_{\tau(\sigma_{1}),\mathbf{v}_i}(\omega_\tau) > 0$ for $\mathbf{v}_i = \tau(\sigma_{2}\setminus \sigma_{1})$ and  $\Psi_{\tau(\sigma_{2}),\mathbf{v}_j}(\omega_\tau) > 0$ for $\mathbf{v}_j = \tau(\sigma_{1}\setminus \sigma_{2})$. 
The general strategy for the remainder of the proof is as follows.
Simultaneously performing the same row and column operations on $(\mathbf{A}_{n+1}^{\omega_\tau}|_{\tau(\sigma_1)}:(\bv_i,1,\omega_\tau(\bv_i)))$ and $\mathbf{A}_{n+1}|_{\tau(\sigma_1)}$ does not change the sign of the product of their determinants (and thus the sign of $\Psi_{\tau(\sigma_1),\mathbf{v}_i}(\omega_\tau)$), so long as each operation is possible on both matrices. 
Furthermore, we add rows to or subtract rows from the last row of $(\mathbf{A}_{n+1}^{\omega_\tau}|_{\tau(\sigma_1)}:(\bv_i,1,\omega_\tau(\bv_i)))$ without changing the sign of its determinant so that $(\mathbf{A}_{n+1}^{\omega_\tau}|_{\tau(\sigma_1)}:(\bv_i,1,\omega_\tau(\bv_i)))$ contains $(\mathbf{A}_{\ell+4}^{\omega_{\tau'}}|_{\tau'(\sigma_1')}:(\bv_i',1,\omega_{\tau'}(\bv_i')))$ as a block submatrix, which is a matrix of a previous case corresponding to a word $w_0w_1\cdots w_\ell$, with $\ell < k+1$. 
Here, $\tau' \in \frakT(w_0\cdots w_\ell)$ is the restriction of $\tau$ to the ladders of $\hatP(w_0\cdots w_\ell)$, $\tau'(\sigma'_1)$ is the simplex in $\calO(Q_{w_0\cdots w_\ell})$ corresponding to the maximal chain in $\hatP(w_0\cdots w_\ell)$ which is a subset of the maximal chain in $\hatP(\bw)$ corresponding with $\tau(\sigma_1)$, and $\bv'_j$ is equal to $\bv_j$ restricted to entries $1,2,\ldots,\ell + 4$.
By the inductive hypothesis, we then have $\Psi_{\tau'(\sigma_1'),\mathbf{v}_i'}(\omega_{\tau'}) = \mathrm{sign} \det(\mathbf{A}_{\ell+4}|_{\tau'(\sigma_1')})\cdot \det(\mathbf{A}_{\ell+4}^{\omega_{\tau'}}|_{\tau'(\sigma_1')}:(\bv_i',1,\omega_{\tau'}(\bv'_i))) >0$, from which we argue that $\Psi_{\tau(\sigma_1),\mathbf{v}_i}(\omega_\tau)>0$.
Moreover, via a similar argument, we can argue that $\Psi_{\tau(\sigma_2), \bv_j}(\omega_\tau)>0$.

To apply this strategy, we consider two possibilities, extensions and turns, each of which admit multiple subcases. 
To define extensions and turns, suppose $\hatP(w_0\cdots w_k)$ consists of $t$ ladders $\mL^1,...,\mL^t$ with $\widehat{0}$ and $\widehat{1}$.
The case when $w_{k+1}=w_k$ is referred to as an \textit{extension} because it corresponds with extending the ladder ending with $w_k$, namely $\mL^t$.
Furthermore, we refer to the case when $w_{k+1}\neq w_k$ as a \textit{turn} because it corresponds to a turn in $\hatP(\bw)$. 
Note that a turn effectively adds an additional ladder $\mL^{t+1}$ to $\hatP(w_0\cdots w_k)$ to form $\hatP(\bw)$. 
In the case of an extension, we need to consider both when there is and is not a twist applied to $\mL^t$. Furthermore, in each of these cases we must consider the subcases when the simplices differ by elements in the square $\Sq(\bw_{k+1})$, and when they differ by elements in a square $\Sq(\bw_{i})$ with $i\leq k$. 

In the case of a turn, it is somewhat more complicated than for an extension, as we have two possible twists which result in four cases to check. We can either twist $\mL^t$ or the new ladder $\mL^{t+1}$ produced by the turn, or we may twist both or neither. See Figure~\ref{fig:TurnCases}. In each of these cases, there are more subcases to check than for an extension, as we must consider when the simplices differ by elements in one of the squares $\Sq(\bw_{k+1})$, $\Sq(\bw_{k})$, $\Sq(\bw_{k-1})$, and when they differ in a square $\Sq(\bw_{i})$ with $i\leq k-1$. 
Extensions and turns combined result in a total of 36 cases to check, which are listed in Appendix~\ref{sec:appendix}.

\begin{center}
\begin{figure}
\begin{tikzpicture}
\begin{scope}[xshift=-50, yshift=0, scale=0.4] 
	\node[](c01) at (-5,13) {{\small Extension:}};
\end{scope}

\begin{scope}[xshift=0, yshift=0, scale=0.4] 
	\pvx[label=left:{\tiny $a-1$}](a-1) at (0,10) {};	
	\pvx[label=left:{\tiny $a$}](a) at (-1,9) {};
	\pvx[label=right:{\tiny $a+1$}](a1) at (1,9) {};
	\pvx[label=right:{\tiny $n-1$}](k-1) at (3,7) {};
	\pvx[label=right:{\tiny $n$}](k) at (4,6) {};
	\pvx[label=right:{\tiny $n+1$}](k1) at (5,5) {};
	\pvx[label=right:{\tiny $4$}](3) at (4,4) {};

 	\draw[line width=1pt] (a1)--(a-1)--(a)--(3) -- (k1) -- (k) -- (k-1);
 	\draw[line width=1pt] (0.3,10.3)--(a);
  	\draw[line width=1pt] (1.3,9.3)--(a1);	
  	\draw[line width=1pt] (1.3,8.7)--(a1);	
  	\draw[line width=1pt] (2.7,7.3)--(k-1);		

	\node[](dd) at (2,7.75) {$\Ddots$};
	\node[](uu1) at (1.75,10) {$\udots$};	
	\node[](uu1) at (.75,11) {$\udots$};		

	\node[](qq) at (-5,8) {$Q_\mathbf{w}$};

\end{scope}

\begin{scope}[xshift=200, yshift=0, scale=0.4] 
	\pvx[label=left:{\tiny $\langle a-1\rangle$}](a-1) at (0,10) {};	
	\pvx[label=left:{\tiny $\langle a \rangle$}](a) at (-1,9) {};
	\pvx[label=right:{\tiny $\langle a-1,b \rangle$}](a1) at (1,9) {};
	\pvx[label=right:{\tiny $\langle n-1\rangle$}](k-1) at (3,7) {};
	\pvx[label=right:{\tiny $\langle n \rangle$}](k) at (4,6) {};
	\pvx[label=right:{\tiny $\langle n+1\rangle$}](k1) at (5,5) {};
	\pvx[label=left:{\tiny $\langle a,n+1 \rangle$}](3) at (4,4) {};
	\pvx[label=left:{\tiny $\langle a,b \rangle$}](ab) at (0,8) {};
	\pvx[label=left:{\tiny $\langle a,n-1 \rangle$}](ak-1) at (2,6) {};
	\pvx[label=left:{\tiny $\langle a,n \rangle$}](ak) at (3,5) {};	
	\pvx[label=left:{\tiny $\langle 4 \rangle$}](1hat) at (5,3) {};

 	\draw[line width=1pt] (a1)--(a-1)--(a) -- (0.3,7.7); 
 	\draw[line width=1pt] (1.7,6.3)--(3) -- (k1) -- (k) -- (k-1);
 	\node[](dd) at (0.9,6.75) {$\Ddots$};	
 	
 	\draw[line width=1pt] (0.3,10.3)--(a);
  	\draw[line width=1pt] (1.3,9.3)--(a1);	
  	\draw[line width=1pt] (1.3,8.7)--(a1);	
  	\draw[line width=1pt] (2.7,7.3)--(k-1);		
  	\draw[line width=1pt] (ab)--(a1);		 	
  	\draw[line width=1pt] (ak-1)--(k-1);		 	  	
  	\draw[line width=1pt] (ak)--(k);
  	\draw[line width=1pt] (3)--(1hat);

	\node[](dd) at (2,7.75) {$\Ddots$};
	\node[](uu1) at (1.75,10) {$\udots$};	
	\node[](uu1) at (0.75,11) {$\udots$};		

	\node[](qq) at (-5,8) {$\hatP(\bw)$};
\end{scope}		

\begin{scope}[xshift=-50, yshift=-130, scale=0.4] 
	\node[](c01) at (-5,13) {{\small Turn:}};
\end{scope}

\begin{scope}[xshift=0, yshift=-130, scale=0.4] 
	\pvx[label=left:{\tiny $a-1$}](a-1) at (0,10) {};	
	\pvx[label=left:{\tiny $a$}](a) at (-1,9) {};
	\pvx[label=right:{\tiny $a+1$}](a1) at (1,9) {};
	\pvx[label=right:{\tiny $n-2$}](k-1) at (3,7) {};
	\pvx[label=right:{\tiny $n-1$}](k) at (4,6) {};
	\pvx[label=right:{\tiny $n$}](k1) at (5,5) {};
	\pvx[label=right:{\tiny $4$}](3) at (4,4) {};
	\pvx[label=left:{\tiny $n+1$}](kk) at (3,5) {};	

 	\draw[line width=1pt] (a1)--(a-1)--(a)--(3) -- (k1) -- (k) -- (k-1);
  	\draw[line width=1pt] (k)--(kk);	
 	\draw[line width=1pt] (0.3,10.3)--(a);
  	\draw[line width=1pt] (1.3,9.3)--(a1);	
  	\draw[line width=1pt] (1.3,8.7)--(a1);	
  	\draw[line width=1pt] (2.7,7.3)--(k-1);		

	\node[](dd) at (2,7.75) {$\Ddots$};
	\node[](uu1) at (1.75,10) {$\udots$};	
	\node[](uu1) at (.75,11) {$\udots$};		

	\node[](qq) at (-5,8) {$Q_\mathbf{w}$};		
	
\end{scope}

\begin{scope}[xshift=200, yshift=-130, scale=0.4] 
	\pvx[label=left:{\tiny $\langle a-1\rangle$}](a-1) at (0,10) {};	
	\pvx[label=left:{\tiny $\langle a \rangle$}](a) at (-1,9) {};
	\pvx[label=right:{\tiny $\langle a-1,b \rangle$}](a1) at (1,9) {};
	\pvx[label=right:{\tiny $\langle n-2\rangle$}](k-1) at (3,7) {};
	\pvx[label=right:{\tiny $\langle n-1 \rangle$}](k) at (4,6) {};
	\pvx[label=right:{\tiny $\langle n\rangle$}](k1) at (5,5) {};
	\pvx[label=right:{\tiny $\langle a,n \rangle$}](3) at (4,4) {};
	\pvx[label=left:{\tiny $\langle a,b \rangle$}](ab) at (0,8) {};
	\pvx[label=left:{\tiny $\langle a,n-2 \rangle$}](ak-1) at (2,6) {};
	\pvx[label=left:{\tiny $\langle a,n-1 \rangle$}](ak) at (3,5) {};	
	\pvx[label=right:{\tiny $\langle n,n+1 \rangle$}](2hat) at (3,3) {};
	\pvx[label=right:{\tiny $\langle 4 \rangle$}](1hat) at (2,2) {};	
	\pvx[label=left:{\tiny $\langle n+1 \rangle$}](dfs) at (2,4) {};

 	\draw[line width=1pt] (a1)--(a-1)--(a) -- (0.3,7.7); 
 	\draw[line width=1pt] (1.7,6.3)--(3) -- (k1) -- (k) -- (k-1);
 
 	\draw[line width=1pt] (0.3,10.3)--(a);
  	\draw[line width=1pt] (1.3,9.3)--(a1);	
  	\draw[line width=1pt] (1.3,8.7)--(a1);	
  	\draw[line width=1pt] (2.7,7.3)--(k-1);		
  	\draw[line width=1pt] (ab)--(a1);		 	
  	\draw[line width=1pt] (ak-1)--(k-1);		 	  	
  	\draw[line width=1pt] (ak)--(k);
  	\draw[line width=1pt] (3)--(2hat);
  	\draw[line width=1pt] (1hat)--(2hat);
  	\draw[line width=1pt] (ak)--(dfs)--(2hat);

 	\node[](dd) at (0.9,6.75) {$\Ddots$};	
	\node[](dd) at (2,7.75) {$\Ddots$};
	\node[](uu1) at (1.75,10) {$\udots$};	
	\node[](uu1) at (0.75,11) {$\udots$};				

	\node[](qq) at (-5,8) {$\hatP(\bw)$};
\end{scope}		
\end{tikzpicture}
\label{fig:ExtentionAndTurn}
\caption{The two main cases to check, i.e. an extension and turn.}
\end{figure}
\end{center}

\begin{center}
\begin{figure}
\begin{tikzpicture}
\begin{scope}[xshift=15, yshift=0, scale=0.4] 
	\pvx[label=left:{}](a-1) at (0,10) {};	
	\pvx[label=left:{}](a) at (-1,9) {};
	\pvx[label=right:{}](a1) at (1,9) {};
	\pvx[label=right:{\tiny $\langle n-2\rangle$}](k-1) at (3,7) {};
	\pvx[label=right:{\tiny $\langle n-1 \rangle$}](k) at (4,6) {};
	\pvx[label=right:{\tiny $\langle n\rangle$}](k1) at (5,5) {};
	\pvx[label=right:{\tiny $\langle a,n \rangle$}](3) at (4,4) {};
	\pvx[label=left:{}](ab) at (0,8) {};
	\pvx[label=left:{\tiny $\langle a,n-2 \rangle$}](ak-1) at (2,6) {};
	\pvx[label=left:{\tiny $\langle a,n-1 \rangle$}](ak) at (3,5) {};	
	\pvx[label=right:{\tiny $\langle n,n+1 \rangle$}](2hat) at (3,3) {};
	\pvx[label=right:{\tiny $\langle 4 \rangle$}](1hat) at (2,2) {};	
	\pvx[label=left:{\tiny $\langle n+1 \rangle$}](dfs) at (2,4) {};

 	\draw[line width=1pt] (a1)--(a-1)--(a) -- (0.3,7.7); 
 	\draw[line width=1pt] (1.7,6.3)--(3) -- (k1) -- (k) -- (k-1);

 	\draw[line width=1pt] (0.3,10.3)--(a);
  	\draw[line width=1pt] (1.3,9.3)--(a1);	
  	\draw[line width=1pt] (1.3,8.7)--(a1);	
  	\draw[line width=1pt] (2.7,7.3)--(k-1);		
  	\draw[line width=1pt] (ab)--(a1);		 	
  	\draw[line width=1pt] (ak-1)--(k-1);		 	  	
  	\draw[line width=1pt] (ak)--(k);
  	\draw[line width=1pt] (3)--(2hat);
  	\draw[line width=1pt] (1hat)--(2hat);
  	\draw[line width=1pt] (ak)--(dfs)--(2hat);

 	\node[](dd) at (0.9,6.75) {$\Ddots$};	
	\node[](dd) at (2,7.75) {$\Ddots$};
	\node[](uu1) at (1.75,10) {$\udots$};	
	\node[](uu1) at (0.75,11) {$\udots$};		

	\node[](qq) at (2,0) {No twist};
\end{scope}		

\begin{scope}[xshift=125, yshift=0, scale=0.4] 
	\pvx[label=left:{}](a-1) at (0,10) {};	
	\pvx[label=left:{}](a) at (-1,9) {};
	\pvx[label=right:{}](a1) at (1,9) {};
	\pvx[label=right:{\tiny $\langle a,n-2\rangle$}](k-1) at (3,7) {};
	\pvx[label=right:{\tiny $\langle a,n-1 \rangle$}](k) at (4,6) {};
	\pvx[label=right:{\tiny $\langle a,n\rangle$}](k1) at (5,5) {};
	\pvx[label=right:{\tiny $\langle n \rangle$}](3) at (4,4) {};
	\pvx[label=left:{}](ab) at (0,8) {};
	\pvx[label=left:{\tiny $\langle n-2 \rangle$}](ak-1) at (2,6) {};
	\pvx[label=left:{\tiny $\langle n-1 \rangle$}](ak) at (3,5) {};	
	\pvx[label=right:{\tiny $\langle n,n+1 \rangle$}](2hat) at (3,3) {};
	\pvx[label=right:{\tiny $\langle 4 \rangle$}](1hat) at (2,2) {};	
	\pvx[label=left:{\tiny $\langle n+1 \rangle$}](dfs) at (2,4) {};

 	\draw[line width=1pt] (a1)--(a-1)--(a) -- (0.3,7.7); 
 	\draw[line width=1pt] (1.7,6.3)--(3) -- (k1) -- (k) -- (k-1);
 	
 	\draw[line width=1pt] (0.3,10.3)--(a);
  	\draw[line width=1pt] (1.3,9.3)--(a1);	
  	\draw[line width=1pt] (1.3,8.7)--(a1);	
  	\draw[line width=1pt] (2.7,7.3)--(k-1);		
  	\draw[line width=1pt] (ab)--(a1);		 	
  	\draw[line width=1pt] (ak-1)--(k-1);		 	  	
  	\draw[line width=1pt] (ak)--(k);
  	\draw[line width=1pt] (3)--(2hat);
  	\draw[line width=1pt] (1hat)--(2hat);
  	\draw[line width=1pt] (ak)--(dfs)--(2hat);

 	\node[](dd) at (0.9,6.75) {$\Ddots$};	
	\node[](dd) at (2,7.75) {$\Ddots$};
	\node[](uu1) at (1.75,10) {$\udots$};	
	\node[](uu1) at (0.75,11) {$\udots$};		

	\node[](qq) at (2,0) {Twist on $\mL^t$};
\end{scope}		
\begin{scope}[xshift=245, yshift=0, scale=0.4] 
	\pvx[label=left:{}](a-1) at (0,10) {};	
	\pvx[label=left:{}](a) at (-1,9) {};
	\pvx[label=right:{}](a1) at (1,9) {};
	\pvx[label=right:{\tiny $\langle n-2\rangle$}](k-1) at (3,7) {};
	\pvx[label=right:{\tiny $\langle n \rangle$}](k) at (4,6) {};
	\pvx[label=right:{\tiny $\langle n-1\rangle$}](k1) at (5,5) {};
	\pvx[label=right:{\tiny $\langle a,n-1 \rangle$}](3) at (4,4) {};
	\pvx[label=left:{}](ab) at (0,8) {};
	\pvx[label=left:{\tiny $\langle a,n-2 \rangle$}](ak-1) at (2,6) {};
	\pvx[label=left:{\tiny $\langle a,n \rangle$}](ak) at (3,5) {};	
	\pvx[label=right:{\tiny $\langle n+1 \rangle$}](2hat) at (3,3) {};
	\pvx[label=right:{\tiny $\langle 4 \rangle$}](1hat) at (2,2) {};	
	\pvx[label=left:{\tiny $\langle n,n+1 \rangle$}](dfs) at (2,4) {};

 	\draw[line width=1pt] (a1)--(a-1)--(a) -- (0.3,7.7); 
 	\draw[line width=1pt] (1.7,6.3)--(3) -- (k1) -- (k) -- (k-1);

 	\draw[line width=1pt] (0.3,10.3)--(a);
  	\draw[line width=1pt] (1.3,9.3)--(a1);	
  	\draw[line width=1pt] (1.3,8.7)--(a1);	
  	\draw[line width=1pt] (2.7,7.3)--(k-1);		
  	\draw[line width=1pt] (ab)--(a1);		 	
  	\draw[line width=1pt] (ak-1)--(k-1);		 	  	
  	\draw[line width=1pt] (ak)--(k);
  	\draw[line width=1pt] (3)--(2hat);
  	\draw[line width=1pt] (1hat)--(2hat);
  	\draw[line width=1pt] (ak)--(dfs)--(2hat);

 	\node[](dd) at (0.9,6.75) {$\Ddots$};	
	\node[](dd) at (2,7.75) {$\Ddots$};
	\node[](uu1) at (1.75,10) {$\udots$};	
	\node[](uu1) at (0.75,11) {$\udots$};		

	\node[](qq) at (2,0) {Twist on $\mL^{t+1}$};
\end{scope}		
\begin{scope}[xshift=360, yshift=0, scale=0.4] 
	\pvx[label=left:{}](a-1) at (0,10) {};	
	\pvx[label=left:{}](a) at (-1,9) {};
	\pvx[label=right:{}](a1) at (1,9) {};
	\pvx[label=right:{\tiny $\langle a,n-2\rangle$}](k-1) at (3,7) {};
	\pvx[label=right:{\tiny $\langle a,n \rangle$}](k) at (4,6) {};
	\pvx[label=right:{\tiny $\langle a,n-1\rangle$}](k1) at (5,5) {};
	\pvx[label=right:{\tiny $\langle n-1 \rangle$}](3) at (4,4) {};
	\pvx[label=left:{}](ab) at (0,8) {};
	\pvx[label=left:{\tiny $\langle n-2 \rangle$}](ak-1) at (2,6) {};
	\pvx[label=left:{\tiny $\langle n \rangle$}](ak) at (3,5) {};	
	\pvx[label=right:{\tiny $\langle n+1 \rangle$}](2hat) at (3,3) {};
	\pvx[label=right:{\tiny $\langle 4 \rangle$}](1hat) at (2,2) {};	
	\pvx[label=left:{\tiny $\langle n,n+1 \rangle$}](dfs) at (2,4) {};

 	\draw[line width=1pt] (a1)--(a-1)--(a) -- (0.3,7.7); 
 	\draw[line width=1pt] (1.7,6.3)--(3) -- (k1) -- (k) -- (k-1);

 	\draw[line width=1pt] (0.3,10.3)--(a);
  	\draw[line width=1pt] (1.3,9.3)--(a1);	
  	\draw[line width=1pt] (1.3,8.7)--(a1);	
  	\draw[line width=1pt] (2.7,7.3)--(k-1);		
  	\draw[line width=1pt] (ab)--(a1);		 	
  	\draw[line width=1pt] (ak-1)--(k-1);		 	  	
  	\draw[line width=1pt] (ak)--(k);
  	\draw[line width=1pt] (3)--(2hat);
  	\draw[line width=1pt] (1hat)--(2hat);
  	\draw[line width=1pt] (ak)--(dfs)--(2hat);
  	
 	\node[](dd) at (0.9,6.75) {$\Ddots$};	
	\node[](dd) at (2,7.75) {$\Ddots$};
	\node[](uu1) at (1.75,10) {$\udots$};	
	\node[](uu1) at (0.75,11) {$\udots$};		

	\node[](qq) at (2,0) {Twist on $\mL^t$ and $\mL^{t+1}$};
\end{scope}

\end{tikzpicture}
\caption{The four possible twists at a turn.}
\label{fig:TurnCases}
\end{figure}
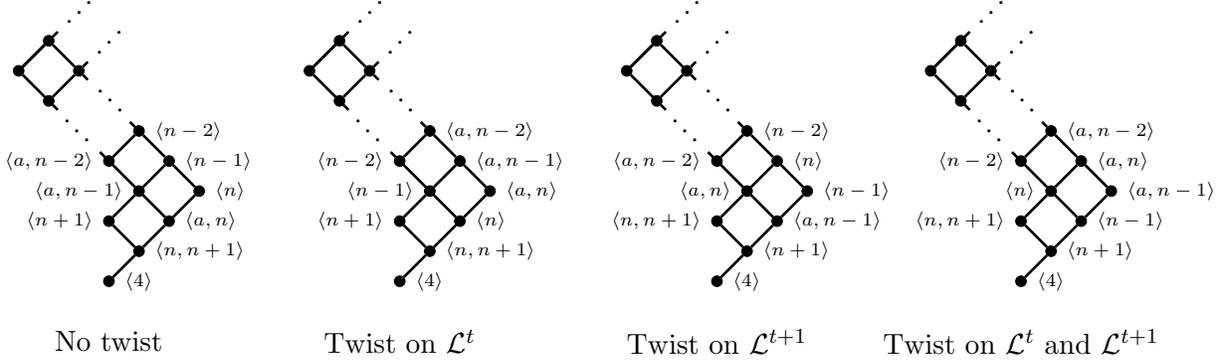
\end{center}


We will prove two of the thirty-six cases here, to demonstrate the techniques used in the remaining cases, which are checked similarly. The first case we consider here is an extension case with a twist $\tau$, where $\tau$ twists the last ladder $\mL^t$ (and possibly other ladders $\mL^i$ with $i< t$). The particular subcase chosen is the one where the simplices $\tau(\sigma_1)$ and $\tau(\sigma_2)$ differ in the final square $\Sq(\bw_{k+1})$. In this case we assume without loss of generality that $\mathbf{v}_{\langle n\rangle} \in \tau(\sigma_1)$ and $\mathbf{v}_{\langle a,n+1\rangle} \in \tau(\sigma_2)$, see Figure~\ref{fig.ExtensionSubcase}.
This is the extension case 2.c in Appendix~\ref{sec:appendix}.

\begin{figure}
\begin{center}
\begin{tikzpicture}[scale=.4]
\begin{scope}[xshift=200, yshift=0]

	\pvx[label=left:{\tiny $\langle a\rangle$}](a-1) at (0,10) {};	
	\pvx[label=left:{\tiny $\langle a-1 \rangle$}](a) at (-1,9) {};
	\pvx[label=right:{\tiny $\langle a,b \rangle$}](a1) at (1,9) {};
	\pvx[label=right:{\tiny $\langle a,n-1\rangle$}](k-1) at (3,7) {};
	\pvx[label=right:{\tiny $\langle a,n \rangle$}](k) at (4,6) {};
	\pvx[label=right:{\tiny $\langle a,n+1\rangle$}](k1) at (5,5) {};
	\pvx[label=left:{\tiny $\langle n+1 \rangle$}](3) at (4,4) {};
	\pvx[label=left:{\tiny $\langle b \rangle$}](ab) at (0,8) {};
	\pvx[label=left:{\tiny $\langle n-1 \rangle$}](ak-1) at (2,6) {};
	\pvx[label=left:{\tiny $\langle n \rangle$}](ak) at (3,5) {};	
	\pvx[label=left:{\tiny $\langle 4 \rangle$}](1hat) at (5,3) {};

 	\draw[line width=1pt] (a1)--(a-1)--(a) -- (0.3,7.7); 
 	\draw[line width=1pt] (1.7,6.3)--(3) -- (k1) -- (k) -- (k-1);
 	\node[](dd) at (0.9,6.75) {$\Ddots$};	
 	
 	\draw[line width=1pt] (0.3,10.3)--(a);
  	\draw[line width=1pt] (1.3,9.3)--(a1);	
  	\draw[line width=1pt] (1.3,8.7)--(a1);	
  	\draw[line width=1pt] (2.7,7.3)--(k-1);		
  	\draw[line width=1pt] (ab)--(a1);		 	
  	\draw[line width=1pt] (ak-1)--(k-1);		 	  	
  	\draw[line width=1pt] (ak)--(k);
  	\draw[line width=1pt] (3)--(1hat);

	\node[](dd) at (2,7.75) {$\Ddots$};
	\node[](uu1) at (1.75,10) {$\udots$};	
	\node[](uu1) at (0.75,11) {$\udots$};

	\draw[line width=2pt, color=red] (4.9,3)--(2.9,5)--(3.9,6)--(2.9,7);
	\draw[line width=2pt, color=blue] (5.1,3)--(4.1,4)--(5.1,5)--(3.1,7);
\end{scope}		

\end{tikzpicture}
\end{center}
\caption{A subcase of an extension, with $\mL^t$ twisted by $\tau$. Here $\tau(\sigma_1)$ is in red, and $\tau(\sigma_2)$ is in blue.}
\label{fig.ExtensionSubcase}
\end{figure}
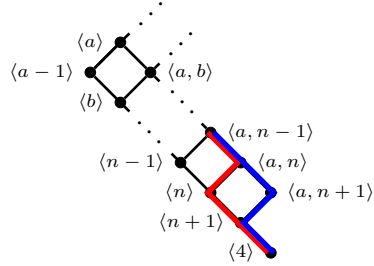

We need to show that $\Psi_{\tau(\sigma_1),\mathbf{v}_{\langle a,n+1\rangle }}(\omega_{\tau}) > 0$ and $\Psi_{\tau(\sigma_2),\mathbf{v}_{\langle n \rangle}}(\omega_{\tau}) > 0$, where $\tau\in \frakT(w_0w_1\cdots w_{k+1})$ twists $\mL^t$. We have $(\mathbf{A}_{n+1}^{\omega_\tau}|_{\tau(\sigma_1)}:(\mathbf{v}_{\langle a, n+1\rangle},1,2^{2n-4}))$ is equal to
\[
\begin{blockarray}{ccccccccc}
&\mathbf{v}_\emptyset & \cdots  & \mathbf{v}_{\langle a,n-1 \rangle} & \mathbf{v}_{\langle n \rangle} & \mathbf{v}_{\langle a,n\rangle} & \mathbf{v}_{\langle n+1\rangle} &\mathbf{v}_{\langle 4\rangle} & \mathbf{v}_{\langle a,n+1\rangle} \\
\begin{block}{c(cccccccc)}
1 & 0 &  \cdots & 1 & 1 & 1 & 1 & 1 & 1 \bigstrut[t] \\
2 & 0 & \cdots & 1 & * & 1 & * & 1 & 1 \\
3 & 0 & \cdots & 1 & * & 1 & * & 1 & 1 \\
4 & 0   & \cdots & 0 & 0 & 0 & 0 & 1 & 0 \\
\vdots & \vdots & & \vdots  & \vdots & \vdots & \vdots & \vdots & \vdots \\
a & 0 & \cdots & 1 & 0 & 1 & 0 & 1 & 1 \\
\vdots  & \vdots &   & \vdots & \vdots & \vdots & \vdots & \vdots & \vdots  \\
n & 0  & \cdots & 0 & 1 & 1 & 1 & 1 & 1 \\
n+1 & 0  & \cdots & 0 & 0 & 0 & 1 & 1 & 1 \\
n+2 & 1   & \cdots & 1 & 1 & 1 & 1 & 1 & 1\\
n+3 & 1  & \cdots & 2^{2n-8} & 2^{2n-7} & 2^{2n-6} & 2^{2n-5} & 2^{2n-3} & 2^{2n-4} \\
\end{block}
\end{blockarray} \; ,
 \]
where the $*$ entries are either zeros or ones. In the case that $\hatP(\bw)$ is not a single ladder, all the $*$ entries are ones. In the case that $\hatP(\bw)$ is a single ladder, then either $a=2$ or $a=3$, so either the $*$ entries in row $2$ are both $1$ while the $*$ entries of row $3$ are both $0$, or vice versa. We perform the following operations. 
First subtract column $\bv_{\langle n+1\rangle}$ from $\bv_{\langle a, n+1\rangle}$, then subtract column $\bv_{\langle n\rangle}$ from column $\bv_{\langle n+1\rangle}$. 
Then subtract row $4$ from row $n+1$. Then swap rows $n+1$ and $n+2$, and swap columns $\bv_{\langle n+1\rangle}$ and $\bv_{\langle 4 \rangle}$.

This yields the matrix
\[
\begin{blockarray}{ccccccccc}
&\mathbf{v}_\emptyset & \cdots  & \mathbf{v}_{\langle a,n-1 \rangle} & \mathbf{v}_{\langle n \rangle} & \mathbf{v}_{\langle a,n\rangle} & \mathbf{v}_{\langle 4\rangle} & \mathbf{v}_{\langle n+1\rangle}  & \mathbf{v}_{\langle a,n+1\rangle} \\
\begin{block}{c(cccccccc)}
1 & 0 &  \cdots & 1 & 1 & 1 & 1 & 0 & 0 \bigstrut[t] \\
2 & 0 & \cdots & 1 & * & 1 & 1 & 0 & 0 \\
3 & 0 & \cdots & 1 & * & 1 & 1 & 0 & 0 \\
4 & 0   & \cdots & 0 & 0 & 0 & 1 & 0 & 0 \\
\vdots & \vdots & & \vdots  & \vdots & \vdots & \vdots & \vdots & \vdots \\
a & 0 & \cdots & 1 & 0 & 1 & 1 & 0 & 0 \\
\vdots  & \vdots &   & \vdots & \vdots & \vdots & \vdots & \vdots & \vdots  \\
n & 0  & \cdots & 0 & 1 & 1 & 1 & 0 & 0 \\
n+2 & 1   & \cdots & 1 & 1 & 1 & 1 & 0 & 0\\
n+1 & 0  & \cdots & 0 & 0 & 0 & 0 & 1 & 1 \\
n+3 & 1  & \cdots & 2^{2n-8} & 2^{2n-7} & 2^{2n-6} & 2^{2n-3} & 2^{2n-5}-2^{2n-7}  & 2^{2n-4}-2^{2n-6} \\
\end{block}
\end{blockarray}  
\] 

The first $n+1$ rows of columns $\mathbf{v}_\emptyset,\ldots,\mathbf{v}_{\langle 4\rangle}$ now correspond to a simplex $\sigma'$ within $\tau(\sigma_1)\cap \tau(\sigma_2)$, and hence are linearly independent. 
Therefore, we can use row operations to transform it into $I_{n+1}$. Note that performing these operations does not alter the rest of the matrix, as the first $n+1$ rows of the last two columns are $0$. After obtaining $I_{n+1}$ in the upper left corner, the first $n+1$ entries in row $n+3$ can be zeroed out using the ones in $I_{n+1}$. 
As a result, the above matrix simplifies to $\begin{pmatrix}
I_{n+1} & \mathbf{0}_{(n1)\times 2} \\
\mathbf{0}_{2\times (n+1)} & B
\end{pmatrix}$, where 

\[
B = 
\begin{pmatrix}
1 & 1 \\
2^{2n-5} - 2^{2n-7} & 2^{2n-4} - 2^{2n-6} 
\end{pmatrix}.
\]
Performing all the previous row and column operations on $\mathbf{A}_{n} |_{\sigma'}$ (when possible) and $(\mathbf{A}_{n+1}^{\omega_\tau}|_{\tau(\sigma_1)}:(\mathbf{v}_{\langle a, n+1\rangle},1,2^{2n-4}))$ does not change the sign of $\Psi_{\tau(\sigma_1),\mathbf{v}_{\langle a,n+1\rangle}}( \omega_{\tau})$.
Note that the only operations done above which change the sign of the determinant were row and column swaps.
One column swap and one row swap were used initially on $(\mathbf{A}_{n+1}^{\omega_\tau}|_{\tau(\sigma_1)}:(\mathbf{v}_{\langle a, n+1\rangle},1,2^{2n-4}))$, thus their contributions to the sign of the determinant cancel.
Any row swaps arising in the row reduction to the identity matrix for the upper-left block occur for both of these matrices, and thus their contributions to the sign cancel.
Therefore, since $\det(B) = 2^{2n-4} - 2^{2n-6} - 2^{2n-5} + 2^{2n-7} > 0$, this shows that $\Psi_{\tau(\sigma_1),\mathbf{v}_{\langle a,n+1\rangle}}(\omega_{\tau}) > 0$.
Similarly, we can show that $\Psi_{\tau(\sigma_2),\mathbf{v}_{\langle n\rangle}}(\omega_{\tau}) > 0$.

The case above did not require the induction hypothesis, so the next case demonstrates a situation where it is needed. 
We consider the subcase of a turn where both $\mL^t$ and $\mL^{t+1}$ are twisted, with $\mathbf{v}_{\langle n+1\rangle}$, $\mathbf{v}_{\langle n-1\rangle}$, $\mathbf{v}_{\langle n\rangle}$, and
$\mathbf{v}_{\langle n-2\rangle}$
all contained in $\tau(\sigma_1)\cap \tau(\sigma_2)$.
See Figure~{\ref{fig:TurnSubcase}}. This is case 4.b.i in the list of cases provided in Appendix~\ref{sec:appendix}.

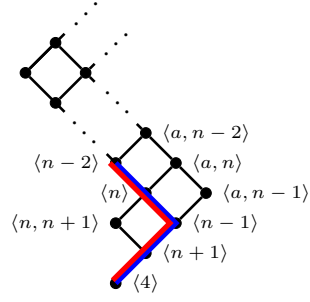
\begin{figure}
\centering
\begin{tikzpicture}
\begin{scope}[xshift=0, yshift=0, scale=0.4] 
	\pvx[label=left:{}](a-1) at (0,10) {};	
	\pvx[label=left:{}](a) at (-1,9) {};
	\pvx[label=right:{}](a1) at (1,9) {};
	\pvx[label=right:{\tiny $\langle a,n-2\rangle$}](k-1) at (3,7) {};
	\pvx[label=right:{\tiny $\langle a,n \rangle$}](k) at (4,6) {};
	\pvx[label=right:{\tiny $\langle a,n-1\rangle$}](k1) at (5,5) {};
	\pvx[label=right:{\tiny $\langle n-1 \rangle$}](3) at (4,4) {};
	\pvx[label=left:{}](ab) at (0,8) {};
	\pvx[label=left:{\tiny $\langle n-2 \rangle$}](ak-1) at (2,6) {};
	\pvx[label=left:{\tiny $\langle n \rangle$}](ak) at (3,5) {};	
	\pvx[label=right:{\tiny $\langle n+1 \rangle$}](2hat) at (3,3) {};
	\pvx[label=right:{\tiny $\langle 4 \rangle$}](1hat) at (2,2) {};	
	\pvx[label=left:{\tiny $\langle n,n+1 \rangle$}](dfs) at (2,4) {};

 	\draw[line width=1pt] (a1)--(a-1)--(a) -- (0.3,7.7); 
 	\draw[line width=1pt] (1.7,6.3)--(3) -- (k1) -- (k) -- (k-1);

 	\draw[line width=1pt] (0.3,10.3)--(a);
  	\draw[line width=1pt] (1.3,9.3)--(a1);	
  	\draw[line width=1pt] (1.3,8.7)--(a1);	
  	\draw[line width=1pt] (2.7,7.3)--(k-1);		
  	\draw[line width=1pt] (ab)--(a1);		 	
  	\draw[line width=1pt] (ak-1)--(k-1);		 	  	
  	\draw[line width=1pt] (ak)--(k);
  	\draw[line width=1pt] (3)--(2hat);
  	\draw[line width=1pt] (1hat)--(2hat);
  	\draw[line width=1pt] (ak)--(dfs)--(2hat);
  	
 	\node[](dd) at (0.9,6.75) {$\Ddots$};	
	\node[](dd) at (2,7.75) {$\Ddots$};
	\node[](uu1) at (1.75,10) {$\udots$};	
	\node[](uu1) at (0.75,11) {$\udots$};		
	
	\draw[line width=2pt, color=blue] (2,2)--(4,4)--(2,6);
	\draw[line width=2pt, color=red] (1.9,2.1)--(3.8,4)--(1.8,6);
	
\end{scope}		

\end{tikzpicture}
\caption{A subcase of a turn with both $\mL^t$ and $\mL^{t+1}$ twisted by $\tau$. Here $\tau(\sigma_1)$ is in red and $\tau(\sigma_2)$ is in blue.}
\label{fig:TurnSubcase}
\end{figure}

Let $\mathbf{v}_i = \tau(\sigma_1)\setminus \tau(\sigma_2)$, $\mathbf{v}_j = \tau(\sigma_2)\setminus \tau(\sigma_1)$, and $\tau\in \frakT(w_0w_1\cdots w_{k+1})$ twists $\mL^t$ and $\mL^{t+1}$.
We need to show that $\Psi_{\tau(\sigma_1),\mathbf{v}_j}(\omega_\tau) > 0$ (and $\Psi_{\tau(\sigma_2),\mathbf{v}_i}(\omega_\tau ) > 0$). 
We have that $(\mathbf{A}_{n+1}^{\omega_\tau}|_{\tau(\sigma_1)}:(\mathbf{v}_j,1,2^{\rho(\bv_j)}))$ is equal to
\[
\begin{blockarray}{cccccccccc}
&\mathbf{v}_\emptyset & \cdots & \mathbf{v}_i & \cdots & \mathbf{v}_{\langle n-1 \rangle} & \mathbf{v}_{\langle n\rangle} & \mathbf{v}_{\langle n+1\rangle} &\mathbf{v}_{\langle 4\rangle} & \mathbf{v}_j \\
\begin{block}{c(ccccccccc)}
1 & 0 & \cdots & 1 & \cdots & 1 & 1 & 1 & 1 & 1 \bigstrut[t] \\
2 & 0 & \cdots & * & \cdots & * & * & 1 & 1 & * \\
3 & 0 & \cdots & * & \cdots & * & * & 1 & 1 & * \\
4 & 0 & \cdots & 0 & \cdots & 0 & 0 & 0 & 1 & 0 \\
\vdots & \vdots &  & \vdots &  & \vdots & \vdots & \vdots & \vdots & \vdots \\
a & 0 & \cdots & 0  & \cdots & 0 & 0 & 0 & 1 & 0 \\
\vdots  & \vdots &  & \vdots &  & \vdots & \vdots & \vdots & \vdots & \vdots  \\
n-1 & 0 & \cdots & 0 & \cdots & 1 & 1 & 1 & 1 & 0 \\
n & 0 & \cdots & 0 & \cdots & 0 & 1 & 1 & 1 & 0 \\
n+1 & 0 & \cdots & 0 & \cdots & 0 & 0 & 1 & 1 & 0 \\
n+2 & 1 & \cdots & 1 & \cdots & 1 & 1 & 1 & 1 & 1\\
n+3 & 1 & \cdots & 2^{\rho(\bv_i)} & \cdots & 2^{2n-7} & 2^{2n-9} & 2^{2n-5} & 2^{2n-3} & 2^{\rho(\bv_j)} \\
\end{block}
\end{blockarray} \; .
 \]

We perform the following operations. 
Let $R_i$ denote the $i$-th row of the above matrix.
First replace $R_{n+3}$ with $R_{n+3}-(2^{2n-7}-2^{2n-9})(R_{n-1}-R_n)$, and then replace $R_{n+3}$ with $R_{n+3} - (2^{2n-9}-2^{2n-7})(R_n-R_{n+1})$. 
The effect of these two row operations is to exchange the location of $2^{2n-7}$ and $2^{2n-9}$ in the last row of the matrix above.
Then replace $R_{n+1}$ with $R_{n+1}-R_4$ and $R_{n+3}$ with  $R_{n+3}-(2^{2n-3}-2^{2n-5})R_4$.
Next, use the new $R_{n+1}$ to zero out all other entries in column $\mathbf{v}_{\langle n+1\rangle}$.
One effect of these row operations is to replace the $2^{2n-3}$ by $2^{2n-5}$ and zero out the $2^{2n-5}$ entry in column $\bv_{\langle n+1 \rangle}$.
Swap rows $n+1$ and $n+2$, and swap columns $\mathbf{v}_{\langle 4\rangle}$ and $\mathbf{v}_{\langle n+1\rangle}$. This yields
\[
\begin{blockarray}{cccccccccc}
&\mathbf{v}_\emptyset & \cdots & \mathbf{v}_i & \cdots & \mathbf{v}_{\langle n-1 \rangle} & \mathbf{v}_{\langle n\rangle} &\mathbf{v}_{\langle 4\rangle}  & \mathbf{v}_{\langle n+1\rangle} & \mathbf{v}_j  \\
\begin{block}{c(ccccccccc)}
1 & 0 & \cdots & 1 & \cdots & 1 & 1 & 1 & 0 & 1 \bigstrut[t] \\
2 & 0 & \cdots & * & \cdots & * & * & 1 & 0 & * \\
3 & 0 & \cdots & * & \cdots & * & * & 1 & 0 & * \\
4 & 0 & \cdots & 0 & \cdots & 0 & 0 & 1 & 0 & 0 \\
\vdots & \vdots &  & \vdots &  & \vdots & \vdots & \vdots & \vdots & \vdots \\
a & 0 & \cdots & 0  & \cdots & 0 & 0 & 1 & 0 & 0 \\
\vdots  & \vdots &  & \vdots &  & \vdots & \vdots & \vdots & \vdots & \vdots  \\
n-1 & 0 & \cdots & 0 & \cdots & 1 & 1 & 1 & 0 & 0 \\
n & 0 & \cdots & 0 & \cdots & 0 & 1 & 1 & 0 & 0 \\
n+2 & 1 & \cdots & 1 & \cdots & 1 & 1 & 1 & 0 & 1\\
n+1 & 0 & \cdots & 0 & \cdots & 0 & 0 & 0 & 1 & 0 \\
n+3 & 1 & \cdots & 2^{\rho(\bv_i)} & \cdots & 2^{2n-9} & 2^{2n-7} & 2^{2n-5} & 0 & 2^{\rho(\bv_j)} \\
\end{block}
\end{blockarray}\; .
 \]

Performing the same row operations on $\mathbf{A}_{n+1}|_{\tau(\sigma_1)}$ as we did on $(\mathbf{A}_{n+1}^{\omega_\tau}|_{\tau(\sigma_1)}:(\mathbf{v}_j,1,2^{\rho(\bv_j)}))$ preserves the sign of $\Psi_{\tau(\sigma_1),\mathbf{v}_j}(\omega_\tau )$.
The function $\mathrm{sign} (\Psi_{\tau(\sigma_1),\mathbf{v}_j}(\omega_\tau ))$ is a sign of a product of two terms: first, the $\mathrm{sign} \det$ applied to the matrix obtained by deleting the last row and last column in the matrix above, and second the determinant of the full matrix above.
To compute the sign of the determinant of the full matrix, we swap the last two rows and last two columns only in the matrix above, which will clarify the induction step. 
This also preserves the sign of $\Psi_{\tau(\sigma_1),\mathbf{v}_j}(\omega_\tau )$ as the contributed signs to the determinant cancel. We obtain precisely 
$$
\begin{pmatrix}
(\mathbf{A}_n^{\omega_{\tau'}}|_{\tau'(\sigma_1')}:(\mathbf{v}_j',1,2^{\rho(\bv_j')})) & \mathbf{0} \\
\mathbf{0}^T & 1 \\
\end{pmatrix}.
$$ 
Here, $\tau_{t+1}\tau' = \tau$ where $\tau'\in \frakT(w_0\cdots w_k)$, $\tau'(\sigma'_1)$ is the simplex in $\calO(Q_{w_0\cdots w_k})$ arising from the maximal chain in $\hatP(w_0\cdots w_k)$, which is a subset of the maximal chain in $\hatP(\bw)$ corresponding with $\tau(\sigma_1)$, and $\bv'_j$ is equal to $\bv_j$ restricted to entries $1,2,\ldots,n$. 

This shows that 
\begin{align*}
\mathrm{sign} (\Psi_{\tau(\sigma_1),\mathbf{v}_j}(\omega_\tau )) &= \mathrm{sign}\left( \mathrm{sign}\det
\begin{pmatrix}
\mathbf{A}_n|_{\tau'(\sigma_1')} & \mathbf{0} \\
\mathbf{0}^T & 1 \\
\end{pmatrix} \cdot 
\det 
\begin{pmatrix}
(\mathbf{A}_n^{\omega_{\tau'}}|_{\tau'(\sigma_1')}:(\mathbf{v}_j',1,2^{\rho(\bv_j')})) & \mathbf{0} \\
\mathbf{0}^T & 1  \\
\end{pmatrix}
\right) \\
&=\mathrm{sign}
\left( 
\Psi_{\tau'(\sigma_1'),\mathbf{v}_j'}(\omega_{\tau'} )
\right). 
\end{align*}
By the inductive hypothesis, $\Psi_{\tau'(\sigma_1'),\mathbf{v}_j'}(\omega_{\tau'} ) > 0$, and so $\Psi_{\tau(\sigma_1),\mathbf{v}_j}(\omega_\tau )>0$ as desired. Similarly, one checks that $\Psi_{\tau(\sigma_2),\mathbf{v}_i}(\omega_\tau ) > 0$. 
\end{proof}

\begin{corollary}
For $\bw \in \calV$, the component of the flip graph of  $\calO(Q_\bw)$ containing all regular triangulations admits a $\mathbb{Z}_2^t$ action given by twists.
\end{corollary}

\begin{proof}
By Theorem~\ref{thm:regularity},
a twist of a canonical triangulation is a regular triangulation of  $\calO(Q_\bw)$, so both triangulations lie in the same connected component of the flip graph. 
Any two triangulations in this component are connected by a sequence of flips, hence Theorem~\ref{thm:square} allows us to extend the action of twists on all triangulations in this component.  
Moreover, this action respects the edge structure of the flip graph. 
This implies that this component admits a $\mathbb{Z}_2^t$ action given by twists as claimed.
\end{proof}

\section{Future Directions}\label{sec:future}

We conclude with several conjectures. 
Throughout this article, we studied generalized snake posets $\hatP(\bw)$, with our main goal being to study the secondary polytope of $\calO(Q_{\bw})$. 

In Theorem~\ref{thm.Cayleygraph}, when $\bw=\varepsilon L^{n-1}$ and $\hatP(\bw)$ is the $n$-ladder, we saw that the $1$-skeleton of the secondary polytope of $\calO(Q_{\bw})$ is the Cayley graph of $\mathfrak{S}_{n+1}$, which is an $n$-regular graph.
Furthermore, Corollary~\ref{cor:kFlips} showed that each triangulation which is generated by applying twists to the canonical triangulation admits the same number of flips as the canonical triangulation.
In light of these results, along with computational evidence, we conjecture that the secondary polytope of $\mathcal{O}(Q_\mathbf{w})$ is simple, and the degree of each vertex is equal to the dimension of the secondary polytope. 

\begin{conjecture}
For $\bw \in \calV$, the flip graph of regular triangulations for $\calO(Q_{\bw})$ is $k$-regular, where $k$ is the dimension of the secondary polytope of $\calO(Q_{\bw})$.
\end{conjecture}

In the case when $\hatP(\bw)$ is the $n$-ladder and its secondary polytope is a permutohedron, Lemma~\ref{lem.3} implies that the dual graph of every triangulation of $\mathcal{O}(Q_{\mathbf{w}})$ is the same as the dual graph of the canonical triangulation. 
When $\hatP(\bw)$ contains a turn, our computations support the following conjecture.
\begin{conjecture}
If $J(Q_{\bw}) = \hatP(\bw)$ contains a turn, then $\calO(Q_{\bw})$ has a regular triangulation whose dual graph is not isomorphic to the dual graph of the canonical triangulation.
\end{conjecture}
Naturally, the next case to study in-depth is when $\hatP(\bw)$ is nearly a ladder.
We have verified the following conjecture for $n=3,4,5,6,7$.
\begin{conjecture}
If $J(Q_{\bw})=\hatP(\bw)$, where $\bw=\varepsilon L R^{n-2}$ for $n\geq3$, then the number of triangulations of $\calO(Q_{\bw})$ whose dual graph is isomorphic to the dual graph of the canonical triangulation is $4n(n-2)!$ .
\end{conjecture}

We know that, by Theorem~\ref{thm:unimodular}, all triangulations of $\calO(Q_\bw)$ are unimodular. 
Moreover, all of our computations support the following conjecture.
\begin{conjecture}
If $\bw \in \calV$, all triangulations of $\calO(Q_\bw)$ are regular. 
\end{conjecture}

 When $\bw \in \calV$, a twist of a canonical triangulation of $\calO(Q_{\bw})$ again yields a regular triangulation, by Theorem~\ref{thm:regularity}. 
 Therefore, if the above conjecture holds, we obtain an action of the twist group on the set of all (regular) triangulations.  
 Hence, the number of triangulations would be divisible by the order of the twist group. In the special case when $Q_{\bw}=S_n$ the twist group has order $2^{n+1}$. 
 We make the following conjecture about the precise number of regular triangulations of $\calO(S_n)$ where there appears to be a relationship between regular triangulations of $\calO(S_n)$ and odd Catalan numbers. 
 We have verified this conjecture for $n=1,2,3$.
\begin{conjecture}
The number of regular triangulations of $\calO(S_n)$  is $2^{n+1} \cdot\mathrm{Cat}(2n+1).$ 
\end{conjecture}



\bibliographystyle{plain}
\bibliography{Bibliography}

\addresseshere

\newpage
\input{appendix.tex}

\end{document}

%% file: appendix.tex
\begin{multicols}{2}
[
\section{Appendix} 
The cases to check for the proof of Theorem~\ref{thm:regularity}. 
For each of the cases below, one must check that both $\Psi_{\tau(\sigma_1),\mathbf{v}_i}(\omega_\tau) > 0$ and $\Psi_{\tau(\sigma_2),\mathbf{v}_j}(\omega_\tau) > 0$. 
]
\label{sec:appendix}
\noindent \textbf{The Extension cases:}
\begin{itemize}
\item[1.] No twist on $\mL^t$ by $\tau$.  
\begin{itemize}
\item[(a)] $\mathbf{v}_{\langle n+1\rangle},\mathbf{v}_{\langle a,n+1\rangle}\in \tau(\sigma_1 \cap \sigma_2)$.
\item[(b)] $\mathbf{v}_{\langle a,n\rangle},\mathbf{v}_{\langle a,n+1\rangle}\in \tau(\sigma_1 \cap \sigma_2)$ .
\item[(c)] $\mathbf{v}_{\langle a,n\rangle} \in \tau(\sigma_1),\mathbf{v}_{\langle n+1\rangle} \in \tau(\sigma_2)$.
\end{itemize}

\item[2.] $\mL^t$ is twisted by $\tau$.
\begin{itemize}
\item[(a)] $\mathbf{v}_{\langle n+1\rangle},\mathbf{v}_{\langle a,n+1\rangle}\in \tau(\sigma_1 \cap \sigma_2)$.
\item[(b)] $\mathbf{v}_{\langle n\rangle},\mathbf{v}_{\langle n+1\rangle}\in \tau(\sigma_1 \cap \sigma_2)$ .
\item[(c)] $\mathbf{v}_{\langle n\rangle} \in \tau(\sigma_1), \mathbf{v}_{\langle a,n+1\rangle}\in \tau(\sigma_2)$.
\end{itemize}
\end{itemize}
\vspace{0.5cm}
\noindent\textbf{ The Turn cases:}

\begin{itemize}
\item[1.] No twist on $\mL^t$ or $\mL^{t+1}$ by $\tau$.
\begin{itemize}
\item[(a)] $\mathbf{v}_{\langle n,n+1\rangle}$,\hspace{-0.1cm} $\mathbf{v}_{\langle n+1\rangle}$, \hspace{-0.17cm} $\mathbf{v}_{\langle a,n-1\rangle}\in \tau(\sigma_1 \cap \sigma_2)$.
\item[(b)] $\mathbf{v}_{\langle n,n+1\rangle}$,\hspace{-0.1cm} $\mathbf{v}_{\langle a,n\rangle},\mathbf{v}_{\langle a,n-1\rangle}\in \tau(\sigma_1 \cap \sigma_2)$.
\item[(c)] $\mathbf{v}_{\langle n,n+1\rangle}$, $\mathbf{v}_{\langle a,n\rangle},\mathbf{v}_{\langle n\rangle}\in \tau(\sigma_1 \cap \sigma_2)$.
\item[(d)] $\mathbf{v}_{\langle n+1\rangle} \in \tau(\sigma_1),\mathbf{v}_{\langle a,n\rangle} \in \tau(\sigma_2)$.
\item[(e)] $\mathbf{v}_{\langle a,n-1\rangle} \in \tau(\sigma_1),\mathbf{v}_{\langle n\rangle} \in \tau(\sigma_2)$.
\end{itemize}

\item[2.] $\mL^t$ is twisted by $\tau$, but $\mL^{t+1}$ is not. 
\begin{itemize}
\item[(a)] $\mathbf{v}_{\langle n,n+1\rangle}$, $\mathbf{v}_{\langle n+1\rangle},\mathbf{v}_{\langle n-1\rangle}\in \tau(\sigma_1 \cap \sigma_2)$.
\item[(b)] $\mathbf{v}_{\langle n,n+1\rangle}$, $\mathbf{v}_{\langle n\rangle},\mathbf{v}_{\langle n-1\rangle}\in \tau(\sigma_1 \cap \sigma_2)$.
\item[(c)] $\mathbf{v}_{\langle n,n+1\rangle}$, $\mathbf{v}_{\langle n\rangle},\mathbf{v}_{\langle a,n\rangle}\in \tau(\sigma_1 \cap \sigma_2)$.
\item[(d)] $\mathbf{v}_{\langle n+1\rangle} \in \tau(\sigma_1),\mathbf{v}_{\langle n\rangle} \in \tau(\sigma_2)$.
\item[(e)] $\mathbf{v}_{\langle n-1\rangle} \in \tau(\sigma_1),\mathbf{v}_{\langle a,n\rangle} \in \tau(\sigma_2)$.
\vspace{2cm}
\end{itemize}

\item[3.] $\mL^{t+1}$ is twisted by $\tau$, but $\mL^t$ is not. 
\begin{itemize}
\item[(a)] $\mathbf{v}_{\langle n+1\rangle}$, $\mathbf{v}_{\langle n,n+1\rangle},\mathbf{v}_{\langle a,n\rangle}\in \tau(\sigma_1 \cap \sigma_2)$.
\begin{itemize}
\item[i.] $\mathbf{v}_{\langle a, n-2\rangle}\in \tau(\sigma_1 \cap \sigma_2)$.
\item[ii.] $\mathbf{v}_{\langle n\rangle}\in \tau(\sigma_1 \cap \sigma_2)$.
\item[iii.] $\mathbf{v}_{\langle a, n-2\rangle} \in \tau(\sigma_1),\mathbf{v}_{\langle n\rangle} \in \tau(\sigma_2)$.
\end{itemize}
\item[(b)] $\mathbf{v}_{\langle n+1\rangle}$, $\mathbf{v}_{\langle a,n-1\rangle},\mathbf{v}_{\langle a,n\rangle}\in \tau(\sigma_1  \cap \sigma_2)$.
\begin{itemize}
\item[i.] $\mathbf{v}_{\langle a, n-2\rangle}\in \tau(\sigma_1  \cap \sigma_2)$.
\item[ii.] $\mathbf{v}_{\langle n\rangle}\in \tau(\sigma_1  \cap  \sigma_2)$.
\item[iii.] $\mathbf{v}_{\langle a, n-2\rangle} \in \tau(\sigma_1),\mathbf{v}_{\langle n\rangle} \in \tau(\sigma_2)$. 
\end{itemize}
\item[(c)] $\mathbf{v}_{\langle n+1\rangle}$, $\mathbf{v}_{\langle a,n-1\rangle},\mathbf{v}_{\langle n-1\rangle}\in \tau(\sigma_1  \cap \sigma_2)$.
\item[(d)] $\mathbf{v}_{\langle n,n+1\rangle} \in \tau(\sigma_1),\mathbf{v}_{\langle a,n-1\rangle} \in \tau(\sigma_2)$.
\item[(e)] $\mathbf{v}_{\langle a,n\rangle} \in \tau(\sigma_1),\mathbf{v}_{\langle n-1\rangle} \in \tau(\sigma_2)$.
\begin{itemize}
    \item[i.] $\mathbf{v}_{\langle a,n-2\rangle} \in \tau(\sigma_1 \cap  \sigma_2)$.
    \item[ii.] $\mathbf{v}_{\langle n\rangle} \in \tau(\sigma_1 \cap \sigma_2)$. 
\end{itemize}
\end{itemize}

\item[4.] $\mL^t,\mL^{t+1}$ are both twisted by $\tau$. 
\begin{itemize}
\item[(a)] $\mathbf{v}_{\langle n+1\rangle}$, $\mathbf{v}_{\langle n,n+1\rangle},\mathbf{v}_{\langle n\rangle}\in \tau(\sigma_1  \cap  \sigma_2)$.
\begin{itemize}
\item[i.] $\mathbf{v}_{\langle n-2\rangle}\in \tau(\sigma_1  \cap  \sigma_2)$.
\item[ii.] $\mathbf{v}_{\langle a,n\rangle}\in \tau(\sigma_1  \cap  \sigma_2)$.
\item[iii.] $\mathbf{v}_{\langle  n-2\rangle} \in \tau(\sigma_1),\mathbf{v}_{\langle a,n\rangle} \in \tau(\sigma_2)$. 
\end{itemize}
\item[(b)] $\mathbf{v}_{\langle n+1\rangle}$, $\mathbf{v}_{\langle n-1\rangle},\mathbf{v}_{\langle n\rangle}\in \tau(\sigma_1  \cap  \sigma_2)$.
\begin{itemize}
\item[i.] $\mathbf{v}_{\langle  n-2\rangle}\in \tau(\sigma_1  \cap  \sigma_2)$.
\item[ii.] $\mathbf{v}_{\langle a, n\rangle}\in \tau(\sigma_1  \cap  \sigma_2)$.
\item[iii.] $\mathbf{v}_{\langle  n-2\rangle} \in \tau(\sigma_1),\mathbf{v}_{\langle a,n\rangle} \in \tau(\sigma_2)$.
\end{itemize}
\item[(c)] $\mathbf{v}_{\langle n+1\rangle}$, $\mathbf{v}_{\langle n-1\rangle},\mathbf{v}_{\langle a,n-1\rangle}\in \tau(\sigma_1  \cap \sigma_2)$.
\item[(d)] $\mathbf{v}_{\langle n\rangle} \in \tau(\sigma_1),\mathbf{v}_{\langle a,n-1\rangle} \in \tau(\sigma_2)$.
\item[(e)] $\mathbf{v}_{\langle n,n+1\rangle} \in \tau(\sigma_1),\mathbf{v}_{\langle n-1\rangle} \in \tau(\sigma_2)$.
\begin{itemize}
    \item[i.] $\mathbf{v}_{\langle n-2\rangle} \in \tau(\sigma_1  \cap \sigma_2)$.
    \item[ii.] $\mathbf{v}_{\langle a,n\rangle} \in \tau(\sigma_1  \cap  \sigma_2)$.
\end{itemize}
\end{itemize}

\end{itemize}
\end{multicols}

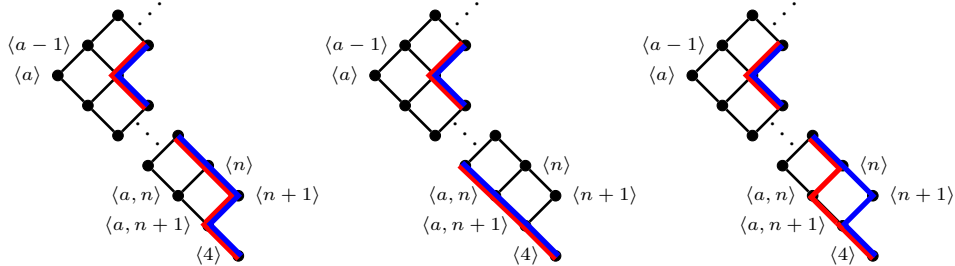
\begin{figure}[ht!]
\begin{center}
\begin{tikzpicture}[scale=.4]
\begin{scope}[xshift=0, yshift=0, yscale=1,xscale=-1] 
	
	\pvx[label=left:{\tiny $\langle a-1\rangle$}](v01) at (1,19) {};
	\pvx[](v03) at (0,18) {};
	\pvx[label=left:{\tiny $\langle a \rangle$}](v02) at (2,18) {};
	\pvx[](v04) at (1,17) {};
	\pvx[](v05) at (-1,17) {};
	\pvx[](v06) at (0,16) {};
	
	\node[] (d2)  at (-0.95,15.8)  {$\Ddots$};
	\node[] (d2)  at (-0.8,20.1)  {$\udots$};
	
	\pvx[](v07) at (-2,16) {};
	\pvx[label=right:{\tiny $\langle n\rangle $}](v08) at (-3,15) {};
	\pvx[label=right:{\tiny $\langle n+1\rangle$}](v09) at (-4,14) {};
	\pvx[](v10) at (-1,15) {};
 	\pvx[label=left:{\tiny $\langle a,n\rangle $}](v11) at (-2,14) {};
    
    \pvx[label=left:{\tiny $\langle a,n+1\rangle$}](v14) at (-3,13) {};	
	\pvx[](v15) at (0,20) {};	
	\pvx[](v16) at (-1,19) {};	
	
	\pvx[label=left:{\tiny $\langle 4\rangle$}](zerohat) at (-4,12) {};
	\draw[line width=1pt] (v14)--(zerohat);
	\draw[line width=2pt, color=red] (-4,11.9)--(-2.8,13.1);
	\draw[line width=2pt, color=blue] (-4.1,12)--(-3.1,13);	
	
	\draw[line width=1pt] (v01)--(v02);
	\draw[line width=1pt] (v03)--(v04);		
	\draw[line width=1pt] (v05)--(v06);
	\draw[line width=1pt] (v01)--(v03);
	\draw[line width=1pt] (v03)--(v05);
	\draw[line width=1pt] (v02)--(v04);
	\draw[line width=1pt] (v04)--(v06);
	
	\draw[line width=1pt] (v07)--(v08);
	\draw[line width=1pt] (v08)--(v09)--(v14);
	\draw[line width=1pt] (v10)--(v11);
	\draw[line width=1pt] (v07)--(v10);
	\draw[line width=1pt] (v08)--(v11);
	
	\draw[line width=1pt] (v14)--(v11);
	\draw[line width=1pt] (v15)--(v16);
	\draw[line width=1pt] (v15)--(v01);
	\draw[line width=1pt] (v16)--(v03);

	\draw[line width=2pt, color=red] (-2.9,13.1)--(-3.8,14)--(-1.9,15.9);
	\draw[line width=2pt, color=blue] (-3,13)--(-4,14)--(-2,16);

	\draw[line width=2pt, color=red] (-0.9,16.9)--(0.2,18)--(-0.9,19.1);
	\draw[line width=2pt, color=blue] (-1,17)--(0,18)--(-1,19);

\end{scope}	

\begin{scope}[xshift=300, yshift=0, yscale=1,xscale=-1] 
	
	\pvx[label=left:{\tiny $\langle a-1\rangle$}](v01) at (1,19) {};
	\pvx[](v03) at (0,18) {};
	\pvx[label=left:{\tiny $\langle a \rangle$}](v02) at (2,18) {};
	\pvx[](v04) at (1,17) {};
	\pvx[](v05) at (-1,17) {};
	\pvx[](v06) at (0,16) {};
	
	\node[] (d2)  at (-0.95,15.8)  {$\Ddots$};
	\node[] (d2)  at (-0.8,20.1)  {$\udots$};
	
	\pvx[](v07) at (-2,16) {};
	\pvx[label=right:{\tiny $\langle n\rangle $}](v08) at (-3,15) {};
	\pvx[label=right:{\tiny $\langle n+1\rangle$}](v09) at (-4,14) {};
	\pvx[](v10) at (-1,15) {};
 	\pvx[label=left:{\tiny $\langle a,n\rangle $}](v11) at (-2,14) {};
    
    \pvx[label=left:{\tiny $\langle a,n+1\rangle$}](v14) at (-3,13) {};	
	\pvx[](v15) at (0,20) {};	
	\pvx[](v16) at (-1,19) {};

	\draw[line width=1pt] (v01)--(v02);
	\draw[line width=1pt] (v03)--(v04);		
	\draw[line width=1pt] (v05)--(v06);
	\draw[line width=1pt] (v01)--(v03);
	\draw[line width=1pt] (v03)--(v05);
	\draw[line width=1pt] (v02)--(v04);
	\draw[line width=1pt] (v04)--(v06);
	
	\draw[line width=1pt] (v07)--(v08);
	\draw[line width=1pt] (v08)--(v09)--(v14);
	\draw[line width=1pt] (v10)--(v11);
	\draw[line width=1pt] (v07)--(v10);
	\draw[line width=1pt] (v08)--(v11);
	
	\draw[line width=1pt] (v14)--(v11);
	\draw[line width=1pt] (v15)--(v16);
	\draw[line width=1pt] (v15)--(v01);
	\draw[line width=1pt] (v16)--(v03);

	\pvx[label=left:{\tiny $\langle 4\rangle$}](zerohat) at (-4,12) {};
	\draw[line width=1pt] (v14)--(zerohat);
	\draw[line width=2pt, color=red] (-4,11.9)--(-2.8,13.1) -- (-0.8,15);
	\draw[line width=2pt, color=blue] (-4.1,12)--(-3.1,13)-- (-0.9,15.1);

	\draw[line width=2pt, color=red] (-0.9,16.9)--(0.2,18)--(-0.9,19.1);
	\draw[line width=2pt, color=blue] (-1,17)--(0,18)--(-1,19);

\end{scope}	
\begin{scope}[xshift=600, yshift=0, yscale=1,xscale=-1] 
	\pvx[label=left:{\tiny $\langle a-1\rangle$}](v01) at (1,19) {};
	\pvx[](v03) at (0,18) {};
	\pvx[label=left:{\tiny $\langle a \rangle$}](v02) at (2,18) {};
	\pvx[](v04) at (1,17) {};
	\pvx[](v05) at (-1,17) {};
	\pvx[](v06) at (0,16) {};
	
	\node[] (d2)  at (-0.95,15.8)  {$\Ddots$};
	\node[] (d2)  at (-0.8,20.1)  {$\udots$};
	
	\pvx[](v07) at (-2,16) {};
	\pvx[label=right:{\tiny $\langle n\rangle $}](v08) at (-3,15) {};
	\pvx[label=right:{\tiny $\langle n+1\rangle$}](v09) at (-4,14) {};
	\pvx[](v10) at (-1,15) {};
 	\pvx[label=left:{\tiny $\langle a,n\rangle $}](v11) at (-2,14) {};
    
    \pvx[label=left:{\tiny $\langle a,n+1\rangle$}](v14) at (-3,13) {};	
	\pvx[](v15) at (0,20) {};	
	\pvx[](v16) at (-1,19) {};	
	
	\pvx[label=left:{\tiny $\langle 4\rangle$}](zerohat) at (-4,12) {};
	\draw[line width=1pt] (v14)--(zerohat);
	\draw[line width=2pt, color=red] (-4,11.9)--(-2.8,13.1);
	\draw[line width=2pt, color=blue] (-4.1,12)--(-3.1,13);
	
	\draw[line width=1pt] (v01)--(v02);
	\draw[line width=1pt] (v03)--(v04);		
	\draw[line width=1pt] (v05)--(v06);
	\draw[line width=1pt] (v01)--(v03);
	\draw[line width=1pt] (v03)--(v05);
	\draw[line width=1pt] (v02)--(v04);
	\draw[line width=1pt] (v04)--(v06);
	
	\draw[line width=1pt] (v07)--(v08);
	\draw[line width=1pt] (v08)--(v09)--(v14);
	\draw[line width=1pt] (v10)--(v11);
	\draw[line width=1pt] (v07)--(v10);
	\draw[line width=1pt] (v08)--(v11);
	
	\draw[line width=1pt] (v14)--(v11);
	\draw[line width=1pt] (v15)--(v16);
	\draw[line width=1pt] (v15)--(v01);
	\draw[line width=1pt] (v16)--(v03);

    \draw[line width=2pt, color=red] (-2.8,13.1)--(-1.95,13.95)--(-3,15);
    \draw[line width=2pt, color=red]
    (-2.9,14.9)--(-1.9,15.9);
	\draw[line width=2pt, color=blue] (-3,13)--(-4,14)--(-2,16);

	\draw[line width=2pt, color=red] (-0.9,16.9)--(0.2,18)--(-0.9,19.1);
	\draw[line width=2pt, color=blue] (-1,17)--(0,18)--(-1,19);

\end{scope}	

\end{tikzpicture}
\end{center}
\caption{The three subcases (a), (b), and (c) of an extension, with no twist on $\mL^t$. Here $\sigma_1$ is in red, and $\sigma_2$ is in blue.}
\label{fig.twisted}
\end{figure}

\begin{center}
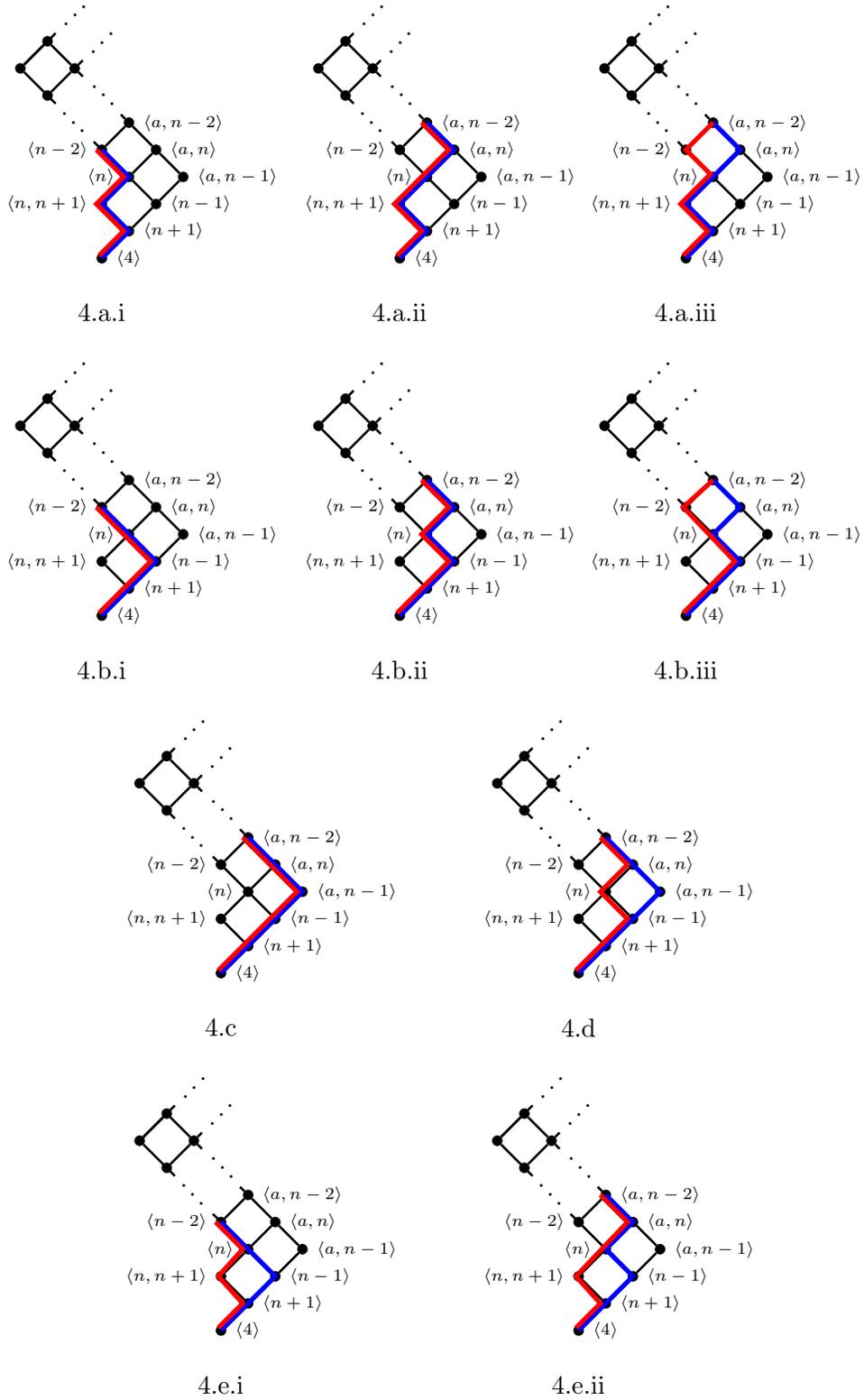
\begin{figure}
\begin{tikzpicture}
\begin{scope}[xshift=0, yshift=0, scale=0.4] 
	\pvx[label=left:{}](a-1) at (0,10) {};	
	\pvx[label=left:{}](a) at (-1,9) {};
	\pvx[label=right:{}](a1) at (1,9) {};
	\pvx[label=right:{\tiny $\langle a,n-2\rangle$}](k-1) at (3,7) {};
	\pvx[label=right:{\tiny $\langle a,n \rangle$}](k) at (4,6) {};
	\pvx[label=right:{\tiny $\langle a,n-1\rangle$}](k1) at (5,5) {};
	\pvx[label=right:{\tiny $\langle n-1 \rangle$}](3) at (4,4) {};
	\pvx[label=left:{}](ab) at (0,8) {};
	\pvx[label=left:{\tiny $\langle n-2 \rangle$}](ak-1) at (2,6) {};
	\pvx[label=left:{\tiny $\langle n \rangle$}](ak) at (3,5) {};	
	\pvx[label=right:{\tiny $\langle n+1 \rangle$}](2hat) at (3,3) {};
	\pvx[label=right:{\tiny $\langle 4 \rangle$}](1hat) at (2,2) {};	
	\pvx[label=left:{\tiny $\langle n,n+1 \rangle$}](dfs) at (2,4) {};

 	\draw[line width=1pt] (a1)--(a-1)--(a) -- (0.3,7.7); 
 	\draw[line width=1pt] (1.7,6.3)--(3) -- (k1) -- (k) -- (k-1);

 	\draw[line width=1pt] (0.3,10.3)--(a);
  	\draw[line width=1pt] (1.3,9.3)--(a1);	
  	\draw[line width=1pt] (1.3,8.7)--(a1);	
  	\draw[line width=1pt] (2.7,7.3)--(k-1);		
  	\draw[line width=1pt] (ab)--(a1);		 	
  	\draw[line width=1pt] (ak-1)--(k-1);		 	  	
  	\draw[line width=1pt] (ak)--(k);
  	\draw[line width=1pt] (3)--(2hat);
  	\draw[line width=1pt] (1hat)--(2hat);
  	\draw[line width=1pt] (ak)--(dfs)--(2hat);
  	
 	\node[](dd) at (0.9,6.75) {$\Ddots$};	
	\node[](dd) at (2,7.75) {$\Ddots$};
	\node[](uu1) at (1.75,10) {$\udots$};	
	\node[](uu1) at (0.75,11) {$\udots$};		

	\node[](qq) at (2,0) {4.a.i};
	
	\draw[line width=2pt, color=blue] (2,2)--(3,3)--(2,4)--(3,5)--(2,6);
	\draw[line width=2pt, color=red] (1.9,2.1)--(2.8,3)--(1.8,4)--(2.8,5)--(1.8,6);
	
\end{scope}		
\begin{scope}[xshift=125, yshift=0, scale=0.4] 
	\pvx[label=left:{}](a-1) at (0,10) {};	
	\pvx[label=left:{}](a) at (-1,9) {};
	\pvx[label=right:{}](a1) at (1,9) {};
	\pvx[label=right:{\tiny $\langle a,n-2\rangle$}](k-1) at (3,7) {};
	\pvx[label=right:{\tiny $\langle a,n \rangle$}](k) at (4,6) {};
	\pvx[label=right:{\tiny $\langle a,n-1\rangle$}](k1) at (5,5) {};
	\pvx[label=right:{\tiny $\langle n-1 \rangle$}](3) at (4,4) {};
	\pvx[label=left:{}](ab) at (0,8) {};
	\pvx[label=left:{\tiny $\langle n-2 \rangle$}](ak-1) at (2,6) {};
	\pvx[label=left:{\tiny $\langle n \rangle$}](ak) at (3,5) {};	
	\pvx[label=right:{\tiny $\langle n+1 \rangle$}](2hat) at (3,3) {};
	\pvx[label=right:{\tiny $\langle 4 \rangle$}](1hat) at (2,2) {};	
	\pvx[label=left:{\tiny $\langle n,n+1 \rangle$}](dfs) at (2,4) {};

 	\draw[line width=1pt] (a1)--(a-1)--(a) -- (0.3,7.7); 
 	\draw[line width=1pt] (1.7,6.3)--(3) -- (k1) -- (k) -- (k-1);

 	\draw[line width=1pt] (0.3,10.3)--(a);
  	\draw[line width=1pt] (1.3,9.3)--(a1);	
  	\draw[line width=1pt] (1.3,8.7)--(a1);	
  	\draw[line width=1pt] (2.7,7.3)--(k-1);		
  	\draw[line width=1pt] (ab)--(a1);		 	
  	\draw[line width=1pt] (ak-1)--(k-1);		 	  	
  	\draw[line width=1pt] (ak)--(k);
  	\draw[line width=1pt] (3)--(2hat);
  	\draw[line width=1pt] (1hat)--(2hat);
  	\draw[line width=1pt] (ak)--(dfs)--(2hat);
  	
 	\node[](dd) at (0.9,6.75) {$\Ddots$};	
	\node[](dd) at (2,7.75) {$\Ddots$};
	\node[](uu1) at (1.75,10) {$\udots$};	
	\node[](uu1) at (0.75,11) {$\udots$};		

	\draw[line width=2pt, color=blue] (2,2)--(3,3)--(2,4)--(4,6)--(3,7);
	\draw[line width=2pt, color=red] (1.9,2.1)--(2.8,3)--(1.8,4)--(3.8,6)--(2.8,7);

	\node[](qq) at (2,0) {4.a.ii};

\end{scope}	
\begin{scope}[xshift=245, yshift=0, scale=0.4]  	
	\pvx[label=left:{}](a-1) at (0,10) {};	
	\pvx[label=left:{}](a) at (-1,9) {};
	\pvx[label=right:{}](a1) at (1,9) {};
	\pvx[label=right:{\tiny $\langle a,n-2\rangle$}](k-1) at (3,7) {};
	\pvx[label=right:{\tiny $\langle a,n \rangle$}](k) at (4,6) {};
	\pvx[label=right:{\tiny $\langle a,n-1\rangle$}](k1) at (5,5) {};
	\pvx[label=right:{\tiny $\langle n-1 \rangle$}](3) at (4,4) {};
	\pvx[label=left:{}](ab) at (0,8) {};
	\pvx[label=left:{\tiny $\langle n-2 \rangle$}](ak-1) at (2,6) {};
	\pvx[label=left:{\tiny $\langle n \rangle$}](ak) at (3,5) {};	
	\pvx[label=right:{\tiny $\langle n+1 \rangle$}](2hat) at (3,3) {};
	\pvx[label=right:{\tiny $\langle 4 \rangle$}](1hat) at (2,2) {};	
	\pvx[label=left:{\tiny $\langle n,n+1 \rangle$}](dfs) at (2,4) {};

 	\draw[line width=1pt] (a1)--(a-1)--(a) -- (0.3,7.7); 
 	\draw[line width=1pt] (1.7,6.3)--(3) -- (k1) -- (k) -- (k-1);

 	\draw[line width=1pt] (0.3,10.3)--(a);
  	\draw[line width=1pt] (1.3,9.3)--(a1);	
  	\draw[line width=1pt] (1.3,8.7)--(a1);	
  	\draw[line width=1pt] (2.7,7.3)--(k-1);		
  	\draw[line width=1pt] (ab)--(a1);		 	
  	\draw[line width=1pt] (ak-1)--(k-1);		 	  	
  	\draw[line width=1pt] (ak)--(k);
  	\draw[line width=1pt] (3)--(2hat);
  	\draw[line width=1pt] (1hat)--(2hat);
  	\draw[line width=1pt] (ak)--(dfs)--(2hat);
  	
 	\node[](dd) at (0.9,6.75) {$\Ddots$};	
	\node[](dd) at (2,7.75) {$\Ddots$};
	\node[](uu1) at (1.75,10) {$\udots$};	
	\node[](uu1) at (0.75,11) {$\udots$};		

	\draw[line width=2pt, color=blue] (2,2)--(3,3)--(2,4)--(4,6)--(3,7);
	\draw[line width=2pt, color=red] (1.9,2.1)--(2.8,3)--(1.8,4)--(2.9,5.1)--(2,6)--(3,7);

	\node[](qq) at (2,0) {4.a.iii};

\end{scope}		
\begin{scope}[xshift=0, yshift=-150, scale=0.4] 
	\pvx[label=left:{}](a-1) at (0,10) {};	
	\pvx[label=left:{}](a) at (-1,9) {};
	\pvx[label=right:{}](a1) at (1,9) {};
	\pvx[label=right:{\tiny $\langle a,n-2\rangle$}](k-1) at (3,7) {};
	\pvx[label=right:{\tiny $\langle a,n \rangle$}](k) at (4,6) {};
	\pvx[label=right:{\tiny $\langle a,n-1\rangle$}](k1) at (5,5) {};
	\pvx[label=right:{\tiny $\langle n-1 \rangle$}](3) at (4,4) {};
	\pvx[label=left:{}](ab) at (0,8) {};
	\pvx[label=left:{\tiny $\langle n-2 \rangle$}](ak-1) at (2,6) {};
	\pvx[label=left:{\tiny $\langle n \rangle$}](ak) at (3,5) {};	
	\pvx[label=right:{\tiny $\langle n+1 \rangle$}](2hat) at (3,3) {};
	\pvx[label=right:{\tiny $\langle 4 \rangle$}](1hat) at (2,2) {};	
	\pvx[label=left:{\tiny $\langle n,n+1 \rangle$}](dfs) at (2,4) {};

 	\draw[line width=1pt] (a1)--(a-1)--(a) -- (0.3,7.7); 
 	\draw[line width=1pt] (1.7,6.3)--(3) -- (k1) -- (k) -- (k-1);

 	\draw[line width=1pt] (0.3,10.3)--(a);
  	\draw[line width=1pt] (1.3,9.3)--(a1);	
  	\draw[line width=1pt] (1.3,8.7)--(a1);	
  	\draw[line width=1pt] (2.7,7.3)--(k-1);		
  	\draw[line width=1pt] (ab)--(a1);		 	
  	\draw[line width=1pt] (ak-1)--(k-1);		 	  	
  	\draw[line width=1pt] (ak)--(k);
  	\draw[line width=1pt] (3)--(2hat);
  	\draw[line width=1pt] (1hat)--(2hat);
  	\draw[line width=1pt] (ak)--(dfs)--(2hat);
  	
 	\node[](dd) at (0.9,6.75) {$\Ddots$};	
	\node[](dd) at (2,7.75) {$\Ddots$};
	\node[](uu1) at (1.75,10) {$\udots$};	
	\node[](uu1) at (0.75,11) {$\udots$};		

	\node[](qq) at (2,0) {4.b.i};
	
	\draw[line width=2pt, color=blue] (2,2)--(4,4)--(2,6);
	\draw[line width=2pt, color=red] (1.9,2.1)--(3.8,4)--(1.8,6);
	
\end{scope}		
\begin{scope}[xshift=125, yshift=-150, scale=0.4] 

	\pvx[label=left:{}](a-1) at (0,10) {};	
	\pvx[label=left:{}](a) at (-1,9) {};
	\pvx[label=right:{}](a1) at (1,9) {};
	\pvx[label=right:{\tiny $\langle a,n-2\rangle$}](k-1) at (3,7) {};
	\pvx[label=right:{\tiny $\langle a,n \rangle$}](k) at (4,6) {};
	\pvx[label=right:{\tiny $\langle a,n-1\rangle$}](k1) at (5,5) {};
	\pvx[label=right:{\tiny $\langle n-1 \rangle$}](3) at (4,4) {};
	\pvx[label=left:{}](ab) at (0,8) {};
	\pvx[label=left:{\tiny $\langle n-2 \rangle$}](ak-1) at (2,6) {};
	\pvx[label=left:{\tiny $\langle n \rangle$}](ak) at (3,5) {};	
	\pvx[label=right:{\tiny $\langle n+1 \rangle$}](2hat) at (3,3) {};
	\pvx[label=right:{\tiny $\langle 4 \rangle$}](1hat) at (2,2) {};	
	\pvx[label=left:{\tiny $\langle n,n+1 \rangle$}](dfs) at (2,4) {};

 	\draw[line width=1pt] (a1)--(a-1)--(a) -- (0.3,7.7); 
 	\draw[line width=1pt] (1.7,6.3)--(3) -- (k1) -- (k) -- (k-1);

 	\draw[line width=1pt] (0.3,10.3)--(a);
  	\draw[line width=1pt] (1.3,9.3)--(a1);	
  	\draw[line width=1pt] (1.3,8.7)--(a1);	
  	\draw[line width=1pt] (2.7,7.3)--(k-1);		
  	\draw[line width=1pt] (ab)--(a1);		 	
  	\draw[line width=1pt] (ak-1)--(k-1);		 	  	
  	\draw[line width=1pt] (ak)--(k);
  	\draw[line width=1pt] (3)--(2hat);
  	\draw[line width=1pt] (1hat)--(2hat);
  	\draw[line width=1pt] (ak)--(dfs)--(2hat);
  	
 	\node[](dd) at (0.9,6.75) {$\Ddots$};	
	\node[](dd) at (2,7.75) {$\Ddots$};
	\node[](uu1) at (1.75,10) {$\udots$};	
	\node[](uu1) at (0.75,11) {$\udots$};		

	\draw[line width=2pt, color=blue] (2,2)--(4,4)--(3,5)--(4,6)--(3,7);
	\draw[line width=2pt, color=red] (1.9,2.1)--(3.8,4)--(2.8,5)--(3.8,6)--(2.8,7);

	\node[](qq) at (2,0) {4.b.ii};

\end{scope}		
\begin{scope}[xshift=245, yshift=-150, scale=0.4] 
	\pvx[label=left:{}](a-1) at (0,10) {};	
	\pvx[label=left:{}](a) at (-1,9) {};
	\pvx[label=right:{}](a1) at (1,9) {};
	\pvx[label=right:{\tiny $\langle a,n-2\rangle$}](k-1) at (3,7) {};
	\pvx[label=right:{\tiny $\langle a,n \rangle$}](k) at (4,6) {};
	\pvx[label=right:{\tiny $\langle a,n-1\rangle$}](k1) at (5,5) {};
	\pvx[label=right:{\tiny $\langle n-1 \rangle$}](3) at (4,4) {};
	\pvx[label=left:{}](ab) at (0,8) {};
	\pvx[label=left:{\tiny $\langle n-2 \rangle$}](ak-1) at (2,6) {};
	\pvx[label=left:{\tiny $\langle n \rangle$}](ak) at (3,5) {};	
	\pvx[label=right:{\tiny $\langle n+1 \rangle$}](2hat) at (3,3) {};
	\pvx[label=right:{\tiny $\langle 4 \rangle$}](1hat) at (2,2) {};	
	\pvx[label=left:{\tiny $\langle n,n+1 \rangle$}](dfs) at (2,4) {};

 	\draw[line width=1pt] (a1)--(a-1)--(a) -- (0.3,7.7); 
 	\draw[line width=1pt] (1.7,6.3)--(3) -- (k1) -- (k) -- (k-1);

 	\draw[line width=1pt] (0.3,10.3)--(a);
  	\draw[line width=1pt] (1.3,9.3)--(a1);	
  	\draw[line width=1pt] (1.3,8.7)--(a1);	
  	\draw[line width=1pt] (2.7,7.3)--(k-1);		
  	\draw[line width=1pt] (ab)--(a1);		 	
  	\draw[line width=1pt] (ak-1)--(k-1);		 	  	
  	\draw[line width=1pt] (ak)--(k);
  	\draw[line width=1pt] (3)--(2hat);
  	\draw[line width=1pt] (1hat)--(2hat);
  	\draw[line width=1pt] (ak)--(dfs)--(2hat);
  	
 	\node[](dd) at (0.9,6.75) {$\Ddots$};	
	\node[](dd) at (2,7.75) {$\Ddots$};
	\node[](uu1) at (1.75,10) {$\udots$};	
	\node[](uu1) at (0.75,11) {$\udots$};		

	\draw[line width=2pt, color=blue] (2,2)--(4,4)--(3,5)--(4,6)--(3,7);
	\draw[line width=2pt, color=red] (1.9,2.1)--(3.8,4)--(1.9,6)--(3,7);

	\node[](qq) at (2,0) {4.b.iii};

\end{scope}	

\begin{scope}[xshift=50, yshift=-300, scale=0.4] 
	\pvx[label=left:{}](a-1) at (0,10) {};	
	\pvx[label=left:{}](a) at (-1,9) {};
	\pvx[label=right:{}](a1) at (1,9) {};
	\pvx[label=right:{\tiny $\langle a,n-2\rangle$}](k-1) at (3,7) {};
	\pvx[label=right:{\tiny $\langle a,n \rangle$}](k) at (4,6) {};
	\pvx[label=right:{\tiny $\langle a,n-1\rangle$}](k1) at (5,5) {};
	\pvx[label=right:{\tiny $\langle n-1 \rangle$}](3) at (4,4) {};
	\pvx[label=left:{}](ab) at (0,8) {};
	\pvx[label=left:{\tiny $\langle n-2 \rangle$}](ak-1) at (2,6) {};
	\pvx[label=left:{\tiny $\langle n \rangle$}](ak) at (3,5) {};	
	\pvx[label=right:{\tiny $\langle n+1 \rangle$}](2hat) at (3,3) {};
	\pvx[label=right:{\tiny $\langle 4 \rangle$}](1hat) at (2,2) {};	
	\pvx[label=left:{\tiny $\langle n,n+1 \rangle$}](dfs) at (2,4) {};

 	\draw[line width=1pt] (a1)--(a-1)--(a) -- (0.3,7.7); 
 	\draw[line width=1pt] (1.7,6.3)--(3) -- (k1) -- (k) -- (k-1);

 	\draw[line width=1pt] (0.3,10.3)--(a);
  	\draw[line width=1pt] (1.3,9.3)--(a1);	
  	\draw[line width=1pt] (1.3,8.7)--(a1);	
  	\draw[line width=1pt] (2.7,7.3)--(k-1);		
  	\draw[line width=1pt] (ab)--(a1);		 	
  	\draw[line width=1pt] (ak-1)--(k-1);		 	  	
  	\draw[line width=1pt] (ak)--(k);
  	\draw[line width=1pt] (3)--(2hat);
  	\draw[line width=1pt] (1hat)--(2hat);
  	\draw[line width=1pt] (ak)--(dfs)--(2hat);
  	
 	\node[](dd) at (0.9,6.75) {$\Ddots$};	
	\node[](dd) at (2,7.75) {$\Ddots$};
	\node[](uu1) at (1.75,10) {$\udots$};	
	\node[](uu1) at (0.75,11) {$\udots$};		

	\node[](qq) at (2,0) {4.c};
	
	\draw[line width=2pt, color=blue] (2,2)--(5,5)--(3,7);
	\draw[line width=2pt, color=red] (1.9,2.1)--(4.8,5)--(2.8,7);
	
\end{scope}		
\begin{scope}[xshift=200, yshift=-300, scale=0.4] 

	\pvx[label=left:{}](a-1) at (0,10) {};	
	\pvx[label=left:{}](a) at (-1,9) {};
	\pvx[label=right:{}](a1) at (1,9) {};
	\pvx[label=right:{\tiny $\langle a,n-2\rangle$}](k-1) at (3,7) {};
	\pvx[label=right:{\tiny $\langle a,n \rangle$}](k) at (4,6) {};
	\pvx[label=right:{\tiny $\langle a,n-1\rangle$}](k1) at (5,5) {};
	\pvx[label=right:{\tiny $\langle n-1 \rangle$}](3) at (4,4) {};
	\pvx[label=left:{}](ab) at (0,8) {};
	\pvx[label=left:{\tiny $\langle n-2 \rangle$}](ak-1) at (2,6) {};
	\pvx[label=left:{\tiny $\langle n \rangle$}](ak) at (3,5) {};	
	\pvx[label=right:{\tiny $\langle n+1 \rangle$}](2hat) at (3,3) {};
	\pvx[label=right:{\tiny $\langle 4 \rangle$}](1hat) at (2,2) {};	
	\pvx[label=left:{\tiny $\langle n,n+1 \rangle$}](dfs) at (2,4) {};

 	\draw[line width=1pt] (a1)--(a-1)--(a) -- (0.3,7.7); 
 	\draw[line width=1pt] (1.7,6.3)--(3) -- (k1) -- (k) -- (k-1);

 	\draw[line width=1pt] (0.3,10.3)--(a);
  	\draw[line width=1pt] (1.3,9.3)--(a1);	
  	\draw[line width=1pt] (1.3,8.7)--(a1);	
  	\draw[line width=1pt] (2.7,7.3)--(k-1);		
  	\draw[line width=1pt] (ab)--(a1);		 	
  	\draw[line width=1pt] (ak-1)--(k-1);		 	  	
  	\draw[line width=1pt] (ak)--(k);
  	\draw[line width=1pt] (3)--(2hat);
  	\draw[line width=1pt] (1hat)--(2hat);
  	\draw[line width=1pt] (ak)--(dfs)--(2hat);
  	
 	\node[](dd) at (0.9,6.75) {$\Ddots$};	
	\node[](dd) at (2,7.75) {$\Ddots$};
	\node[](uu1) at (1.75,10) {$\udots$};	
	\node[](uu1) at (0.75,11) {$\udots$};		

	\draw[line width=2pt, color=blue] (2,2)--(5,5)--(3,7);
	\draw[line width=2pt, color=red] (1.9,2.1)--(3.8,4)--(2.8,5)--(3.8,6)--(2.8,7);

	\node[](qq) at (2,0) {4.d};

\end{scope}		
\begin{scope}[xshift=50, yshift=-450, scale=0.4] 
	\pvx[label=left:{}](a-1) at (0,10) {};	
	\pvx[label=left:{}](a) at (-1,9) {};
	\pvx[label=right:{}](a1) at (1,9) {};
	\pvx[label=right:{\tiny $\langle a,n-2\rangle$}](k-1) at (3,7) {};
	\pvx[label=right:{\tiny $\langle a,n \rangle$}](k) at (4,6) {};
	\pvx[label=right:{\tiny $\langle a,n-1\rangle$}](k1) at (5,5) {};
	\pvx[label=right:{\tiny $\langle n-1 \rangle$}](3) at (4,4) {};
	\pvx[label=left:{}](ab) at (0,8) {};
	\pvx[label=left:{\tiny $\langle n-2 \rangle$}](ak-1) at (2,6) {};
	\pvx[label=left:{\tiny $\langle n \rangle$}](ak) at (3,5) {};	
	\pvx[label=right:{\tiny $\langle n+1 \rangle$}](2hat) at (3,3) {};
	\pvx[label=right:{\tiny $\langle 4 \rangle$}](1hat) at (2,2) {};	
	\pvx[label=left:{\tiny $\langle n,n+1 \rangle$}](dfs) at (2,4) {};

 	\draw[line width=1pt] (a1)--(a-1)--(a) -- (0.3,7.7); 
 	\draw[line width=1pt] (1.7,6.3)--(3) -- (k1) -- (k) -- (k-1);

 	\draw[line width=1pt] (0.3,10.3)--(a);
  	\draw[line width=1pt] (1.3,9.3)--(a1);	
  	\draw[line width=1pt] (1.3,8.7)--(a1);	
  	\draw[line width=1pt] (2.7,7.3)--(k-1);		
  	\draw[line width=1pt] (ab)--(a1);		 	
  	\draw[line width=1pt] (ak-1)--(k-1);		 	  	
  	\draw[line width=1pt] (ak)--(k);
  	\draw[line width=1pt] (3)--(2hat);
  	\draw[line width=1pt] (1hat)--(2hat);
  	\draw[line width=1pt] (ak)--(dfs)--(2hat);
  	
 	\node[](dd) at (0.9,6.75) {$\Ddots$};	
	\node[](dd) at (2,7.75) {$\Ddots$};
	\node[](uu1) at (1.75,10) {$\udots$};	
	\node[](uu1) at (0.75,11) {$\udots$};		

	\draw[line width=2pt, color=blue] (2,2)--(4,4)--(2,6);
	\draw[line width=2pt, color=red] (1.9,2.1)--(2.8,3)--(1.9,4) -- (2.8,5) -- (1.8,6);

	\node[](qq) at (2,0) {4.e.i};

\end{scope}	

\begin{scope}[xshift=200, yshift=-450, scale=0.4] 
	\pvx[label=left:{}](a-1) at (0,10) {};	
	\pvx[label=left:{}](a) at (-1,9) {};
	\pvx[label=right:{}](a1) at (1,9) {};
	\pvx[label=right:{\tiny $\langle a,n-2\rangle$}](k-1) at (3,7) {};
	\pvx[label=right:{\tiny $\langle a,n \rangle$}](k) at (4,6) {};
	\pvx[label=right:{\tiny $\langle a,n-1\rangle$}](k1) at (5,5) {};
	\pvx[label=right:{\tiny $\langle n-1 \rangle$}](3) at (4,4) {};
	\pvx[label=left:{}](ab) at (0,8) {};
	\pvx[label=left:{\tiny $\langle n-2 \rangle$}](ak-1) at (2,6) {};
	\pvx[label=left:{\tiny $\langle n \rangle$}](ak) at (3,5) {};	
	\pvx[label=right:{\tiny $\langle n+1 \rangle$}](2hat) at (3,3) {};
	\pvx[label=right:{\tiny $\langle 4 \rangle$}](1hat) at (2,2) {};	
	\pvx[label=left:{\tiny $\langle n,n+1 \rangle$}](dfs) at (2,4) {};

 	\draw[line width=1pt] (a1)--(a-1)--(a) -- (0.3,7.7); 
 	\draw[line width=1pt] (1.7,6.3)--(3) -- (k1) -- (k) -- (k-1);

 	\draw[line width=1pt] (0.3,10.3)--(a);
  	\draw[line width=1pt] (1.3,9.3)--(a1);	
  	\draw[line width=1pt] (1.3,8.7)--(a1);	
  	\draw[line width=1pt] (2.7,7.3)--(k-1);		
  	\draw[line width=1pt] (ab)--(a1);		 	
  	\draw[line width=1pt] (ak-1)--(k-1);		 	  	
  	\draw[line width=1pt] (ak)--(k);
  	\draw[line width=1pt] (3)--(2hat);
  	\draw[line width=1pt] (1hat)--(2hat);
  	\draw[line width=1pt] (ak)--(dfs)--(2hat);
  	
 	\node[](dd) at (0.9,6.75) {$\Ddots$};	
	\node[](dd) at (2,7.75) {$\Ddots$};
	\node[](uu1) at (1.75,10) {$\udots$};	
	\node[](uu1) at (0.75,11) {$\udots$};		

	\draw[line width=2pt, color=blue] (2,2)--(4,4)--(3,5)--(4,6)--(3,7);
	\draw[line width=2pt, color=red] (1.9,2.1)--(2.8,3)--(1.9,4) -- (3.8,6) -- (2.8,7);

	\node[](qq) at (2,0) {4.e.ii};

\end{scope}	

\end{tikzpicture}
\caption{The subcases in a turn with $\mL^t$ and $\mL^{t+1}$ twisted by $\tau$. Here $\tau(\sigma_1)$ is in red and $\tau(\sigma_2)$ is in blue.}
\label{fig:TurnCasesAppendix}
\end{figure}
\end{center}